\documentclass[a4paper,12pt]{article}
\usepackage[text={6.5in,9.5in},centering]{geometry}
\usepackage{graphicx,url,subfig}
\RequirePackage[OT1]{fontenc}
\RequirePackage{amsthm,amsmath,amssymb,amscd}
\RequirePackage[numbers]{natbib}
\RequirePackage[colorlinks,citecolor=blue,urlcolor=blue]{hyperref}
\usepackage{enumerate}
\usepackage{datetime}
\usepackage{etex}
\usepackage[normalem]{ulem} 
\usepackage{soul} 
\makeatother
\numberwithin{equation}{section}
\allowdisplaybreaks

\theoremstyle{plain}
\newtheorem{theorem}{Theorem}[section]

\newtheorem{lemma}[theorem]{Lemma}
\newtheorem{proposition}[theorem]{Proposition}

\theoremstyle{definition}
\newtheorem{definition}[theorem]{Definition}

\newtheorem{assumption}[theorem]{Assumption}

\newtheorem{remark}[theorem]{Remark}

\newtheorem{example}[theorem]{Example}

\newcommand{\E}{\mathbb{E}}
\newcommand{\W}{\dot{W}}

\newcommand{\D}{\mathbb{D}}
\newcommand{\ud}{\ensuremath{\mathrm{d}}}

\newcommand{\Norm}[1]{\left\|  #1   \right\|}

\newcommand{\InPrd}[1]{\left\langle #1 \right\rangle}

\newcommand{\calB}{\mathcal{B}}

\newcommand{\calF}{\mathcal{F}}

\newcommand{\calH}{\mathcal{H}}
\newcommand{\calI}{\mathcal{I}}

\newcommand{\calN}{\mathcal{N}}

\newcommand{\calT}{\mathcal{T}}

\newcommand{\calS}{\mathcal{S}}
\newcommand{\bbC}{\mathbb{C}}
\newcommand{\bbN}{\mathbb{N}}
\newcommand{\bbP}{\mathbb{P}}
\newcommand{\bbZ}{\mathbb{Z}}

\newcommand*{\one}{{{\rm 1\mkern-1.5mu}\!{\rm I}}}
\newcommand{\RR}{\mathbb{R}}
\newcommand{\R}{\mathbb{R}}
\newcommand{\myEnd}{\hfill$\square$}

\DeclareMathOperator{\LIP}{Lip}
\DeclareMathOperator{\lip}{\mathit{l}}

\DeclareMathOperator{\Vip}{\overline{\varsigma}}

\title{
Regularity and strict positivity of densities for the stochastic heat equation 
on $\RR^d$
     }
\author{
{\bf Le Chen}
\; and \; {\bf Jingyu Huang}
\date{\vspace{0em}\small \today}
}

\begin{document}
\maketitle
\begin{abstract}

In this paper, we study the stochastic heat equation with a general 
multiplicative Gaussian noise that is white in time and colored in space. Both 
regularity and strict positivity of the densities of the solution have been 
established. The difficulty, and hence the contribution, of the paper lie in 
three aspects, which include rough initial conditions, degenerate diffusion 
coefficient, and weakest possible assumptions on the correlation function of the 
noise. In particular, our results cover the parabolic Anderson model starting 
from a Dirac delta initial measure.\\

\end{abstract} 

{\hypersetup{hidelinks}
\tableofcontents
}

\setlength{\parindent}{1.5em}


\section{Introduction}

In this paper, we study both the regularity and strict positivity of the density to the following stochastic heat equation (SHE) with rough initial conditions:
\begin{align}\label{E:SHE}
\begin{cases}
\displaystyle \left(\frac{\partial }{\partial t} -
\frac{1}{2}\Delta \right) u(t,x) =  \rho(t,x,u(t,x))
\:\dot{W}(t,x),&
x\in \R^d,\; t>0, \\
\displaystyle \quad u(0,\cdot) = \mu(\cdot),
\end{cases}
\end{align}
where $d\ge 1$ and $\Delta=\sum_{i=1}^d\partial^2/\partial x_i^2$ is the Laplacian operator.
The noise $\dot{W}$ is a centered Gaussian noise that is white in time and homogeneously colored/correlated in space. Informally,
\[
\E\left[\dot{W}(t,x)\dot{W}(s,y)\right] = \delta_0(t-s)f(x-y),
\]
where $\delta_0$ is the Dirac delta measure with unit mass at zero and $f$ is a ``correlation function'', 
that is, a nonnegative and nonnegative definite function that is not identically equal to zero.
The Fourier transform of $f$ is denoted by $\widehat{f}$
\[
\widehat{f}(\xi)= \calF f (\xi) =\int_{\R^d}\exp\left(- i \: \xi\cdot x \right)f(x)\ud x,
\]
which is again a nonnegative and nonnegative definite measure. Here, $\hat{f}$ 
is usually called  the {\it spectral measure}.
It is well known that the minimum assumption on the correlation function $f$ is {\it Dalang's condition} \cite{Dalang99}, namely,
\begin{align}\label{E:Dalang}
\Upsilon(\beta):=(2\pi)^{-d} \int_{\R^d} \frac{\widehat{f}(\ud \xi)}{\beta+|\xi|^2}<+\infty \quad \text{for some and hence for all $\beta>0$.}
\end{align}
Since we will need some regularity of the solution, we will work under the 
following slightly stronger condition, which will also be called Dalang's 
condition, than condition \eqref{E:Dalang} as in \cite{CH16Comparison}:
\begin{align}\label{E:Dalang2}
 \int_{\R^d} \frac{\widehat{f}(\ud \xi)}{\left(1+|\xi|^2\right)^{1-\alpha}}<+\infty,\quad\text{for some $\alpha\in (0,1]$.}
\end{align} 
One aim of this paper is to work under the weakest possible assumptions on $f$ 
but still under Dalang's condition \eqref{E:Dalang2}. In particular, we don't 
assume any scaling properties on $f$.

The nonlinear coefficient $\rho(t,x,z)$ is assumed to be a continuous function which is differentiable in the third argument with a bounded derivative. 
In particular, our results below will cover the important linear case 
$\rho(t,x,u)=\lambda u$, which is called  
the {\em parabolic Anderson model} (PAM) \cite{CarmonaMolchanov94}.

The precise meaning of the ``rough initial conditions/data'' are specified as follows.
We first note that by the Jordan decomposition, 
any signed Borel measure $\mu$ can be decomposed as $\mu=\mu_+-\mu_-$ where
$\mu_\pm$ are two non-negative Borel measures with disjoint support. Denote $|\mu|:= \mu_++\mu_-$.
The rough initial data refers to any signed (Borel) measure $\mu$ such that 
\begin{align}\label{E:J0finite}
\int_{\R^d} e^{-a |x|^2} |\mu|(\ud x)
<+\infty\;, \quad \text{for all $a>0$}\;.
\end{align}
Let $J_0(t,x)$ be the solution to the homogeneous equation, that is,
\begin{align}\label{E:J0}
J_0(t,x) = (\mu * G(t,\cdot))(x) := \int_{\R^d} G(t,x-y)\mu(\ud y),
\end{align}
where $G(t,x)$ is the heat kernel
\[
G(t,x):=(2\pi t)^{-d/2} \exp\left(-\frac{|x|^2}{2t}\right).
\]
It is easy to see that condition \eqref{E:J0finite} is equivalent to, in case $\mu\ge 0$, the condition
that the solution to the homogeneous equation $J_0(t,x)$ exists for all $t>0$ and $x\in\R^d$.
It is better to keep in mind the following examples of rough initial 
conditions: 
\begin{align}\label{E:ID-Ex}
\mu(\ud x) = \exp\left(|x|^{7/4}\right)\ud x,\qquad \mu= \delta_0 \quad\text{and}\quad \mu=\sum_{n\in\bbZ^d}\exp\left(\sqrt{|n|}\right)\delta_n,
\end{align}
where  $\delta_{x_0}$ is the Dirac delta measure with unit mass at $x=x_0$. Such rough initial conditions are important; see, e.g., Amir, Corwin and Quastel \cite{ACQ11} and Bertini and Giacomin \cite{BG97} where the Dirac delta initial data and the exponential of two sided Brownian motion, respectively, are their crucial choices for the initial conditions. 

\medskip

The aim of this paper is to establish the regularity and strict positivity of the joint density
of $(u(t,x_1),\cdots, u(t,x_m))$ for degenerate diffusion coefficient $\rho$, starting from rough initial data, and under the weakest possible assumptions on the correlation function $f$. 
In particular, due to the importance of the PAM, especially its relation with the stochastic Burgers and the Kardar-Parisi-Zhang (KPZ) equation through the Hopf-Cole transform, all our main results in this paper, namely, Theorems \ref{T:Single}, \ref{T:Mult}, \ref{T:NegMom} and \ref{T:Pos}, apply to the PAM with rough initial data.
The space time white noise case has been recently studied in \cite{CHN16Density}. 
We should emphasize that working on $\R^d$ with noise that is white in time and 
colored in space brings many more challenges that one does not have in the 
space-time white noise case. We will elaborate more on these difficulties in 
the next two subsections. 

Now we are ready to state our main results.

\subsection{Regularity of density}
The first set of results of this paper concern the regularity of the density. 
Theorem \ref{T:Single} below gives necessary and sufficient conditions for  
$u(t,x)$ to have a smooth density, but this result is only for a single 
space-time point.
By one additional assumption on $f$ (Assumption \ref{A:TailBlowup}) and by 
imposing a mild cone condition on the diffusion coefficient $\rho$,
Theorem \ref{T:Mult} below gives sufficient conditions for the existence 
of smooth density at multiple points. These results generalize corresponding 
results in \cite{CHN16Density} for space-time white noise case (hence $d=1$).

\begin{theorem}\label{T:Single}
Suppose that $\rho:[0,\infty)\times\R^d \times\R\mapsto\R$ is continuous and $f$ satisfies condition \eqref{E:Dalang2} for some $\alpha\in (0,1]$.
Let $u(t,x)$ be the solution  to \eqref{E:SHE} starting from an initial measure $\mu$
that satisfies \eqref{E:J0finite}.
Then the following two statements are true:
\begin{enumerate}
 \item[(a)] If $\rho$ is differentiable in the third argument with a bounded Lipschitz continuous derivative,
 then  for all $t>0$ and $x\in\R^d$,  $u(t,x)$ has an absolutely continuous law 
with respect to the Lebesgue measure on $\R$ if and only if 
 \begin{align}\label{E:IFF}
  t>t_0:=\inf\left\{s>0,\: \sup_{y\in\R^d}\left|\rho\left(s,y,(G(s,\cdot)*\mu)(y)\right)\right|\ne 0\right\}.
 \end{align}
 \item[(b)] If $\rho$ is infinitely differentiable in the third argument with bounded derivatives of all orders,
 then  for all $t>0$ and $x\in\R^d$,  $u(t,x)$ has a smooth  density if and 
only if condition \eqref{E:IFF} holds.
\end{enumerate}
\end{theorem}

This theorem will be proved in Section \ref{S:Single}.
\medskip

\begin{example}
Here we would like to point out several examples on $\rho$:\\
(1) Under the degenerate condition (see \eqref{E:r>c} below), it is clear from \eqref{E:IFF} that $t_0=0$.\\
(2) For the PAM, that is, $\rho(t,x,z)=\lambda z$, $\lambda\ne 0$, with delta initial data $\mu=\delta_0$, then $t_0=0$ because
\[
\sup_{y\in\R^d}|\rho(s,y,(G(s,\cdot)*\delta_0)(y))|=|\lambda| \sup_{y\in\R^d} G(s,y) = \frac{|\lambda|}{(2\pi s)^{d/2}}>0.
\]
(3) Let 
$\rho(t,x,z)=\exp\left(\left[(4z-1)(4z-3)\right]^{-1}\right)\one_{[1/4,3/4]}(z)$ 
and $\mu(\ud x) = \one_{[-1,1]}(x)\ud x$. Clearly, $\rho$ is a smooth function 
in the third argument with compact support.  Moreover, one can check from 
\eqref{E:IFF} that $t_0=0$. Therefore, there is a smooth density for all time 
$t>0$. However, Mueller and Nualart's result \cite{MuellerNualart08} fails in 
this example since $\rho(\one_{[-1,1]}(x))=0$ for all $x\in\R^d$. \\
(4) If $\rho(t,x,z)=\one_{[\tau,\tau']}(t)z$ and $\mu(\ud x)=\ud x$ with two deterministic constants $\tau'>\tau>0$, then $u(t,x)$ has a smooth density if only if $t>t_0=\tau$.
\end{example}

For the regularity of density at multiple points and also for the strict 
positivity of density (Theorem \ref{T:Pos} below), we need some 
assumptions on the correlation function $f$ as follows, 
which are satisfied by most common-seen examples (see 
Example \ref{eg:Reg-Kernel} below).


\begin{assumption}
\label{A:TailBlowup}
Assume that for some $A>1$ large enough, the correlation function satisfies that 
\[
\sup_{|z| \ge A} f(z)<\infty\qquad \text{and}\quad 
\inf_{|z|<1/A} f(z) >0.
\]
\end{assumption}

\begin{theorem}\label{T:Mult}
Suppose that the condition \eqref{E:Dalang2} is satisfied for some $\alpha\in(0,1]$.
Let $u(t,x)$ be the solution to \eqref{E:SHE} starting from a nonnegative measure $\mu>0$ that satisfies \eqref{E:J0finite}.
Suppose that for some constants $\beta>0$, $\gamma \in (0,1+\alpha)$ and $\lip_\rho>0$,
\begin{align}\label{E:lip}
|\rho(t,x,z)|\ge \lip_\rho \exp\left\{-\beta\left[\log\frac{1}{|z|\wedge 1}\right]^\gamma\right\},
\qquad\text{for all $(t,x,z)\in\R_+\times\R^d\times\R$}\,.
\end{align}
Then for any $m$ distinct points $\{x_1,\dots,x_m\}\subseteq\R^d$ and $t>0$,
under Assumption \ref{A:TailBlowup}, the following two 
statements are true:
\begin{enumerate}
 \item[(a)] If $\rho$ is differentiable in the third argument with a bounded Lipschitz continuous derivative, 
 then the law of the random vector $\left(u(t,x_1),\dots,u(t,x_m)\right)$ is absolutely continuous
 with respect to the Lebesgue measure on $\R^m$.
 \item[(b)] If $\rho$ is infinitely differentiable in the third argument with bounded derivatives of all orders,
 then the random vector $\left(u(t,x_1),\dots,u(t,x_m)\right)$ has a smooth  density on $\R^m$.
\end{enumerate}
\end{theorem}
This theorem will be proved in Section \ref{S:Mult}.

\medskip

\begin{example}
Let us see several examples on the cone condition \eqref{E:lip}:\\
(1) The degenerate condition (see \eqref{E:r>c} below) trivially ensures the cone condition \eqref{E:lip}. \\
(2) For the PAM, that is, $\rho(t,x,z)=\lambda z$, $\lambda\ne 0$, then the cone condition \eqref{E:lip} is satisfied with $\gamma=\beta=1$ and $\lip_\rho=|\lambda|$.\\
(3) Let $\rho(t,x,z)=z^{2n+1}/(1+z^{2n})$, $n\in\bbN$. This $\rho$ has linear growth for large $x$ and approaches zero as $x$ tends to zero in a polynomial rate. Hence, the cone condition \eqref{E:lip} is also satisfied.
\end{example}

\begin{example}\label{eg:Reg-Kernel}
We remark that all the commonly-seen correlation functions $f$ satisfy 
Assumption \ref{A:TailBlowup}. Here are some examples:\\
(1) For the Riesz kernel $f(x)=|x|^{-\beta'}$ with $\beta'\in (0,2\wedge d)$.\\
(2) Let $f(x)$ be the {\it Bessel kernel} of order $\alpha'>0$ (see, e.g., Section 6.1 of \cite{Grafakos14Modern})
\begin{align}\label{E:Bessel}
f(x)= \int_0^\infty w^{\frac{\alpha'-2-d}{2}}e^{-w}e^{-\frac{|x|^2}{4w}}\ud w.
\end{align}
Note that $f$ does not satisfy any scaling property. Dalang's condition 
\eqref{E:Dalang}/\eqref{E:Dalang2} requires that $\alpha'>d-2$; see Proposition \ref{P:Rate}. Assumption \ref{A:TailBlowup} is trivially satisfied. \\
(3) For the fractional noise $f(x) = \prod_{j=1}^d |x_j|^{2H_j-2}$ with $H_j\in (1/2,1)$ and 
$d-\sum_{j=1}^d H_i<1$,
Assumption \ref{A:TailBlowup} is not satisfied (the supremum outside a ball can be infinity). One can see that coordinate-wisely Assumption \ref{A:TailBlowup} is still true and Lemma \ref{L:VTildeV} can adapted to the coordinate-wise form easily. Hence, Theorem \ref{T:Mult} still holds in this case.
\end{example}

\begin{remark}
For the space-time white noise (hence $d=1$) case, while Assumption \ref{A:TailBlowup} is not satisfied.
Nevertheless, Theorem \ref{T:Mult} for the space-time white noise case has been proved in \cite{CHN16Density}. 
\end{remark}

Let us  now explain the difficulties and hence the contributions of the above two theorems. 
It is known that to establish the regularity of density one usually needs to apply the {\it Bouleau-Hirsch's criterion} (see Theorem \ref{T:Density} below).
The first difficulty is to show that $u(t,x)\in \D^\infty$ (see Section \ref{SS:Malliavin} for the notation).
For this property, one can either find a proof in Bally and Pardoux \cite{BP98} in case of SHE on an interval with Dirichlet boundary conditions, or in Que-Sardanyons and Sanz-Sol\'e \cite{QuerSanz04WaveSmooth} for the stochastic wave equation on $\R^3$ (see Nualart and Que-Sardanyons \cite{NualartQuer07} as well).
All these results rely on the property that
\begin{align}\label{E:SupNorm}
 \sup_{(t,x)\in [0,T]\times\R^d} \E\left[|u(t,x)|^p\right]<\infty,\quad \text{for any $T>0$,}
\end{align}
which doesn't hold any more for rough initial data. 
Hence, this quite standard fact --- $u(t,x)\in \D^\infty$  --- requires a proof in our setting. 
To prove this property, we need to introduce the Malliavin calculus {\it localized} to a space-time subdomain as in \cite{CHN16Density} in order to avoid possible singularities at $t=0$ and $x=\infty$; see Section \ref{SS:Malliavin} and Section \ref{S:MalliavinD} for more detail.

The second and also the major obstacle is to establish the negative moments 
of the determinant of the corresponding Malliavin matrix (see Section \ref{SS:Malliavin}). 
This obstacle can be circumvented by imposing the following nondegenerate condition
\begin{align}\label{E:r>c} 
 \inf_{x\in\R}|\rho(x)|\ge c \qquad \text{for some constant $c>0$.}
\end{align}
or similar conditions as in \cite{NualartQuer07} and \cite{HHNS14}.
However, this condition excludes the important case --- PAM. 
An alternative compromise to avoid this issue is to prove a ``local'' result as those in \cite{BP98}, where the smooth 
joint density of $(u(t,x_1),\dots,u(t,x_d))$ is proved over the domain $\{\rho\ne 0\}^d$ instead of $\R^d$.
As for the degenerate case, that is, the case that includes the PAM, 
Pardoux and Zhang \cite{PardouxZhang93} showed that the Malliavin matrix is 
invertible a.s., which enabled them to establish the existence of the density. 
Much later Mueller and Nualart \cite{MuellerNualart08} succeeded in establishing 
the smooth density. Both \cite{PardouxZhang93} and  \cite{MuellerNualart08} 
handle the one-dimensional SHE over an interval with space-time white noise. 
Recently, Chen, Hu and Nualart \cite{CHN16Density} extended the above results to the one-dimensional SHE over the whole $\R$. 
In this paper, we carry out this program to extend these results further to SHE on $\R^d$. 
Indeed, following similar ideas as in \cite{CHN16Density, MuellerNualart08}, we can transform the arguments of the proof of the strict positivity of the solution in \cite{CH16Comparison} into a stopping-time argument in order to obtain a better convergence rate (see \eqref{E:Rate}), which will in turn guarantee the existence of negative moments of all orders for a related SHE over $\R^d$. More precisely, we will establish Theorem \ref{T:NegMom} for the following SHE:
\begin{align}\label{E:SPDE-Var}
 \begin{cases}
  \left(\displaystyle\frac{\partial}{\partial t} - \frac{1}{2} \Delta \right) u(t,x) =
H(t,x) \sigma(t,x,u(t,x))  \dot{W}(t,x),& t>0\;,\: x\in\R^d,\\[1em]
u(0,\cdot) =\mu(\cdot),
 \end{cases}
\end{align}
where $H(t,x)$ is a bounded and adapted process and $\sigma(t,x,z)$ is a measurable and locally bounded function
which is Lipschitz continuous in $z$, uniformly in both $t$ and $x$, satisfying that  $\sigma(t,x,0)=0$.

\begin{theorem}
\label{T:NegMom}
Suppose that condition \eqref{E:Dalang2} is satisfied for some $\alpha\in(0,1]$.
Let $u(t,x)$ be the solution to \eqref{E:SPDE-Var} starting from a deterministic and nonnegative measure $\mu>0$ that satisfies \eqref{E:J0finite}.
Let $\Lambda>0$ be a constant such that
\begin{align}\label{E:Lambda}
\left|H(t,x,\omega)\sigma(t,x,z)\right|\le \Lambda |z| \qquad\text{for all $(t,x,z,\omega)\in \R_+\times\R^d\times\R\times\Omega$.}
\end{align}
Then for any compact set $K\subseteq\R^d$ and $t>0$,
there exists a finite constant $B>0$ which only depends on $\Lambda$, $\delta$, $K$ 
and $t>0$ such that for any $\delta\in (0,1)$ and for all small enough $\epsilon>0$,
\begin{align}\label{E:Rate}
\bbP\left(\inf_{x\in K}u(t,x) <\epsilon \right)&\le \exp\left(-B 
\left\{|\log(\epsilon)|\cdot | \log\left(|\log(\epsilon)|\right)|^\delta 
\right\}^{1+\alpha}\right).
\end{align}
Consequently, for all $p>0$,
\begin{align}\label{E:NonMom1}
\E\left(\left[\inf_{x\in K}u(t,x) \right]^{-p}\right)<+\infty.
\end{align}
\end{theorem}
This theorem will be proved in Section \ref{S:NegMom}.

\begin{remark}
\label{R:gPAM}
Finally, we would like to mention a recent work that is partially related to this paper.
In \cite{CFG17gPAM}, Cannizzaro, Friz and Gassiat studied the following generalized PAM (gPAM),
\[
\left(\partial_t -\Delta\right) u(t,x) = g(u(t,x))\xi(x), \quad u(0,\cdot)= u_0(\cdot),
\]
where $\xi=\xi(x,\omega)$ is space (not space-time) white noise, $x\in \mathbb{T}^2$ (two-dimensional torus) and the initial data is H\"older continuous.
The diffusion coefficient $g(\cdot)$ is assumed to be sufficiently smooth and to be compactly supported (see Proposition 3.28 [{\it ibid}]). This last property excludes the linear PAM.
The space dimension $2$ is the critical case when one needs to use Hairer's theory of the regularity structures \cite{Hairer2014RS,Hairer2011solving} to handle properly the renormalization procedure. For this path-wise approach, it is natural to obtain some almost-sure results such as the strict positivity of the solution; see Theorem 5.1 [{\it ibid}]. Hence by the first part of Bouleau-Hirsch's criterion (see part (1) of Theorem \ref{T:Density} below), they obtain the existence of the density.
However, it seems very hard to establish the smooth density in their 
framework since it requires the $L^p(\Omega)$-moments of the solution. 
Comparing to our results here, we work under the cases when the noises are 
good enough that there is no need of renormalization in order to make sense of 
the solution. 
The difficulties/contributions of this paper lie in degenerate 
diffusion coefficients $\rho$, rough initial data $\mu$, and weakest possible 
assumptions on the correlation function $f$ (but still under Dalang's condition 
\eqref{E:Dalang2}). Contrary to 
[{\it ibid}], the moments formula/bounds are the basic tools for us.
\end{remark}

\subsection{Strict positivity of density}
As for the strict positivity of the density, most known results assume the boundedness of the diffusion coefficient $\rho$;
see, e.g., Theorem 2.2 of Bally and Pardoux \cite{BP98},
Theorem 4.1 of Hu {\it et al} \cite{HHNS14} and Theorem 5.1 of E. Nualart \cite{Nualart13Density}.
This nondegenerate condition, which excludes the important case: the parabolic Anderson model 
$\rho(u)=\lambda u$, has been removed in a recent work by Chen {\it et al} \cite[Theorem 1.4]{CHN16Density}.
Moreover, the results in \cite{CHN16Density} allow the rough initial conditions. 
However, the results in \cite{CHN16Density} cover only the case of space-time white noise (hence $d=1$).
The goal of the next theorem is to extend Theorem 1.4 of \cite{CHN16Density} to higher spatial dimensions with more general noises. 
Recall that Theorem \ref{T:Mult} is proved under Assumption \ref{A:TailBlowup} on $f$. Here we need the following two assumptions on $f$.
Denote
\begin{align}\label{E:k}
k(t):=\int_{\R^d}f(z)G(t,z)\ud z =(2\pi)^{-d} 
\int_{\R^d}\widehat{f}(\ud\xi)\exp\left(-\frac{t|\xi|^2}{2}\right), 
\end{align}
where the second equality is due to the Plancherel theorem.

\begin{assumption}
\label{A:Rate-Sharp}
Assume that for some $\beta\in (0,1)$, the limit $\lim_{t\downarrow 0} t^\beta 
k(t)$
exists and belongs to $(0,\infty)$, or equivalently, $k(t)\asymp t^{-\beta}$ as 
$t\rightarrow 0_+$ (see the notation at the end of this section).
\end{assumption}

\begin{assumption}
\label{A:Continuity}
Assume that the correlation function $f$ is 
a locally bounded function on
$\R^d\setminus\{0\}$.
\end{assumption}

\begin{theorem}\label{T:Pos}
Suppose $\rho(t,x,z)=\rho(z)$ and $\rho\in C^\infty(\R)$ such that all derivatives of $\rho$ are bounded.
Let $\{x_1,\cdots,x_m\}\subseteq \R^d$ be $m$ distinct points. 
Then under Assumptions \ref{A:Rate-Sharp} and \ref{A:Continuity}, for any $t>0$,  
the joint law of the random vector $(u(t,x_1),\dots,u(t,x_m))$ admits a smooth density $p(y)$, and $p(y)>0$
if $y$ belongs both  to $\{\rho\ne 0\}^m$ and to the interior of the support of the law of $(u(t,x_1),\dots,u(t,x_m))$.
\end{theorem}

This theorem will be proved in Section \ref{S:Pos}.

\begin{example}\label{eg:Pos-Kernel}
We remark that all the commonly-seen correlation functions $f$ satisfy 
Assumption \ref{A:Rate-Sharp} (see Proposition \ref{P:Rate} below for the proof). Here are some examples:\\
(1) For the space-time white noise case (hence $d=1$), Assumption \ref{A:Rate-Sharp} is satisfied with $\beta=1/2$.\\
(2) For the Riesz kernel $f(x)=|x|^{-\beta'}$ with $\beta'\in (0,2\wedge d)$, 
Assumption \ref{A:Rate-Sharp} is satisfied with $\beta=\beta'/2$.\\
(3) For the fractional noise $f(x) = \prod_{j=1}^d |x_j|^{2H_j-2}$ with $H_j\in (1/2,1)$ and 
$d-\sum_{j=1}^d H_i<1$, Assumption \ref{A:Rate-Sharp} is satisfied with $\beta=d-\sum_{j=1}^d H_j$.\\
(4) When $f(x)$ is the Bessel kernel \eqref{E:Bessel} of order $\alpha'\in(0\vee (d-2),d)$,  Assumption \ref{A:Rate-Sharp} is satisfied with $\beta=(d-\alpha')/2$.
\end{example}

\begin{remark}\label{R2:Examples}
Examples (1), (2) and (4) in Example \ref{eg:Pos-Kernel} clearly satisfy Assumption 
\ref{A:Continuity}. 
In general, let $H\subseteq\R^d$ be the set of points where $f$ fails to be 
locally bounded. In this case, Theorem \ref{T:Pos} is still true provided that the 
$m$ points 
$\{x_1,\dots,x_m\}\subseteq\R^d$ satisfy the condition that
\[
\mathop{\cup_{i,j=1,\cdots m}}_{i\ne j} \{x_i-x_j\} \cap H = \emptyset.
\]
In case of examples (1), (2) and (4) in Example \ref{eg:Pos-Kernel}, $H=\{0\}$. 
Hence, we only need to require that all $m$ points are distinct. In case 
of example (3) in Example \ref{eg:Pos-Kernel}, we have that $H=\{z\in\R^d, 
\text{$z_i=0$ for some $i=1,\dots, d$}\}$. Hence, we need to instead require 
that the projections of the $m$ points on each coordinate are distinct. 
\end{remark}

We should emphasize that moving from the space-time white noise to the spatially colored noises of the current paper makes a significant difference in the perturbation strategy. For the space-time white noise case as in \cite{BP98, CHN16Density}, the perturbation function takes the following form (see, e.g., Eq. (8.8) of \cite{CHN16Density})
\[
h_n^i(t,x) = c_n \one_{\{[T-2^{-n},T]\}}(t)\one_{\{[x_i-2^{-n},x_i+2^{-n}]\}}(x),
\]
where $c_n$ is some normalization constant. Here, $x$ lives in a compact set, 
which could be used to simplify many arguments. However, for the spatially 
colored noise,  one needs to take the following perturbation function (see 
\eqref{E:hni} below) as in \cite{Nualart13Density}, 
\[
h_{n}^i(t,x)= c_n \one_{[T-2^{-n},T]}(t) G(T-s,x_i-x).
\]
Here,  $x$ no longer lives in a compact set, which makes arguments much more involved than those in \cite{CHN16Density}. 

Another complication/difficulty comes from the rough initial data (see examples in \eqref{E:ID-Ex}). Since we need to prove some almost-sure results for some supremums (see \eqref{E:my(ii)}), we will need to show many sharp moments estimates and their increments in order to first apply the Kolmogorov continuity theorem to take care of the supremums and then apply Borel-Cantelli lemma to obtain path-wise results. 
If one assumes that initial data to be bounded, then property \eqref{E:SupNorm} holds, with which all these arguments can be significantly simplified.

\begin{remark}
Finally, let us point out some subtleties on the assumptions on the correlation 
functions $f$ among the three main theorems of this paper, namely, Theorems 
\ref{T:Single}, \ref{T:Mult} and \ref{T:Pos}. 
Theorem \ref{T:Single} makes the weakest assumption on the correlation function 
$f$, namely, Dalang's condition \eqref{E:Dalang2}. 
All examples on $f$ in Examples \ref{eg:Reg-Kernel} and \ref{eg:Pos-Kernel} 
work for Theorem \ref{T:Single}. 
On the other hand, one can construct examples as follows that work for Theorem 
\ref{T:Single} but not for either Theorem \ref{T:Mult} or Theorem 
\ref{T:Pos}:\\
(1) For $d=1$ and for any $a>0$ and $c\in [0,1/2]$, one can show that 
\[
f(x) = \delta_0(x) + c 
(\delta_{a}(x) + \delta_{-a}(x))
\]
is nonnegative and nonnegative definite 
since $\widehat{f}(\xi)=1+2c\cos(a\xi)\ge 0$. More generally, one can have 
\[
f(x) = \delta_0(x) + \sum_{i=1}^\infty c_i 
(\delta_{a_i}(x) + \delta_{-a_i}(x))
\]
with $a_i>0$ and $c_i>0$ such that $\sum_{i=1}^\infty c_i\le 1/2$.
In this case, $\widehat{f}(\xi)=1+2\sum_{i=1}^\infty c_i\cos(a_i\xi)$.
\\
(2) Similarly, in any spatial dimensions $d\ge 1$,
one can construct the following example: Let $\beta\in(0,d\wedge 2)$ and 
\[
f(x) = |x|^{-\beta} + \sum_{i=1}^\infty c_i 
(|x-a_i|^{-\beta} + |x+a_i|^{-\beta})
\]
where $a_i\in\R^d \setminus\{0\}$ and $c_i>0$ such that $\sum_{i=1}^\infty 
c_i\le 1/2$. In this case, 
\[
\widehat{f}(\xi) =  C_{\beta, d} |\xi|^{\beta-d}\left(1+2\sum_{i=1}^\infty c_i 
\cos(a_i\cdot\xi)\right)\ge 0.
\]
\end{remark}

\subsection{Outline of the paper and some notation}

The paper is organized as follows.  
In Section \ref{S:Pre}, we give some preliminaries over the definition of the solution and some known results about the solution that will be used in this paper (Section \ref{SS:Definition}), some auxiliary functions (Section \ref{SS:AuxFun}), Malliavin calculus and its localized version (Section \ref{SS:Malliavin}), and a criterion for the strict positivity of density (Section \ref{SS:Criterion}).

The two regularity results (Theorems \ref{T:Single} and \ref{T:Mult}) are proved in Section \ref{S:Regularity}. In particular, we first prove the existence of negative moments of all orders (Theorem \ref{T:NegMom}) in Section \ref{S:NegMom}. 
In Section \ref{S:MalliavinD}, we study the Malliavin derivatives of $u(t,x)$ in the localized Sobolev spaces.
The existence and smoothness of density at a single space-time point (Theorem \ref{T:Single}) and at multiple points (Theorem \ref{T:Mult}) are proved in Sections \ref{S:Single} and \ref{S:Mult}, respectively.

The strict positivity of the density (Theorem \ref{T:Pos}) is proved in Section \ref{S:Pos}.
The proof of Theorem \ref{T:Pos} is outlined in Section \ref{SS:Pos}, which is essentially an application of Theorem \ref{T:Criteria} in Section \ref{SS:Criterion}. 
The rest parts, namely, subsections from \ref{SS:Psin} to \ref{SS:Bddedness}, 
are devoted to prove all properties that are needed in the proof of Theorem 
\ref{T:Pos} in Section \ref{SS:Pos}.

\bigskip
Throughout this paper, we  use $C$ to denote a generic constant whose value may vary at different occurrences.
The function $\rho(u(t,x))$ should be understood as $\rho(t,x,u(t,x))$.
Let $\Norm{\cdot}_p$ denote the $L^p(\Omega)$-norm.
 For $x\in\R^d$ and $A\in (\R^{d})^{\otimes m}$, 
\[
|x| := \sqrt{x_1^2+\cdots+x_d^2}\quad\text{and}\quad
\Norm{A}:=\left(\sum_{i_1,\dots,i_m=1}^d A_{i_1,\dots,i_m}^2\right)^{1/2}.
\]
For two functions $g_1,g_2 : \R_+ \mapsto \R$, we write $g_1(t) \asymp g_2(t)$ as $t\rightarrow 0_+$ if for some constants $C_1, C_2>0$, 
both $g_1(t) \le C_1 g_2(t)$ and 
$g_2(t) \le C_2 g_1(t)$ hold as $t\rightarrow0_+$.

\numberwithin{equation}{subsection}
\section{Some preliminaries}
\label{S:Pre}

\subsection{Definition and existence of a solution}
\label{SS:Definition}

Let $\calH$ be the completion of the Schwartz space $\calS(\R^d)$ of rapidly decreasing smooth functions, endowed with the inner product
\[
\InPrd{\phi,\psi}_{\calH}=\int_{\R^d}\int_{\R^d} \phi(x)f(x-y)\psi(y)\ud x\ud y 
=
\int_{\R^d}
\calF\phi(\xi)\overline{\calF\psi(\xi)} \widehat{f}(\ud\xi)
,\quad\text{for $\phi,\psi\in\calS(\R^d)$}.
\]
For any measurable sets $\calT\subseteq [0,\infty)$ and $S\subseteq \R^d$, let 
$\calH_{\calT,\calS}$  be the completion of $\calS(\R_+\times \R^d)$ with respect to the inner product
\begin{equation}\label{E:HTS}
\langle \phi, \psi\rangle_{\mathcal{H}_{\calT,\calS}} = \int_{\calT} \int_{S}\int_{S} \phi(s, x) \psi(s, y) f(x-y) \ud x \ud y \ud s\,,\quad\text{for $\phi,\psi\in\calS(\R_+\times\R^d)$}.
\end{equation}
It is known that both $\calH$ and $\calH_{\calT,\calS}$ may contain distributions.

Recall that a {\em spatially homogeneous Gaussian noise that is white in time} is an
$L^2(\Omega)$-valued mean zero Gaussian process on a complete probability space $\left(\Omega,\calF,\bbP\right)$
\[
\left\{W(\psi):\: \psi\in C_c^{\infty}\left([0,\infty)\times\R^{d}\right)\:\right\},
\]
such that
\begin{align}\label{E:Cov}
\E\left[W(\psi)W(\phi)\right] = \int_0^{\infty} \ud s\iint_{\R^{2d}}\psi(s,x)\phi(s,y)f(x-y)\ud x\ud y.
\end{align}
Let $\calB_b(\R^d)$ be the collection of Borel measurable sets with finite Lebesgue measure.
As in Dalang-Walsh theory \cite{Dalang99,DalangEtc09Minicourse, DalangQuer, FK13SHE, Walsh},
one can extend $F$ to a $\sigma$-finite $L^2(\Omega)$-valued martingale measure $B\mapsto F(B)$
defined for $B\in \calB_b(\R_+\times\R^d)$. Then define
\[
W_t(B) :=F\left([0,t]\times B \right), \qquad B\in\calB_b(\R^d).
\]
Let $(\calF_t,t\ge 0)$ be the natural filtration generated by $M_\cdot(\cdot)$
and augmented by all $\bbP$-null sets $\calN$ in $\calF$, that is,
\[
\calF_t := \sigma\left(W_s(A):\: 0\le s\le t,
A\in\calB_b\left(\R^d\right)\right)\vee
\calN,\quad t\ge 0,
\]
Then for any adapted, jointly measurable (with respect to
$\calB\left((0,\infty)\times\R^d\right)\times\calF$) random field $\{X(t,x): t>0,x\in\R^d\}$ such that
\[
\int_0^\infty\ud s\iint_{\R^{2d}}\ud x\ud y\:
\Norm{X(s,y)X(s,x)}_{p/2} f(x-y) <\infty,
\]
the stochastic integral
\[
\int_0^\infty \int_{\R^d} X(s,y)W(\ud s\ud y)
\]
is well-defined in the sense of Dalang-Walsh. Here we only require the joint-measurability instead of
predictability; see Proposition 2.2 in \cite{CK15SHE} for this case or
Proposition 3.1 in \cite{ChenDalang13Heat} for the space-time white noise case.
Throughout this paper, $\Norm{\cdot}_p$ denotes the $L^p(\Omega)$-norm.

\bigskip

We formally write the SPDE \eqref{E:SHE} in the integral form
\begin{align}\label{E:WalshSI}
u(t,x)= J_0(t,x)+ I(t,x)
\end{align}
where
\[
I(t,x):= \iint_{[0,t]\times \R^d} G(t-s,x-y) \rho(u(s,y))W(\ud s\ud y).
\]
The above stochastic integral is understood in the sense of Walsh \cite{Dalang99,Walsh}.

\begin{definition}\label{D:Solution}
A process $u=\left(u(t,x),\:(t,x)\in(0,\infty)\times\R^d \right)$  is called a {\it
random field solution} to \eqref{E:SHE} if
\begin{enumerate}[(1)]
 \item $u$ is adapted, that is, for all $(t,x)\in(0,\infty)\times\R^d$, $u(t,x)$ is
$\calF_t$-measurable;
\item $u$ is jointly measurable with respect to
$\calB\left((0,\infty)\times\R^d\right)\times\calF$;
\item $\Norm{I(t,x)}_2<+\infty$ for all $(t,x)\in(0,\infty)\times\R^d$;
\item  $I$ is $L^2(\Omega)$-continuous, that is, the function $(t,x)\mapsto I(t,x)$ mapping $(0,\infty)\times\R^d$ into
$L^2(\Omega)$ is continuous;
\item $u$ satisfies \eqref{E:WalshSI} a.s.,
for all $(t,x)\in(0,\infty)\times\R^d$.
\end{enumerate}
\end{definition}

The following results are from \cite{CK15SHE} and \cite{CH16Comparison}, where $\rho$ is assumed to be a function of one variable. The extension to the current setting, that is, $\rho$ is a function of three variables, is routine.  
In particular, Theorem \ref{T:ExUni} gives the existence and uniqueness of a random field solution. 
Theorem \ref{T:Mom} supplies us with some useful moment bounds. 
Theorem \ref{T:Holder} gives the H\"older regularity of the solution and finally, 
the comparison principle is stated in Theorem \ref{T:Comp}.

\begin{theorem}[Theorem 2.4 in \cite{CK15SHE}]
\label{T:ExUni}
If the initial data $\mu$ satisfies \eqref{E:J0finite},
then under Dalang's condition \eqref{E:Dalang},
SPDE \eqref{E:SHE} has a unique (in the sense of versions) random field solution
$\left\{u(t,x): t>0,x \in\R^d\right\}$ starting from $\mu$.
This solution is $L^2(\Omega)$-continuous.
\end{theorem}

\begin{theorem}[Theorem 1.5 in \cite{CH16Comparison}]
\label{T:Mom}
Under Dalang's condition \eqref{E:Dalang},
if the initial data $\mu$ is a signed measure that satisfies \eqref{E:J0finite},
then the solution $u$ to \eqref{E:SHE} for any given $t>0$ and $x\in\R^d$ is in $L^p(\Omega)$, $p\ge 2$, and
\begin{align}\label{E:Mom}
 \Norm{u(t,x)}_p \le
 \sqrt{2}\left[\Vip+\sqrt{2} \left(|\mu|*G(t,\cdot)\right)(x) \right]H\left(t;\gamma_p\right)^{1/2},
\end{align}
where $\Vip=|\rho(0)|/\LIP_\rho$ and $\gamma_p=32p\LIP_\rho^2$ and $H(t;\gamma_p)$ is defined in \eqref{E:H} below.
Moreover, if for some $\alpha\in (0,1]$ condition \eqref{E:Dalang2} is satisfied,
then when $p\ge 2$ is large enough, there exists some constant $C>0$ such that 
{
\begin{align}\label{E:MomAlpha}
\Norm{u(t,x)}_p \leq C \Big[\: \Vip+\left(|\mu|*G(t,\cdot)\right)(x) \Big] \exp\left(C p^{1/\alpha} t\right)\,.
\end{align}}
\end{theorem}

\begin{theorem}[Theorem 1.6 of \cite{CH16Comparison}]
\label{T:Holder}
Suppose that $\mu$ is any measure that satisfies \eqref{E:J0finite}
and $f$ satisfies \eqref{E:Dalang2} for some $\alpha\in (0,1]$.
Then the solution to \eqref{E:SHE} starting from $\mu$ is a.s.
$\beta_1$-H\"older continuous in time and
$\beta_2$-H\"older continuous in space on $(0,\infty)\times\R^d$ with
\[
\beta_1\in \left(0,\alpha/2\right)\quad
\text{and}\quad
\beta_2\in \left(0,\alpha\right).
\]
\end{theorem}

\begin{theorem}[Comparison principle \cite{CH16Comparison}]
\label{T:Comp}
Assume that $f$ satisfies Dalang's condition \eqref{E:Dalang}.
Let $u_1(t,x)$ and $u_2(t,x)$ be two solutions to \eqref{E:SHE}
with the initial measures $\mu_1$ and $\mu_2$ that satisfy \eqref{E:J0finite}, respectively. Then 
\begin{itemize}
\item[(a)] (Weak comparison principle) If $\mu_1 \le \mu_2$, then 
\begin{equation}\label{E: W comp path}
\bbP \left(u_1(t,x)\leq u_2(t,x)\right)=1\,\qquad \text{for all $t \geq 0$ and $x\in\R^d$}.
\end{equation} 
\item[(b)] (Strong comparison principle) 
If, in addition, $f$ satisfies \eqref{E:Dalang2} 
for some $\alpha\in (0,1]$, then $\mu_1 < \mu_2$ implies that
\begin{equation}\label{E: S comp path}
\bbP \left(u_1(t,x)<u_2(t,x) \ \text{for all}\  t \geq 0 \ \text{and}\  x\in \RR^d\right)=1\,. 
\end{equation} 
\end{itemize}
\end{theorem}

\subsection{Some auxiliary functions}
\label{SS:AuxFun}

Recall that $k(t)$ is defined in \eqref{E:k}.
Define 
\begin{align}\label{E:hn}
   h_0(t):=1\quad\text{and}\quad h_n(t)= \int_0^t \ud s \: h_{n-1}(s) k(t-s) \quad \text{if $n\ge 1$.}
\end{align}
Let
\begin{align}\label{E:H}
 H(t;\gamma):= \sum_{n=0}^\infty \gamma^n h_n(t).
\end{align}
This function is defined through the correlation function $f$. The following lemma
tells us that this function has an exponential asymptotic bound.

\begin{lemma}[Lemma 2.5 in \cite{CK15SHE} or Lemma 3.8 in \cite{BC16}]
\label{L:EstHt}
For all $t\ge 0$ and $\gamma\ge 0$,
\begin{align}
\label{E:Var} 
\limsup_{t\rightarrow\infty} \frac{1}{t}\log H(t;\gamma)\le \inf\left\{\beta>0:  \:\Upsilon\left(\beta\right) < \frac{1}{\gamma}\right\}.
\end{align}
\end{lemma}

\begin{lemma}[Lemma 3.1 in \cite{CH16Comparison}]
\label{L:IntIneq}
Suppose that $\mu$ is a signed measure that satisfies condition \eqref{E:J0finite} and
let $J_0(t,x)$ be the solution to the homogeneous equation (see \eqref{E:J0}).
If a nonnegative function $g:\R_+\times\R^{d}\mapsto\R_+$ satisfies the following integral inequality
\begin{align}\label{E:IntIneq1}
\begin{aligned}
 g(t,x)^2\le J_0^2(t,x)+ \lambda^2 \int_0^t\ud s\iint_{\R^{2d}} &
 G(t-s,x-y_1) G(t-s,x-y_2) \\
 &\times f(y_1-y_2) g(s,y_1)g(s,y_2)\ud y_1\ud y_2,
\end{aligned}
\end{align}
for all $t>0$ and $x\in\R^d$, then{
\begin{align}\label{E:IntIneq2}
\begin{aligned}
 g(t,x)\le \left(|\mu|*G(t,\cdot)\right)(x) \: H(t;2\lambda^2)^{1/2}.
\end{aligned}
\end{align}}
\end{lemma}

In the proof of Lemma 3.1 of \cite{CH16Comparison}, the authors actually prove the following result, which will be useful in this paper.
\begin{lemma}\label{L:IntIneq2}
Suppose that $\mu$ is a signed measure that satisfies condition \eqref{E:J0finite}. For $\lambda\ge 0$, define
\begin{align}\label{E:gn}
\begin{aligned}
g_0(t,x)&:=(|\mu|*G(t,\cdot))(x) \quad\text{and for $n\ge 1$}\\
 g_{n}^2(t,x) &= J_0^2(t,x) + \lambda^2 \int_0^t\ud s\iint_{\R^{2d}}  G(t-s,x-y_1)G(t-s,x-y_2) f(y_1-y_2)\\
 &\hspace{12em}\times g_{n-1}(s,y_1)g_{n-1}(s,y_2)\ud y_1 \ud y_2,
\end{aligned}
\end{align}
with $t>0$ and $x\in\R^d$. 
Then for all $n\ge 0$, $t>0$ and $x\in\R^d$, it holds that
\begin{align}\label{E:Indt-Lp}
g_n(t,x)\le g_0(t,x) \left(\sum_{i=0}^n (2\lambda^2)^i h_i(t)\right)^{1/2}.
\end{align}
\end{lemma}

\begin{lemma}\label{L:iterate}
Suppose that $\mu$ is a signed measure that satisfies condition \eqref{E:J0finite}. For $\lambda> 0$, define
\begin{align*}
g(t,x) &:= \left(|\mu|*G(t,\cdot)\right)(x) H(t;2\lambda^2)^{1/2} \\
h(t,x) &:= \int_0^t\ud s\iint_{\R^{2d}}\ud y\ud y'
\: G(t-s,x-y)g(s,y) f(y-y') g(s,y') G(t-s,x-y').
\end{align*}
Then
\[
h(t,x)\le \lambda^{-2} g^2(t,x).
\]
\end{lemma}
\begin{proof}
Set $\gamma=2\lambda^2$ and $g_0(t,x):= \left(|\mu|*G(t,\cdot)\right)(x)$.
By the same arguments as the proof of Lemma 3.1 of \cite{CH16Comparison} and the definition of $H(t;\gamma)$ in \eqref{E:H}, we see that 
\begin{align*}
h(t,x)\le &2g_0^2(t,x)\int_0^t \ud s\left(\sum_{i=0}^\infty \gamma^i h_i(s)\right) k(t-s)\\
= &2g_0^2(t,x)\sum_{i=0}^\infty \gamma^i h_{i+1}(s)\\
=&\frac{1}{\lambda^2} g^2_0(t,x)\left(H(t;\gamma)-h_0(t)\right) \le 
\frac{1}{\lambda^2} g^2_0(t,x)H(t;\gamma).
\end{align*}
\end{proof}

The following identity will be used many times in this paper:
\begin{equation}\label{E:GGGG}
G(s, x)G(t,y) = G(s+t, x+y) G\left(\frac{st}{s+t}, \frac{tx-sy}{s+t}\right)\,.
\end{equation}

\subsection{Malliavin calculus}\label{SS:Malliavin}
Now we recall some basic facts on  Malliavin calculus associated with $W$.
Denote by $C_p^\infty(\R^n)$ the space of smooth functions with
all their partial derivatives having at most polynomial growth at infinity.
Let $\calS$ be the space of simple functionals of the form
\begin{align}\label{E:Ff}
F=f(W(A_1),\dots,W(A_n)),
\end{align}
where $f\in C_p^\infty(\R^n)$ and
$A_1,\dots, A_n$ are Borel subsets of $[0,\infty)\times\R^d$ with finite Lebesgue measure.
The derivative of $F$ is a two-parameter stochastic process defined as follows
\[
D_{t,x}F=\sum_{i=1}^n \frac{\partial f}{\partial x_i}\left(W(A_1),\dots,W(A_n)\right)\one_{A_i}(t,x).
\]
In a similar way we define the iterated derivative $D^{k}F$.
The derivative operator $D^{k}$ for positive integers $k\ge 1$ is a closable  operator from $L^p(\Omega)$
into $L^p\left(\Omega;L^2([0,\infty);\calH)^{k}\right)$ for any $p\ge 1$.
Let $k$ be some positive integer.
For any $p>1$, let $\D^{k,p}$ be the completion of $\calS$ with respect to the
norm
\begin{align}
\begin{aligned}
\Norm{F}_{k,p}^p:= \E\left(|F|^p\right)+\sum_{j=1}^k\E\Bigg[
\Big(\int_{([0,\infty)\times\R^{2d})^{j}}
&\left(D_{\theta_1,z_1}\cdots D_{\theta_j,z_j}F\right)
\left(D_{\theta_1,z_1'}\cdots D_{\theta_j,z_j'}F\right)\\
&\times \prod_{i=1}^j f(z_i-z_i')\ud \theta_i \ud
z_i\ud z_i'\Big)^{p/2}
\Bigg].
\end{aligned}
\end{align}
Denote $\D^{\infty}:=\cap_{k,p}\D^{k,p}$.

Suppose that $F=(F^1,\dots,F^d)$ is a $d$-dimensional random vector whose components are in $\D^{1,2}$.
The following random symmetric nonnegative definite matrix
\begin{align}
\sigma_F = \left(
\InPrd{D F^i,DF^j}_{L^2([0,\infty);\calH)}
\right)_{1\le i,j\le d}
\end{align}
is called the {\it Malliavin matrix} of $F$.
The classical criteria for the existence and regularity of the density are the
following:

\begin{theorem}[Bouleau and Hirsch \cite{BH91}]
\label{T:Density}
 Suppose that $F=(F^1,\dots,F^d)$ is a $d$-dimensional random vector whose components are in $\D^{1,2}$. Then
\begin{enumerate}
 \item[(1)] If $\det(\sigma_F)>0$ almost surely, the law of $F$ is absolutely continuous with respect to the Lebesgue measure.
 \item[(2)] If $F^i\in\D^\infty$ for each $i=1,\dots,d$ and $\E\left[\left(\det\sigma_F\right)^{-p}\right]<\infty$
 for all $p\ge 1$, then $F$ has a smooth  density.
\end{enumerate}
\end{theorem}

As in \cite{CHN16Density}, we need to introduce a localized version of the above theorem in order to deal with the rough initial data. For any measurable sets $\calT\subset [0,\infty)$ and $S\subset \R^d$, and for any $p\ge 1$,
let $\D^{k,p}_{\calT,\calS}$ be the completion of $\calS$ with respect to the
norm
\begin{align}
\begin{aligned}
\Norm{F}_{k,p,\calT,\calS}^p:= \E\left(|F|^p\right)+\sum_{j=1}^k\E\Bigg[
\Big(\int_{(\calT\times S^2)^{j}}
&\left(D_{\theta_1,z_1}\cdots D_{\theta_j,z_j}F\right)
\left(D_{\theta_1,z_1'}\cdots D_{\theta_j,z_j'}F\right)\\
&\times \prod_{i=1}^j f(z_i-z_i')\ud \theta_i \ud
z_i\ud z_i'\Big)^{p/2}
\Bigg].
\end{aligned}
\end{align}
Similarly, denote $\D^{\infty}_{\calT,\calS}:=\cap_{k,p}\D^{k,p}_{\calT,\calS}$.
Let $\calH_{\calT,\calS}$ be the Hilbert space completed from the Schwartz space $\calS(\R^{1+d})$ with respect to the following inner product:
\begin{align}\label{E:HTS-norm}
\InPrd{g,h}_{\calH_{\calT,\calS}} = \int_{\calT\times \calS^2} g(s,y)h(s,z) f(y-z)\ud y\ud z\ud s.
\end{align}
Note that $\calH_{\calT,\calS}$ may contain distributions.

\begin{theorem}[Chen {\it et al} \cite{CHN16Density}]
\label{T:Density2}
 Suppose that $F=(F^1,\dots,F^d)$ is a $d$-dimensional random vector whose components are in $\D^{1,2}_{\calT,\calS}$. Let 
 \[
 \sigma_{F,\calT,\calS} := \left(
\InPrd{D F^i,DF^j}_{\calH_{\calT,\calS}}
\right)_{1\le i,j\le d}.
 \]
 Then
\begin{enumerate}
 \item[(1)] If $\det(\sigma_{F,\calT,\calS})>0$ almost surely, the law of $F$ is absolutely continuous with respect to the Lebesgue measure.
 \item[(2)] If $F^i\in\D_{\calT,\calS}^\infty$ for each $i=1,\dots,d$ and $\E\left[\left(\det\sigma_{F,\calT,\calS}\right)^{-p}\right]<\infty$
 for all $p\ge 1$, then $F$ has a smooth  density.
\end{enumerate}
\end{theorem}

\begin{lemma}[Chen {\it et al} \cite{CHN16Density}]
\label{L:Dinfty}
Let $\{F_m,m\ge 1\}$ be a sequence of random variables converging to $F$ in
$L^p(\Omega)$ for some $p>1$. Suppose that $\sup_{m}\Norm{F_m}_{n,p,\calT,\calS}<\infty$
for some integer $n\ge 1$. Then $F\in\D_{\calT,\calS}^{n,p}$.
\end{lemma}

\subsection{Sufficient conditions for strict positivity of density}
\label{SS:Criterion}
In this part, we introduce a criterion for the strict positivity of density. 
Recall that $W=\{W_t,t\ge 0\}$ can be viewed as a cylindrical Wiener process in the Hilbert space $\calH$ with
the covariance given by \eqref{E:Cov}.
Let $\mathbf{h}=(h^1,\dots,h^m) \in L^2(\R_+;\calH)^m$ and $\mathbf{z}=(z_1,\dots,z_m)\in\R^m$.
Define a translation of $W_t$, denoted by  $\widehat{W}_t$, as follows:
\begin{equation}
\widehat{W}_t(g):=
W_t(g) + \sum_{i=1}^m z_i\int_0^t \ud s\iint_{\R^{2d}} \ud y\ud y'\:  h^i(s,y) g(y) f(y-y'), \quad\text{for any $g\in \calH$.}
\label{e.def-hat-W}
\end{equation}
Then $\big\{\: \widehat{W}_t, \: t\ge 0\big\}$ is a cylindrical Wiener process in $\calH$ on the
probability space $(\Omega,\calF,\widehat{\mathbb{P}})$, where
\begin{align*}
\frac{\ud \widehat{\mathbb{P}}}{\ud \mathbb{P}}=\exp\Bigg( &
-\sum_{i=1}^m z_i \int_0^\infty\ud s\iint_{\R^{2d}} h^i(s,y)W(\ud s\ud y)\\
&-\frac{1}{2}\sum_{i=1}^m z_i^2\int_0^\infty\ud s \int_{\R^{2d}}\ud y\ud y'\: 
h^i(s,y)h^i(s,y') f(y-y')
\Bigg).
\end{align*}

For any predictable process $Z\in L^2(\Omega\times\R_+; \calH)$, we have that
\begin{align*}
\int_0^\infty \int_{\R^d} Z(s,y)\widehat{W}(\ud s\ud y) = &
\int_0^\infty \int_{\R^d} Z(s,y)W(\ud s\ud y) \\
&+\sum_{i=1}^m 
z_i\int_0^\infty\int_{\R^{2d}} Z(s,y)h^i(s,y')f(y-y')\ud s\ud y\ud y'.
\end{align*}
In the following, we write $\sum_{i=1}^m z_i h^i(s,y)=:\InPrd{\mathbf{z},\mathbf{h}(s,y)}$.
Let $\widehat{u}_{\mathbf{z}}^n(t,x)$ be the solution to \eqref{E:SHE} with respect to $\widehat{W}$, that is,
\begin{align}
\begin{aligned}
 \widehat{u}_{\mathbf{z}}(t,x) = J_0(t,x)&+ \int_0^t\int_{\R^d} G(t-s,x-y) \rho(\widehat{u}_{\mathbf{z}}(s,y))W(\ud s\ud y)\\
 &+\int_{0}^t\int_{\R^{2d}} G(t-s,x-y)\rho(\widehat{u}_{\mathbf{z}}(s,y)) 
\InPrd{\mathbf{z},\mathbf{h}(s,y')}f(y-y')\ud s\ud y\ud y'.
\end{aligned}
\end{align}
Then, the law of $u(t,x)$ under $\mathbb{P}$ coincides with the law of $\widehat{u}_{\mathbf{z}}(t,x)$ under $\widehat{\mathbb{P}}$.

\bigskip
The following theorem is an extension of Theorem 3.3 of Bally and Pardoux \cite{BP98}, which allows one to consider the case with unbounded diffusion parameter such as the parabolic Anderson model.

\begin{theorem}[Chen {\it et al} \cite{CHN16Density}]
\label{T:Criteria}
Let $F$ be an $m$-dimensional random vector measurable with respect to $W$,
such that each component of $F$ is in $\D^{3,2}$.
Assume that for some $f\in C(\R^m)$  and for some  open subset
 $\Gamma$ of $\R^m$, it holds that
\[
\one_\Gamma(y) \: \left[\bbP\circ F^{-1}\right] (\ud y)
=
\one_\Gamma(y) f(y) \ud y.
\]
Fix a point  $y_*\in\Gamma$. Suppose that there exists a sequence $\{\mathbf{h}_n\}_{n\in\bbN} \subseteq  L^2(\R_+;\calH)^m$
such that the associated random field
$\phi_n(\mathbf{z})= F(\widehat{W}^{\mathbf{z},n})$
satisfies the following two conditions.
\begin{enumerate}
\item[(i)] There are constants $c_0>0$ and $r_0>0$ such that for all $r\in (0,r_0]$,
the following limit holds true:
\begin{align}\label{E:my(i)}
\liminf_{n\rightarrow\infty}
\bbP\left(
|F-y_*|\le r
\:\:\text{and}\:\:
|\det\partial_{\mathbf{z}} \phi_n(0)|\ge \frac{1}{c_0}
\right)>0.
\end{align}
\item[(ii)]  There are some constants $\kappa>0$  and $K>0$ such that
\begin{align}\label{E:my(ii)}
 \lim_{n\rightarrow\infty}
 \bbP\left(
 \sup_{|\mathbf{z}|\le \kappa}\Norm{\phi_n(\mathbf{z})}_{C^2} \le K
 \Bigg| \: |F-y_*|\le r_0\right)=1,
\end{align}
where
\[
\Norm{\phi_n(\mathbf{z})}_{C^2}:=
|\phi_n(\mathbf{z})|+\Norm{\partial_{\mathbf{z}}\phi_n(\mathbf{z})}+\Norm{\partial_{\mathbf{z}}^2\phi_n(\mathbf{z})}.
\]
\end{enumerate}
Then $f(y_*)>0$.
\end{theorem}

\section{Regularity of densities}
\label{S:Regularity}
In this section, we will prove Theorems \ref{T:Single} and \ref{T:Mult}.

\subsection{Nonnegative moments (Proof of Theorem \ref{T:NegMom})}
\label{S:NegMom}

In this section, we will prove Theorem \ref{T:NegMom}. 
We will need the following lemma.

\begin{lemma}[Lemma 3.4 in \cite{CHN16Density}]\label{L:1Init}
For any $a,b\in\R$, $\gamma\ge 0$ and $T>0$ such that $b-a>\gamma T$, it holds that
\[
0<\inf_{0\le t+s\le T} \inf_{a-\gamma(t+s)\le x\le b+\gamma (t+s)}(\one_{[a-\gamma s,b+\gamma s]}*G_1(t,\cdot))(x)
  \le 1,
\] 
where $G_1(t,x)$ is the heat kernel function on $\R$.
\end{lemma}
\bigskip

In the following proof, we will use the notation: For any $a,b\in\R^d$ and $\alpha,\beta\in\R$, 
\[
[a+\alpha,b+\beta] := [a_1+\alpha,b_1+\beta]\times\cdots\times[a_d+\alpha,b_d+\beta].
\]

\begin{proof}[Proof of Theorem \ref{T:NegMom}]
Recall that $\calF_t$ is the natural filtration generated by the noise $\W$. 
Fix an arbitrary compact set $K\subset\R^d$ and let $T>0$. We are going to prove Theorem \ref{T:NegMom} for $\inf_{x\in K}u(T,x)$ in two steps.

{\bigskip\noindent\bf Case I.} In this case, we assume that
\begin{itemize}
 \item[(H)] For some cube $[a,b]=[a_1,b_1]\times\cdots\times[a_d,b_d]$ and some nonnegative function $g$ the initial measure $\mu$ satisfies that $\one_{[a,b]}(x)\mu(\ud x) =g(x)\ud x$. Moreover, for some $c>0$, $g(x)\ge c \one_{[a,b]}(x)$ for all $x\in\R^d$.
\end{itemize}

Thanks to the weak comparison principle (see \eqref{E: W comp path}), we may assume that $g(x)=\one_{[a,b]}(x)$. 
For any $t>0$, we denote 
\begin{equation}
\calI_t = [a-\gamma t, b+ \gamma t]\,.
\end{equation}
Choose and fix $\gamma$ such that $K \subseteq \calI_T$. Take 
\begin{equation}
\begin{aligned}
\beta = & \frac{1}{2} \inf_{0 \leq t + s \leq T} \inf_{x \in \calI_{t+s}} \left( \one_{[a-\gamma s, b + \gamma s]} * G(t,\cdot) \right)(x)\\
 = & \frac{1}{2} 
 \prod_{i=1}^d \inf_{0 \leq t + s \leq T} \inf_{x_i \in [a_i-\gamma(t+s),b_i+\gamma(t+s)]} \left( \one_{[a_i-\gamma s, b_i + \gamma s]} * G_1(t,\cdot) \right)(x_i)\,.
\end{aligned}
\end{equation}
Thanks to Lemma \ref{L:1Init}, we have $0 < \beta < 1/2$. Define 
\begin{equation}
T_0:=0,\quad\text{and}\quad 
T_k := \inf \left\{t > T_{k-1}: \inf _{x \in \calI_t} u(t,x ) \leq \beta^k \right\}\,, \quad k\ge 1\,.
\end{equation}
Let $\dot{W}_k(t, x) = \dot{W}(t+ T_{k-1}, x)$, and $H_k(t, x) = H(t+T_{k-1}, x)$. 
For each $k$, let $u_k(t,x)$ be the unique solution to \eqref{E:SPDE-Var} with initial data $u_k(0,x)=\beta^{k-1}\one_{\{[a-\gamma T_{k-1},b+\gamma T_{k-1}]\}}(x)$. 
So we see that the random field $w_k(t,x) = \beta^{1-k} u_k(t,x)$ solves the equation  
\begin{equation}
\begin{cases}
\displaystyle \left(\frac{\partial }{\partial t} -\frac{1}{2}\Delta\right) w_k(t,x)= H_k(t,x) \sigma_k(t,x, w_k(t,x))\dot{W}_k(t,x),& t>0,\: x\in\R^d, \\
w_k(0,x) = I_{[a-\gamma T_{k-1}, b+ \gamma T_{k-1}]}(x),
\end{cases}
\end{equation}
where 
\begin{equation}
\sigma_k(t,x,z) = \beta^{1-k} \sigma(t,x, \beta^{k-1}z)\,.
\end{equation}
Set $\tau_n := \frac{2T}{n}$ and 
\begin{equation}
S(t_1, t_2) = \Big\{ (t,x): t \in [0, t_2-t_1], x \in [a-\gamma(t_1+t_2), b+ \gamma(t_1+t_2)]\Big\}\,.
\end{equation}
Define the following events
\begin{equation}
\mathfrak{D}_{k,n} := \left\{ T_k - T_{k-1}\leq \tau_n\right\}\,,\quad\text{for $1\le k\le n$.}
\end{equation}
Then following exactly the same arguments as those in the proof of Theorem 1.5 in \cite{CHN16Density} we see that
for $k\le n$,
\begin{multline}
\qquad
 \bbP\left(\mathfrak{D}_{k,n}\bigcap\left\{\sup_{(t,x)\in S(T_{k-1},T_k)} \left|w_k(t,x)-w_k(0,x)\right|\ge 1-\beta \right\} \middle| \calF_{T_{k-1}}\right)\\
 \le
 \beta^{-p}\:
 \E\left[\sup_{(t,x)\in [0,\tau_n]\times \calI_T } \left|I_k(t,x)\right|^p\middle| \calF_{T_{k-1}}\right]
 \label{E_:bE}.\qquad
\end{multline}

Next we need to find a deterministic upper bound for the conditional probability in \eqref{E_:bE}.
\begin{align*}
\E \left[ \left|I_k(s,x) - I_k(s', x)\right|^p\right]\leq C |s-s'|^{\frac{\alpha p}{2}} \sup_{(t,y)\in [0, \tau_n]\times \RR^d} \|w_k(t,y)\|^p_p\,,
\end{align*}
for all $(s,s',x)\in [0,\tau_n]^2\times\calI_T$.
By Theorem \ref{T:Mom}, we see that for some constant $Q>0$, 
\begin{equation}
\sup_{(t,y)\in[0,\tau_n]\times\R^d}
\Norm{w_k(t,y)}^p_p \leq Q^p \exp \left( Q p^{\frac{1+\alpha}{\alpha}}t \right)=:C_{p,\tau_n}\,.
\end{equation}
Choose $p$ large enough such that $1-\frac{2}{p}\left(\frac{2}{\alpha}-1\right)>0$.
Then by the Kolmogorov continuity theorem,
for some constant $C>0$ and for all $0 < \eta < 1-\frac{2}{p}\left(\frac{2}{\alpha}-1\right)$, we have that
\begin{align}
\E\left[
\sup_{(t,x)\in [0,\tau_n]\times \calI_T }\left|\frac{I_k(t,x)}{\tau_n^{\alpha\eta/2}}\right|^p\right]&\le
\E\left[
\sup_{(s,s',x')\in [0,\tau_n]^2\times \calI_T}\left|\frac{I_k(s,x)-I_k(s',x)}{|s-s'|^{\alpha\eta/2}}\right|^p\right]
\le C^p C_{p,\tau_n}\,,
\label{E_:Holder}
\end{align}
which implies that 
for some constant $Q'>0$, 
\begin{align*}
\beta^{-p} \: \E\left[ \sup_{(s,x)\in [0,\tau_n]\times \calI_T} |I_k(s,x)|^p\right]
\leq  \tau_n^{\alpha \eta p/2} \exp \left( Q' p^{\frac{1+\alpha}{\alpha}} \tau_n\right) = 
\exp \left( Q' p^{\frac{1+\alpha}{\alpha}}\tau_n +\frac{1}{2}\alpha\eta\log(\tau_n) p\right)\,.
\end{align*}
Take $\eta=\theta\left(1-\frac{2}{p}\left[\frac{2}{\alpha}-1\right]\right)$ with some $\theta\in(0,1)$.
Then the above exponent becomes
\[
f(p):= Q' p^{\frac{1+\alpha}{\alpha}}\tau_n +\frac{1}{2}\theta\left(\alpha-\frac{2}{p}\left[2-\alpha\right]\right)\log(\tau_n) p.
\]
Optimizing in $p$ shows that $f(p)$ achieves its global minimum at
\[
p'=\left(\frac{\alpha ^2 \theta  n\log\left(\frac{n}{2T}\right)}{4(\alpha +1) Q' T}\right)^{\alpha}
\]
and  for some constant $Q''>0$,
\begin{equation}
\min_{p\ge 2} f(p) \le  -Q'' n^{\alpha} (\log n)^{1+\alpha}\,.
\end{equation}
Thus, for some finite constant $C>0$, 
\begin{equation}
P \left ( \mathfrak{D}_{k,n} \big| \mathcal{F}_{T_{k-1}} \right) \leq C \exp \left( -C n^{\alpha} (\log n)^{1+\alpha}\right)\,.
\end{equation}
Then following exactly the same arguments as those leading to (3.8) in \cite{CHN16Density}, we see that
for some constant $B=B(\Lambda,T)>0$ such that for $\epsilon>0$ small enough,
\begin{align}\label{E:CaseI}
\bbP\left(\inf_{x\in K} u(T,x) <\epsilon \right)\le \exp\left(-B\left\{|\log(\epsilon)| \log\left(|\log(\epsilon)|\right) \right\}^{1+\alpha}\right).
\end{align}

{\bigskip\bf\noindent Case II.~}
Now we consider the general initial data.
Set $Q=(a,b)^d$ for some arbitrary constants $a<b$. Choose and fix an arbitrary $\theta\in (0,T\wedge 1)$. Set 
\[
\Theta(\omega):=1\wedge \inf_{x\in Q}  u(\theta,x,\omega).
\]
Since $u(\theta,x)>0$ for all $x\in\R$ a.s. (see Theorem \ref{T:Comp}) and  $u(\theta,x,\omega)$ is continuous in $x$, we see that 
$\Theta>0$ a.s.
Hence, 
$u(\theta,x,\omega)\ge \Theta(\omega) \one_{Q}(x)$ for all $x\in\R$.
Denote $V(t,x,\omega):=\Theta(\omega)^{-1} u(t+\theta,x,\omega)$. By the Markov property, $V(t,x)$ solves the following time-shifted SPDE 
\begin{align}
 \begin{cases}
\left(\displaystyle\frac{\partial}{\partial t} - \frac{1}{2}\Delta \right) V(t,x) =
\widetilde{H}(t,x) \widetilde{\sigma}(t,x,u(t,x))  \dot{W}_\theta(t,x),& t>0\;,\: x\in\R^d,\\[1em]
V(0,x) = \Theta^{-1} u(\theta,x),
 \end{cases}
\end{align}
where $\W_\theta (t,x)=\W(t+\theta,x)$, $\widetilde{H}(t,x) = H(t+\theta,x)$ and $\widetilde{\sigma}(t,x,z,\omega)=\Theta(\omega)^{-1}\sigma\left(t+\theta,x,\Theta(\omega) z\right)$.
Notice that condition \eqref{E:Lambda} is satisfied by $\widetilde{H}$ and $\widetilde{\sigma}$ with the same constant $\Lambda$, that is, 
\[
\left|\widetilde{H}(t,x,\omega)\widetilde{\sigma}(t,x,z,\omega)\right|\le \Lambda |z| \qquad\text{for all $(t,x,z,\omega)\in \R_+\times\R^2\times\Omega$.}
\]
The initial data $V(0,x)$ satisfies assumption (H) in Case I. 
Hence, we can conclude from \eqref{E:CaseI} that for $\epsilon>0$ small enough,
\begin{align*}
\bbP\left(\inf_{x\in K} u(T+\theta,x) <\epsilon \right) & \le
\bbP\left(\inf_{x\in K} V(T,x) < \Theta^{-1}\epsilon 
\right)\\
&\le \int_0^1 1\wedge \exp\left(-B\left\{|\log(\zeta^{-1}\epsilon)| 
\log\left(|\log(\zeta^{-1}\epsilon)|\right) \right\}^{1+\alpha}\right) 
\mu_\Theta(\ud \zeta),
\end{align*} 
where $\mu_\Theta$ denotes the law of the random variable $\Theta$, 
which is supported over $[0,1]$.
For any $\delta \in (0,1)$,
\begin{align*}
 &\frac{\displaystyle \bbP\left(\inf_{x\in K} u(T+\theta,x) <\epsilon 
\right)}{\displaystyle \exp\left(-B\left\{|\log(\epsilon)|\cdot 
|\log\left(|\log(\epsilon)|\right)|^\delta \right\}^{1+\alpha}\right)}
\\
\le & \int_0^1 
1\wedge \exp\Bigg(-B\left\{|\log(\zeta^{-1}\epsilon)| 
\log\left(|\log(\zeta^{-1}\epsilon)|\right) 
\right\}^{1+\alpha}\\
& \hspace{6em}+B 
\left\{|\log(\epsilon)| 
\cdot |\log\left(|\log(\epsilon)|\right)|^\delta \right\}^{1+\alpha}\Bigg)
\mu_\Theta(\ud 
\zeta).
\end{align*}
The integrand of the right hand side of the 
above inequality is bounded by one and goes to zero as $\epsilon$ goes to zero 
because $\delta\in (0,1)$. Hence, the dominated convergence theorem shows that 
\[
 \lim_{\epsilon\rightarrow 0_+}\frac{\bbP\left(\inf_{x\in K} u(T+\theta,x) 
<\epsilon 
\right)}{\exp\left(-B\left\{|\log(\epsilon)| 
\cdot |\log\left(|\log(\epsilon)|\right)|^\delta \right\}^{1+\alpha}\right)}
=0,
\]
which implies that 
\[
\bbP\left(\inf_{x\in K} u(T+\theta,x) 
<\epsilon 
\right)\le \exp\left(-B'\left\{|\log(\epsilon)|\cdot [ 
\log\left(|\log(\epsilon)|\right)]^\delta \right\}^{1+\alpha}\right)
\]
for $\epsilon$ small enough. 
Then using the fact that $(t,x)\mapsto u(t,x)$ is continuous a.s., by 
letting $\theta$ go to zero, we can conclude that \eqref{E:Rate} holds for 
$\epsilon>0$ small enough.

Finally, the existence of the negative moments is a direct consequence of \eqref{E:Rate}.
This completes the proof of Theorem \ref{T:NegMom}.
\end{proof}

\subsection{Malliavin derivatives of \texorpdfstring{$u(t,x)$}{}}
\label{S:MalliavinD}

In this section, we study the Malliavin derivatives of $u(t,x)$. 
Due to the difficulties caused by the initial data, we need to show that $u(t,x)$ lives in some Sobolev spaces restricted to some measurable sets $\calT\times\calS \subset[0,\infty)\times\R^d$.

\begin{proposition}\label{P:D1}
Suppose that $\rho$ is a $C^1$ function with bounded Lipschitz
continuous derivative. Suppose that the initial data $\mu$ satisfies condition \eqref{E:J0finite}. Then
\begin{enumerate}
\item[(1)] For any
$(t,x)\in [0,T]\times \R^d$,
$u(t,x)$ belongs to $\D^{1,p}$ for all $p\ge 1$.
\item[(2)] The Malliavin derivative $Du(t,x)$ defines an $L^2([0,T];\calH)$-valued
process that satisfies the following linear stochastic differential equation
\begin{equation}\label{E:SPDEM}
\begin{aligned}
D_{\theta,\xi}u(t,x) =&
\quad\rho(u(\theta,\xi)) G(t-\theta,x-\xi)\\
&+
\int_\theta^t \int_{\R^d} G(t-s,x-y)\rho'(u(s,y))D_{\theta,\xi}u(s,y) W(\ud s, \ud
y),
\end{aligned}
\end{equation}
for all $(\theta,\xi)\in[0,T]\times \R^d$.
\item[(3)] If $\rho\in C^\infty(\R^d)$ and it has bounded derivatives of all orders,
and if for some measurable sets $\calT \subset [0,t]$ and $\calS \subset \RR^d$, the initial data satisfies the following condition
\begin{equation}
\sup_{(s, y) \in \calT \times \calS}   J_0^2(s,y) < \infty\,,
\end{equation}
then $u(t,x)\in \D_{\calT,\calS}^\infty$.
\end{enumerate}
\end{proposition}
Throughout this section, let $\calH_T$ denote the space $L^2([0,T];\calH)$, that is, 
\[
\InPrd{h,g}_{\calH_T} := 
\int_0^T \langle h(s, \cdot), g(s, \cdot)\rangle_{\mathcal{H}} \ud s, \quad\text{for $h,g\in\calH_T$.} 
\]
Recall the space $\calH_{\calT,\calS}$ defined by the norm in \eqref{E:HTS-norm}.

\begin{proof}[Proof of part (1) of Proposition \ref{P:D1}]
Fix $p\ge 2$ and $T>0$.
Consider the Picard approximations $u_m(t,x)$ in the proof of the existence of the random field solution in \cite{CH16Comparison}, that is, $u_0(t,x)=J_0(t,x)$, and for $m\ge 1$,
\[
u_m(t,x)=J_0(t,x) + \int_0^t\int_{\R^d}G(t-s,x-y)\rho(u_{m-1}(s,y))W(\ud s\ud y).
\]
It is proved in \cite{CH16Comparison} that $u_m(t,x)$ converges to $u(t,x)$ in $L^p(\Omega)$ as $m\rightarrow\infty$ and 
\begin{align}\label{E:MomentUm}
\sup_{m\in\bbN} \Norm{u_m(t,x)}_p\le C_1 ((1+|\mu|)*G(t,\cdot))(x)\quad\text{for all $t\in (0,T]$ and $x\in\R^d$.}
\end{align}
Now we claim that for some constants $C_2>0$ and $\gamma>0$, it holds that 
\begin{align}\label{E:IndDum}
\sup_{m\in\bbN}\E\left(\Norm{D u_m(t,x)}_{\calH_T}^p\right)< 
\left[C_2((1+|\mu|)*G(t,\cdot))(x) H(t;\gamma)^{1/2}\right]^{p}
\end{align}
for all $t>0$ and $x\in\R^d$.
It is clear that $0\equiv Du_0(t,x)$ satisfies \eqref{E:IndDum}.
Assume that $Du_k(t,x)$ satisfies \eqref{E:IndDum} for all $k<m$. Now we shall show that $Du_m(t,x)$ satisfies \eqref{E:IndDum} as well.
Notice that
\begin{align}\notag
D_{\theta,\xi}u_m(t,x)
= &  \quad G(t-\theta,x-\xi)\rho\left(u_{m-1}(\theta,\xi)\right) \\
\label{E:DunSPDE}
& +
\int_\theta^t\int_{\R^d} G(t-s,x-y)\rho'\left(u_{m-1}(s,y)\right)
D_{\theta,\xi}u_{m-1}(s,y)W(\ud s\ud y)\\
\notag
=:& A_m(\theta,\xi) + B_m(\theta,\xi).
\end{align}
Then by Minkowski's inequality and \eqref{E:MomentUm}, we see that 
\begin{align*}
\E\left(\Norm{A_m}_{\calH_T}^p\right)
= &\E\Bigg(\bigg[
\int_0^t\ud \theta\iint_{\R^{2d}}
\ud \xi\ud \xi'\:
G(t-\theta,x-\xi)\rho(u_{m-1}(\theta,\xi)) f(\xi-\xi')\\
&\hspace{10em}\times G(t-\theta,x-\xi')\rho(u_{m-1}(\theta,\xi'))
\bigg]^{p/2}\Bigg)\\
&\le
C \Bigg(
\int_0^t\ud \theta\iint_{\R^{2d}}
\ud \xi\ud \xi'\:
G(t-\theta,x-\xi)\left(1+\Norm{u_{m-1}(\theta,\xi)}_p\right) f(\xi-\xi')\\
&\hspace{10em}\times G(t-\theta,x-\xi')
\left(1+\Norm{u_{m-1}(\theta,\xi')}_p\right) \Bigg)^{p/2}\\
&\le
C' \Bigg(
\int_0^t\ud \theta\iint_{\R^{2d}}
\ud \xi\ud \xi'\:
G(t-\theta,x-\xi)\left(1+(|\mu|*G(\theta,*))(\xi)\right) f(\xi-\xi')\\
&\hspace{10em}\times G(t-\theta,x-\xi')
\left(1+(|\mu|*G(\theta,*))(\xi')\right) \Bigg)^{p/2}
\end{align*}
where the constant $C'$ does not depend on $m$. 
Therefore, by Lemma \ref{L:IntIneq2} with $n=1$, we see that for some $C_3>0$ which depends on the Lipschitz constant of $\rho$ and $C_1$ such that
\[
\sup_{m\in\bbN}\E\left(\Norm{A_m}_{\calH_T}^p\right)< 
\left[C_3((1+|\mu|)*G(t,\cdot))(x)\right]^{p}
\quad\text{for all $t>0$ and $x\in\R^d$.}
\]
%

As for $B_m$, by the Burkholder-Davis-Gundy inequality and by the boundedness of $\rho'$,
\begin{align*}
\E\left(\Norm{B_m}_{\calH_T}^p\right) 
\le & C\: \E\Bigg(
\bigg[
\int_0^t\ud s \iint_{\R^{2d}} \ud y\ud y' \:
G(t-s,x-y) \Norm{Du_{m-1}(s,y)}_{\calH_T} f(y-y')\\
& \hspace{10em}\times
G(t-s,x-y') \Norm{Du_{m-1}(s,y')}_{\calH_T}
\bigg]^{p/2}
\Bigg) \\
\le & C \:
\Bigg(
\int_0^t\ud s \iint_{\R^{2d}} \ud y\ud y' \:
G(t-s,x-y) \Norm{\Norm{Du_{m-1}(s,y)}_{\calH_T}}_p f(y-y')\\
& \hspace{10em}\times
G(t-s,x-y') \Norm{\Norm{Du_{m-1}(s,y')}_{\calH_T}}_p
\Bigg)^{p/2}.
\end{align*}
Then by the induction assumption and Lemma \ref{L:iterate},
\[
\E\left(\Norm{B_m}_{\calH_T}^p\right) \le 
\left[C \gamma^{-1/2} ((1+|\mu|)*G(t,\cdot))(x) H(t;\gamma)^{1/2}\right]^{p}.
\]
Since the function $\gamma\mapsto H(t;\gamma)$ is nondecreasing, by increasing the value of $\gamma$, one can make sure that the above mapping is a contraction. 
Therefore, \eqref{E:IndDum} is true. Finally, by Lemma \ref{L:Dinfty}, we can conclude that $u(t,x)\in \D^{1,p}$.
This proves part (1) of Proposition \ref{P:D1}.
\end{proof}

\bigskip

\begin{proof}[Proof of part (2) of Proposition \ref{P:D1}]
The proof is similar to the proof of part (2) of Proposition 5.1 in \cite{CHN16Density}.
Fix $t\in (0,T]$ and $x\in\R^d$.
By Lemma 1.2.3 of \cite{Nualart09}, $D_{\theta,\xi}u_m(t,x)$ converges to $D_{\theta,\xi}u(t,x)$ in the weak topology of $L^2(\Omega;\calH_T)$, that is, for any $h\in\calH_T$ and any square integrable random variable $F\in\calF_t$, 
\[
\lim_{n\rightarrow\infty}
\E\left(\InPrd{D_{\theta,\xi}u_m(t,x)-D_{\theta,\xi}u(t,x),h}_{\calH_T}F\right)=0.
\]
Hence, we need to prove that the right-hand of \eqref{E:DunSPDE} converges to the right-hand side of \eqref{E:SPDEM} in this weak topology of $L^2(\Omega;\calH_T)$.
By the Cauchy-Schwartz inequality, 
\begin{multline}
\label{E_:First}
\E\left(\left|\InPrd{\left(\rho(u)-\rho(u_m)\right)G(t-\cdot,x-\cdot),h}_{\calH_T}F\right|\right)
\le \LIP_\rho\Norm{h}_{\calH_T}\Norm{F}_2\\
\times \bigg(
\int_0^t\iint_{\R^2} G(t-s,x-y)\Norm{u(s,y)-u_m(s,y)}_2\\
\times G(t-s,x-y')\Norm{u(s,y')-u_m(s,y')}_2 f(y-y')\ud s\ud y\ud y'
\bigg)^{1/2}.
\end{multline}
Denote 
\[
F_m(t,x):= 
\begin{cases}
 \Norm{u(t,x)}_2 &\text{if $m=-1$,}\\[0.5em]
\Norm{u(t,x)-u_m(t,x)}_2 &\text{if $m\ge 0$.}
\end{cases}
\]
For some positive constant $\lambda$ large enough, we have that 
\begin{align*}
F_m^2(t,x) &\le \lambda 
\int_0^t \ud s\iint_{\R^{2d}}\ud y\ud y'\:
G(t-s,x-y)G(t-s,x-y')f(y-y') F_{m-1}(s,y)F_{m-1}(s,y')
\end{align*}
for all $m\ge 0$.
We claim that 
\begin{align}\label{E_:Fm}
F_m(t,x) \le C (|\mu|*G(t,\cdot))(x) \left[\sum_{i=m+1}^\infty \lambda^i h_i(t)\right]^{1/2} \quad\text{for all $m\ge -1$.}
\end{align}
It is true for $m=-1$ by Theorem \ref{T:Mom}.
Suppose that it is true for $m>-1$. We now prove that it holds for $m+1$. By the induction assumption,
\begin{align*}
F_{m+1}^2(t,x)\le C^2 \lambda \int_0^t\ud s 
\left(\sum_{i=m+1}^\infty \lambda^i h_i(s)\right)\iint_{\R^{2d}}\ud y\ud y' \: 
&G(t-s,x-y)G(t-s,x-y')f(y-y') \\
&\times (|\mu|*G(s,\cdot))(y)\: (|\mu|*G(s,\cdot))(y').
\end{align*}
Then by the same arguments as the proof of Lemma 2.2 of \cite{CH16Comparison}, we see that 
\begin{align*}
F_{m+1}(t,x) &\le C(|\mu|*G(t,\cdot))(x)
\left(\lambda \int_0^t \left[\sum_{i=m+1}^\infty \lambda^i h_i(s)\right] k(t-s) \ud s\right)^{1/2}\\
&=C(|\mu|*G(t,\cdot))(x)
\left( \sum_{i=m+1}^\infty \lambda^{i+1} h_{i+1}(t)\right)^{1/2}.
\end{align*}
Hence, it holds for $m+1$. This proves \eqref{E_:Fm}.
By the same argument as above, we see that
\begin{align}\label{E:Fm->0}
F_m^2(t,x)\le  
C\left[(|\mu|*G(t,\cdot))(x)\right]^2
 \sum_{i=m+2}^\infty \lambda^{i} h_{i}(t)\rightarrow 0 \quad \text{as $m\rightarrow\infty$,}
\end{align}
because $\sum_{i=0}^\infty \gamma^{i} h_{i}(s)=H(t;\gamma)<\infty$ for any $\gamma>0$.
Therefore, 
\[
\lim_{m\rightarrow\infty}
\E\left(\InPrd{\left(\rho(u)-\rho(u_m)\right)G(t-\cdot,x-\cdot),h}_{\calH_T}F\right)=0.
\]

Now denote the second term on the right-hand side of \eqref{E:SPDEM} by $B(\theta,\xi)$. 
Recall that $B_m(\theta,\xi)$ is defined in \eqref{E:DunSPDE}.
It remains to prove that
\begin{align}
\lim_{m\rightarrow\infty}\E\left(\InPrd{B-B_m,h}_{\calH_T} F\right)=0.
\end{align}
Notice that
\begin{align*}
B(\theta,\xi)&-B_m(\theta,\xi)\\
=&\int_0^t\int_{\R^d} G(t-s,x-y)\left(\rho'(u(s,y))-\rho'(u_{m-1}(s,y))\right)
D_{\theta,\xi}u_{m-1}(s,y)W(\ud s\ud y)\\
&+\int_0^t\int_{\R^d} G(t-s,x-y)\rho'(u(s,y))\left(D_{\theta,\xi}u(s,y)-D_{\theta,\xi}u_{m-1}(s,y)\right)
W(\ud s\ud y)\\
=:& B_{m,1}(\theta,\xi) + B_{m,2}(\theta,\xi).
\end{align*}

Because $F$ is square integrable, for some adapted random field $\{\Phi(s,y),s\in[0,T],y\in\R \}$ with 
\begin{align}\label{E:A-Norm}
\int_0^t\iint_{\R^{2d}} \E\left[\Phi(s,y)\Phi(s,y')\right]f(y-y')\ud s\ud y\ud y'<\infty,
\end{align}
it holds that 
\[
F= \E[F] + \int_0^t\int_{\R^d} \Phi(s,y)W(\ud s\ud y).
\]
Hence,
\begin{align*}
\E\left(\InPrd{B_{m,1}(t,x),h}_{\calH_T}F\right)
=&\E \Bigg[\int_0^t\ud s \iint_{\R^d}\ud y\ud y'\: 
G(t-s,x-y) f(y-y')\\
&\times
\Phi(s,y')\left(
\rho'(u(s,y))-\rho'(u_{m-1}(s,y))
\right)
\InPrd{Du_{m-1}(s,y),h}_{\calH_T}\Bigg]\\
\le&\left(\int_0^t\iint_{\R^{2d}} \E\left[\Phi(s,y)\Phi(s,y')\right]f(y-y')\ud s\ud y\ud y'\right)^{1/2}\\
&\times \Bigg[
\int_0^t\ud s\iint_{\R^{2d}}\ud y\ud y'\: 
G(t-s,x-y)f(y-y')G(t-s,x-y')\\
&\hspace{1em}\times\E\bigg(
\left[\rho'(u(s,y))-\rho'(u_{m-1}(s,y))\right]
\left[\rho'(u(s,y'))-\rho'(u_{m-1}(s,y'))\right]\\
&\hspace{5em} \times \InPrd{Du_{m-1}(s,y),h}_{\calH_T} \InPrd{Du_{m-1}(s,y'),h}_{\calH_T}
\bigg)\Bigg]^{1/2}\\
&=: I_1^{1/2} I_{2,m}^{1/2}.
\end{align*}
where we have applied the Cauchy-Schwartz inequality.
It is clear that $I_1<\infty$. 
It remains to prove that $\lim_{m\rightarrow\infty}I_{2,m}=0$. On the one hand, since $\rho'$ is bounded,
\begin{align*}
I_{2,m}\le C \int_0^t\ud s\iint_{\R^{2d}}\ud y\ud y'\:
& G(t-s,x-y) f(y-y') G(t-s,x-y')\\
\times & \Norm{\InPrd{Du_{m-1}(s,y),h}_{\calH_T}}_2
\Norm{\InPrd{Du_{m-1}(s,y'),h}_{\calH_T}}_2.
\end{align*}
Then by \eqref{E:IndDum} and Lemma \ref{L:iterate}, we find an integrable upper bound of the integrand that does not depend on $m$. On the other hand, the expectation in $I_{2,m}$ is bounded by 
\begin{align*}
\LIP_{\rho'}^2 
&\Norm{u(s,y)-u_{m-1}(s,y)}_4
\Norm{u(s,y')-u_{m-1}(s,y')}_4\\
&\times \Norm{\InPrd{Du_{m-1}(s,y),h}_{\calH_T}}_4 
\Norm{\InPrd{Du_{m-1}(s,y),h}_{\calH_T}}_4.
\end{align*}
By the same arguments as those leading to \eqref{E:Fm->0}, and by the uniform bound in \eqref{E:IndDum}, we see that the above quantity converges to zero as $m\rightarrow \infty$. 
Then an application of the dominated convergence theorem completes the proof of part (2) of Proposition \ref{P:D1}.
\end{proof}

\bigskip

\begin{proof}[Proof of part (3) of Proposition \ref{P:D1}]
Fix $S \subset \RR^d$ and $\calT\subset[0,T]$. Recall that $\calH_{\calT,\calS}$ is defined in \eqref{E:HTS}. We will prove the following property by induction:
\begin{equation}\label{E:DnInd}
\sup_{m \in \mathbb{N}}\sup_{(t,x) \in [0, T]\times \RR^d} \left\| \left\| D^n u_m(t,x)\right\|_{\mathcal{H}_{\calT,\calS}^{\otimes n}} \right\|_p < \infty \quad \text{for all} \ n \geq 1 \  \text{and all}\ p \geq 2\,.
\end{equation}

{\noindent\bf Step 1.} Consider the case $n=1$. We will prove by induction that for some finite constant $\Theta_T>0$ independent of $m$,
\begin{equation}\label{E: MallD1m}
\sup_{(t,x) \in [0, T]\times \RR^d} \left\| \left\| D u_m(t,x)\right\|_{\mathcal{H}_{\calT,\calS}} \right\|_p <\Theta_T\,.
\end{equation}
Since $u_0(t,x)$ is deterministic, \eqref{E: MallD1m} is true for $m=0$. Now suppose \eqref{E: MallD1m} holds for all $k \leq m$ and we will consider $k=m+1$. Notice that
\begin{align*}
D_{\theta_1,\xi_1}u_{m+1}(t,x) = & \rho(u_m(\theta_1, \xi_1))G(t-\theta_1, x-\xi_1) \\
& + \int_0^t \int_{\RR^d} G(t-s, x-y) \rho'(u_m(s,y))D_{\theta_1,\xi_1}u_m(s,y)W(\ud s\ud y)\,.
\end{align*}
Thus, 
\begin{align*}
\left\| \|Du_{m+1}(t,x)\|_{\mathcal{H}_{\calT,\calS}}\right\|_p^2 \leq &\quad 2 \left\|\|\rho(u_m(\cdot,*))G(t-\cdot, x-*)\|_{\mathcal{H}_{\calT,\calS}} \right\|^2_p\\
& + 2 \left\| \left\| \int_0^t \int_{\RR^d} G(t-s, x-y)\rho'(u_m(s,y))D u_m(s,y) W(\ud s\ud y) \right\|_{\calH_{\calT,\calS}}\right\|_p^2\\
=:&\: 2 I_1 + 2 I_2\,.
\end{align*}
Using Minkowski's inequality, we obtain 
\begin{align*}
I_1 &= \Bigg[\E\bigg\{ \bigg(\int_{\calT} \iint_{S^2}  G(t-\theta_1, x-\xi_1)G(t-\theta_1, x-\xi_1')f(\xi_1-\xi_1')\\
& \hspace{10em} \times \rho(u_m(\theta_1, \xi_1)) \rho(u_m(\theta_1, \xi_1'))\ud\xi_1 \ud\xi_1' \ud\theta_1\bigg)^{\frac{p}{2}}\bigg\} \Bigg]^{\frac{2}{p}}\\
&\leq \int_{\calT} \iint_{S^2} \left\|\rho(u_m(\theta_1, \xi_1))\rho(u_m(\theta_1, \xi_1')) \right\|_{\frac{p}{2}} G(t-\theta_1, x-\xi_1)G(t-\theta_1, x-\xi_1') f(\xi_1-\xi_1') \ud\xi_1 \ud\xi_1' \ud\theta_1\\
&\leq C \sup_{m\in\bbN} \sup_{s \in \calT, y \in S} \left(1+ \|u_m(s, y)\|_p^2 \right) \int_0^T \iint_{\RR^{2d}} 
G(s, \xi_1)G(s, \xi_1') f(\xi_1-\xi_1') \ud\xi_1 \ud\xi_1' \ud s\\
& \le C \sup_{s \in \calT, y \in S} \left(1+ (|\mu|*G(s,\cdot))(y) \right)<\infty,
\end{align*}
where we have used the fact that $\Norm{u_m(s,y)}_p\le \Norm{u(s,y)}_p$ and the moment bound \eqref{E:Mom}.
As for $I_2$, using the Burkholder-Davis-Gundy inequality, by the boundedness of $\rho'$, we see that
\begin{align*}
I_2 \leq  \Norm{\rho'}_{L^\infty(\R)}^2 \int_0^t\ud s \iint_{\RR^{2d}}\ud y\ud y'\: &G(t-s, x-y) G(t-s, x-y')f(y-y') \\
 \times &\left\|\|Du_m(s,y)\|_{\mathcal{H}_{\calT,\calS}} \right\|_p\cdot \left\|\|Du_m(s,y')\|_{\mathcal{H}_{\calT,\calS}} \right\|_p.
\end{align*}
Then one can apply Lemma \ref{L:IntIneq2} with constant initial data to conclude that \eqref{E: MallD1m} holds.

{\bigskip \noindent\bf Step 2.}
We have proved \eqref{E:DnInd} for $n=1$ in the previous step. We assume that \eqref{E:DnInd} holds for $n-1$ with $n\ge 1$.  Now we will prove by induction on $m$ that for some finite constant $\Theta_T>0$ independent of $m$,
\begin{equation}\label{E:MallDnm}
\sup_{(t,x) \in [0, T]\times \RR^d} \Norm{\Norm{D^n u_m(t,x)}_{\mathcal{H}_{\calT,\calS}^{\otimes n}}}_p <\Theta_T\,.
\end{equation}
Denote 
\begin{align}\label{E:NoT-Alpha}
\begin{gathered}
 \alpha :=((\theta_1,\xi_1),\dots, (\theta_n,\xi_n))
\quad\text{and}\\
\hat{\alpha}_k := ((\theta_1,\xi_1),\dots, (\theta_{k-1},\xi_{k-1}),(\theta_{k+1},\xi_{k+1}),\dots,  (\theta_n,\xi_n)).
\end{gathered}
\end{align}
Clearly, the case $m=0$ is true since $D^n_\alpha u_0(t,x)\equiv 0$. By Lemma 5.6 in \cite{CHN16Density}, 
\begin{align*}
D^n_\alpha u_{m+1}(t,x) =& \sum_{k=1}^n D_{\hat{\alpha}_k}^{n-1} \rho(u_m(\theta_k, \xi_k))G(t-\theta_k, x-\xi_k)\\
& + \int_0^t \int_{\RR^d} D^n_\alpha \: \rho(u_m(s,y)) G(t-s, x-y) W(\ud s\ud y)\,.
\end{align*}
Thus, 
\begin{align*}
&\left\|D^n u_{m+1}(t,x)\right\|_{\mathcal{H}_{\calT,\calS}^{\otimes n}} \\
&\leq \sum_{k=1}^n \bigg ( \int_{\calT}\iint_{S^2}
\langle D_{\hat{\alpha}_k}^{n-1} \rho\left( u_m(\theta_k, \xi_k) \right), D_{\hat{\alpha}_k}^{n-1} \rho\left( u_m(\theta_k, \xi'_k) \right)\rangle_{\mathcal{H}_{\calT,\calS}^{\otimes (n-1)}} \\
&\hspace{8em}\times G(t-\theta_k, x-\xi_k), G(t-\theta_k, x-\xi_k)f(\xi_k-\xi_k') \ud\xi_k \ud\xi_k' \ud\theta_k\bigg)^{\frac{1}{2}}\\
&\qquad + \left\|\int_0^t \int_{\RR^d} D^n\rho(u_m(s,y)) G(t-s, x-y) W(\ud s\ud y) \right\|_{\mathcal{H}_{\calT,\calS}^{\otimes n}}\\
&=: \left(\sum_{k=1}^n J_{1,k}\right) + J_2\,.
\end{align*}
For $J_{1,k}$, by Minkowski's inequality, we see that
\begin{align*}
\Norm{J_{1,k}}_p^2&\le\int_{\calT}\ud \theta_k\iint_{S^2}\ud \xi_k\ud\xi_k'\: 
\left\| \left\|D_{\hat{\alpha}_k}^{n-1} \rho\left( u_m(\theta_k, \xi_k) \right) \right\|_{\mathcal{H}_{\calT,\calS}^{\otimes (n-1)}} \right\|_{p} \left\| \left\|D_{\hat{\alpha}_k}^{n-1} \rho\left( u_m(\theta_k, \xi_k') \right) \right\|_{\mathcal{H}_{\calT,\calS}^{\otimes (n-1)}} \right\|_{p} \\
&\hspace{8em}\times G(t-\theta_k, x-\xi_k) G(t-\theta_k, x-\xi_k') f(\xi_k-\xi_k') \\
&=: J^*_{1,k}.
\end{align*}
As for $J_2$, by the Burkholder-Davis-Gundy inequality and Minkowski's inequality,
\begin{align*}
\Norm{J_2}_p^2\leq& 4p\: \E\Bigg[ \bigg( \int_0^t \iint_{\RR^{2d}} \InPrd{ D^n\rho\left(u_m(s,y)\right), D^n \rho\left( u_m(s,y) \right)}_{\calH_{\calT,\calS}^{\otimes n}} \\
&\qquad \qquad \qquad \times G(t-s,x-y) G(t-s,x-y')f(y-y') \ud y \ud y' \ud s\bigg)^{\frac{p}{2}}\Bigg]^{\frac{2}{p}}\\
\leq& 4p \int_0^t \iint_{\RR^{2d}} \left\| \left\| D^n\rho\left(u_m(s,y)  \right)\right\|_{\mathcal{H}_{\calT,\calS}^{\otimes n}} \right\|_{p} \left\| \left\|D^n\rho\left(u_m(s,y')\right)  \right\|_{\mathcal{H}_{\calT,\calS}^{\otimes n}} \right\|_{p}\\
&\qquad \qquad\qquad\times  G(t-s, x-y) G(t-s, x-y') f(y-y') \ud y \ud y' \ud s\\
=:& 4p J^*_2\,.
\end{align*}
Therefore, by $(a_1+\dots+a_n)^2\le n(a_1^2+\dots+a_n^2)$, we see that
\[
\Norm{\Norm{D^n u_m(t,x)}_{\mathcal{H}_{\calT,\calS}^{\otimes n}}}_p^2 \le C_n \sum_{k=1}^n J_{1,k}^* + C_n 4p J_2^*.
\]
where $C_n=n+1$.
By Lemma 5.5 of \cite{CHN16Density} and the induction assumption,
we have that 
\[
J_{1,k}^* \le C \sup_{m\in\bbN}\sup_{(s,y)\in [0,T]\times\R}\Norm{\:\Norm{D^{n-1} \rho(u_{m}(s,y))}_{\calH_{\calT,\calS}^{\otimes (n-1)}}}_p^2  <\infty.
\]
Let 
\begin{equation}
\Delta ^n (\rho, u) := D^n \rho(u) - \rho'(u) D^n u\,,
\end{equation}
be all terms in the summation of $D_\alpha^n \rho(u)$ that have Malliavin derivatives of order less than or equal to $n-1$.
Then 
\begin{align*}
J_2^* \le & 2\Norm{\rho'}_{L^\infty(\R)}^2 C_n
\int_0^t \iint_{\RR^{2d}} \left\| \left\| D^n u_m(s,y)\right\|_{\mathcal{H}_{\calT,\calS}^{\otimes n}} \right\|_{p} \left\| \left\|D^n u_m(s,y')\right\|_{\mathcal{H}_{\calT,\calS}^{\otimes n}} \right\|_{p}\\
&\qquad \qquad\qquad\times  G(t-s, x-y) G(t-s, x-y') f(y-y') \ud y \ud y' \ud s\\
& +2 C_n 
\int_0^t \iint_{\RR^{2d}} \Norm{\Norm{
\Delta^n\left(\rho, u_m(s,y)\right)
}_{\mathcal{H}_{\calT,\calS}^{\otimes n}}}_{p} 
\Norm{\Norm{
\Delta^n\left(\rho, u_m(s,y')\right)
}_{\mathcal{H}_{\calT,\calS}^{\otimes n}}}_{p}
\\
&\qquad \qquad\qquad\times  G(t-s, x-y) G(t-s, x-y') f(y-y') \ud y \ud y' \ud s\\
=:& \: 2\Norm{\rho'}_{L^\infty(\R)}^2 C_n J_{2,1}^* + 2 C_n J_{2,2}^*.
\end{align*}
By the induction assumption, we have that 
\[
J_{2,2}^* \le C \sup_{m\in\bbN}\sup_{(s,y)\in [0,T]\times\R}\Norm{\:\Norm{\Delta^{n}\left( \rho, u_{m}(s,y)\right)}_{\calH_{\calT,\calS}^{\otimes (n-1)}}}_p^2  <\infty.
\]
Therefore, for some $C_n'>0$,
\begin{align*}
\left\| \left\|D^n u_m(t,x) \right\|_{\mathcal{H}_{\calT,\calS}^{\otimes n}} \right\|^2_p & \leq C_n' + C_n' \int_0^t \iint_{\RR^{2d}} 
\left\| \left\| D^n u_m(s,y)\right\|_{\mathcal{H}_{\calT,\calS}^{\otimes n}}\right\|_{p}\left\| \left\| D^n u_m(s,y')\right\|_{\mathcal{H}_{\calT,\calS}^{\otimes n}}\right\|_{p}\\
&\qquad \qquad \times G(t-s, x-y) G(t-s, x-y') f(y-y') \ud y \ud y' \ud s.
\end{align*}
Finally, an application of Lemma \ref{L:iterate} with $\mu(\ud x) \equiv \sqrt{C_n'}\ud x$ and $g(t,x) =\left\| \left\|D^n u_m(t,x) \right\|_{\mathcal{H}_{\calT,\calS}^{\otimes n}} \right\|_p$ proves \eqref{E:MallDnm}. Therefore, \eqref{E:DnInd} holds. This completes the proof of part (3) of Proposition \ref{P:D1}.
\end{proof}

\subsection{Density at a single point (Proof of Theorem \ref{S:Single})}
\label{S:Single}

In this section, we will prove Theorem \ref{T:Single}. We need to first prove two lemmas and one special case (Theorem \ref{T_:Single} below).

\begin{lemma}\label{L:J0f}
If for some cube $Q'=(a_1',b_1')\times\dots\times(a_d',b_d') \subset \R^d$, $a_i'<b_i'$, the measure $\mu$ restricted to this $\bar{Q}'=[a_1',b_1']\times\dots\times[a_d',b_d']$ has a density $f(x)$ with $f(x)$ being $\beta$-H\"older continuous
for some $\beta\in (0,1)$, then for any compact set $Q\subset Q'$ and $T>0$, there exists some finite constant $C:=C(Q,Q',T,\beta,\mu)>0$ such that
 \[
\left|\int_{\R^d} G(t,x-y) \mu(\ud y) -f(x)\right| \le C t^{\beta/2} \qquad\text{for all $(t,x)\in (0,T]\times Q$.}
 \]
\end{lemma}
\begin{proof}
Fix $x\in Q$. Notice that
\begin{align*}
\left|\int_{\R^d} G(t,x-y) \mu(\ud y) -f(x)\right| \le &
\quad \int_{Q'} G(t,x-y)\left|f(y)-f(x)\right|\ud y\\
 &+\int_{Q'^c} G(t,x-y) |\mu|(\ud y)+ |f(x)|\int_{Q'^c} G(t,x-y) \ud y\\
 =:& I_1+I_2+ I_3.
\end{align*}
By the H\"older continuity of $f$,
\begin{align*}
I_1 &\le C \int_{Q'} G(t,x-y) |y-x|^\beta \ud y\\
&\le C \int_{\R^d} t^{-d/2}G\left(1,\frac{y}{t^{1/2}}\right) |y|^\beta \ud y\\
&= C t^{\beta/2} \int_{\R^d} G\left(1,z\right) |z|^\beta \ud z= C t^{\beta/2}.
\end{align*}

Denote $\text{dist}(y,Q) =\min\left(|y-z|: z\in Q\right)$ and 
$\text{dist}(Q_1,Q_2) =\min\left(|y-z|: y\in Q_1, z\in Q_2\right)$.
It is clear that 
\[
\text{dist}(y,x) \ge \text{dist}(y,Q) \ge \text{dist}(Q'^c,Q) >0,\quad\text{for all $x\in Q$ and $y\in Q'^c$.}
\]
Hence,
\[
e^{-\frac{|y-x|^2}{2t}} \le  e^{-\frac{\text{dist}(Q'^c,Q)^2}{4t}} e^{-\frac{\text{dist}(y,Q)^2}{4T}},
\quad\text{for all $x\in Q$, $y\in Q'^c$ and $t\in (0,T]$,}
\]
which implies that
\[
I_3\le C t^{-d/2} e^{-\frac{\text{dist}(Q,Q'^c)^2}{4t}} \left(\sup_{x\in \bar{Q}'}|f(x)|\right)
\int_{Q'^c} e^{-\frac{\text{dist}(y,Q)^2}{4T}} \ud y
\le C t.
\]
Similarly,
\[
I_2\le C t^{-d/2} e^{-\frac{\text{dist}(Q,Q'^c)^2}{4t}} 
\int_{Q'^c} e^{-\frac{\text{dist}(y,Q)^2}{4T}} |\mu|(\ud y)
\le C t,
\]
which completes the proof of Lemma \ref{L:J0f}.
\end{proof}


\begin{lemma}\label{L:LimPmom}
Let $u$ be the solution with the initial data $\mu$ that satisfies condition \eqref{E:J0finite}.
Suppose there exists a cube $Q'=[a_1,b_1]\times\dots\times[a_d,b_d]$ with $a_i<b_i$ such that
the measure $\mu$ restricted to $Q'$ has a bounded density $g(x)$.  Then for any $Q\subset Q'$, the following properties hold:
\begin{enumerate}
\item[(1)] For all $T>0$, $\sup_{(t,x)\in [0,T]\times Q} \Norm{u(t,x)}_p<+\infty$;
\item[(2)] If $g$ is $\beta$-H\"older continuous on $Q'$
for some $\beta\in(0,1)$, then for all $x\in Q$,
 \[
 \Norm{u(t,x)-u(s,x)}_p\le C_{T} |t-s|^{\frac{\alpha\wedge \beta}{2}}
 \qquad\text{for all $0\le s\le t\le  T$ and $p\ge 2$}.
 \]
\end{enumerate}
\end{lemma}

\begin{proof}
By Theorem \ref{T:Mom}, for all $x\in Q$ and $p\ge 2$,
\begin{align}
 \limsup_{t\rightarrow 0} \Norm{u(t,x)}_p 
 \le C \limsup_{t\rightarrow 0}  \left(|\mu|*G(t,\cdot)\right)(x) \le C \sup_{y\in Q'} |g(y)|,
\end{align}
which implies part (1).
As for part (2), notice that $u(t,x)$ consists of two parts as in \eqref{E:WalshSI}.
This property for $J_0(t,x)$ is guaranteed by Lemma \ref{L:J0f} 
and that for $I(t,x)$ is proved in step 3 of the proof of Theorem 1.6 in \cite{CH16Comparison}.
\end{proof}

Next we will prove a sufficient condition for the existence and smoothness of density at a single point.

\begin{theorem}\label{T_:Single}
Let $u(t,x)$ be the solution to equation \eqref{E:SHE} starting from an initial measure $\mu$ that satisfies \eqref{E:J0finite}. Assume that $\mu$ is proper at some point $x_0\in \RR^d$ with a density function $g$ over a neighborhood $Q$ of $x_0$. Suppose that $\rho(0, x_0, g(x_0)) \neq 0$ and that $g$ is $\beta$-H\"older continuous on $Q$ for some $\beta \in (0,1)$. We also assume that \eqref{E:Dalang2} holds for some $0 < \alpha \leq 1$.  Then we have the following two statements: 
\begin{itemize}
\item[(a)] If $\rho$ is differentiable in the third argument with bounded Lipschitz continuous derivative, then for all $t>0$ and $x\in \RR^d$, $u(t,x)$ has an absolutely continuous law with respect to the Lebesgue measure. 
\item[(b)]If $\rho$ is infinitely differentiable in the third argument with bounded derivatives, then for all $t>0$ and $x\in \RR^d$, $u(t,x)$ has a smooth density. 
\end{itemize} 
\end{theorem}
\begin{proof}
By Proposition \ref{P:D1}, we know that 
\begin{align}
\label{E:SPDE-M}
\begin{aligned}
D_{\theta, \xi} u(t,x) &= \rho(u(\theta, \xi))G(t-\theta, x-\xi) \\
& \qquad + \int_0^t \int_{\RR^d} G(t-s, x-y) \rho' (u(s,y)) D_{\theta,\xi} u(s,y) W(\ud s\ud y)\,.
\end{aligned}
\end{align}
By part (3) of Proposition \ref{P:D1}, we know that $u(t,x) \in \mathbb{D}_{\calT,\calS}^\infty$. Denote by
\begin{equation}
C(t,x) = \int_0^t \|D_{\theta, \cdot}u(t,x)\|^2_{\mathcal{H}} \ud \theta .
\end{equation}
Then both parts (a) and (b) will be proved once we can show that $\E [C(t,x)^{-p}]< \infty$ for all $p \geq 2$. By the assumption on $g$, we may find a cube $Q$ such that 
$|\rho(0, x , g(x))| \geq \delta >0$ for all $x\in Q$. Let $\psi$ be a smooth function supported in $Q$ such that 
\[
0 \leq \psi(\xi) \leq 1
\quad\text{and}\quad
\langle D_{\theta, \cdot}u(t,x)\,, \psi\rangle_{\mathcal{H}}\leq 1.
\]
Set 
\begin{equation}
Y_{\theta}(t,x) =
\iint_{\R^{2d}} D_{\theta, \xi} u(t,x)\:
f(\xi-\xi') \psi(\xi') \ud \xi \ud \xi'.
\end{equation}
Then, choose $0 < r < 1$ and $\epsilon^r < t$.
By the Cauchy-Schwartz inequality,
\begin{align*}
C(t,x)&\geq \int_0^t  
\langle D_{\theta, \cdot}u(t,x)\,, \psi\rangle_{\mathcal{H}}^2 \ud\theta \\
&=\int_0^t Y^2_{\theta}(t,x) \ud\theta \geq \int_0^{\epsilon^{r}}Y_{\theta}^2(t,x)\ud \theta\\
&\geq \epsilon^{r} Y_0^2(t,x) - \int_0^{\epsilon^r} \left| Y_{\theta}^2(t,x)-Y_{0}^{2}(t,x)\right| \ud \theta\,.
\end{align*}
Hence
\begin{align*}
\bbP(C(t,x)< \epsilon) &\leq \bbP \left( |Y_0(t,x)| < \sqrt{2} \epsilon^{\frac{1-r}{2}} \right) + \bbP \left( \int_0^{\epsilon^r} \left| Y_{\theta}^2(t,x)-Y_0^2(t,x)\right| \ud \theta > \epsilon \right)\\
&=: \bbP(A_1) + \bbP (A_2)\,. 
\end{align*}
We will estimate $\bbP(A_1)$ and $\bbP(A_2)$. 
First, for $\bbP(A_1)$, in both sides of  \eqref{E:SPDEM}, take the inner product with $\psi$ in $\mathcal{H}$, we see that $Y_{\theta}(t,x)$ solves the following integral equation
\begin{align}
\begin{aligned}
 Y_{\theta}(t,x) &= \int_{\RR^d} \psi(\xi)\rho(u(\theta, \xi)) G(t-\theta, x-\xi) \ud\xi\\
& \qquad + \int_{\theta}^t\int_{\RR^d}  G(t-s, x-y) \rho'(u(s,y)) Y_\theta(s,y) W(\ud s\ud y)\,.
\end{aligned}
\end{align}
In particular, $Y_0(t,x)$ solves the following SPDE
\begin{equation}
\begin{cases}
\displaystyle
\left( \frac{\partial }{\partial t} - \frac{1}{2}\Delta\right)Y_0(t,x) = \rho' (u(t,x)) Y_0(t,x) \dot{W}(t,x)\\
Y_0(0,x) = \psi(x) \rho(0, x, u_0(x)).
\end{cases}
\end{equation}
Hence Theorem \ref{T:NegMom} implies that for all $p\geq 1$, $t>0$ and $x\in \RR^d$, 
\begin{equation}
\bbP(A_1)\leq C_{t,x,p} \epsilon^p\,, \quad \text{for $\epsilon$ small enough}\,. 
\end{equation}
As for $\bbP(A_2)$, for all $q\ge 1$, by Minkowski's inequality, we see that
\begin{align*}
\bbP(A_2)&\leq \frac{1}{\epsilon^q} \E \left[\left( \int_0^{\epsilon^r} \left| Y_{\theta}^2(t,x) - Y_0^2(t,x) \right|\ud \theta \right)^q\right]\\
&\leq \frac{1}{\epsilon^q} \left( \int_0^{\epsilon^r} \Norm{Y_{\theta}^2(t,x) - Y_0^2(t,x)}_{q} \ud \theta \right)^q\\
&\leq { \epsilon^{q(r-1)}} \sup_{(\theta, x)\in [0, \epsilon^r]\times \RR^d} \left(\Norm{Y_{\theta}(t,x) - Y_0(t,x)}^q_{2q} \Norm{Y_{\theta}(t,x) + Y_0(t,x)}^q_{2q}\right)\,.
\end{align*}
By the same arguments as the proof of Theorem 6.1 of \cite{CHN16Density}, we have
\begin{equation}
\sup_{(\theta, x) \in [0,t]\times \RR^d} \E\left[ \left|Y_{\theta}(t,x)\right|^{2q}\right] < \infty\,.
\end{equation}
Now we write
\begin{equation}\label{E:Y-Y}
Y_{\theta}(t,x) - Y_{0} (t,x) = \Psi_1 - \Psi_2 + \Psi_3\,,
\end{equation}
where 
\begin{align*}
 \Psi_1 &= \int_{\R^d} \psi(\xi)\left[\rho(u(\theta,\xi)) G(t-\theta,x-\xi) - \rho(u(0,\xi)) G(t,x-\xi)\right]\ud \xi\,,\\
 \Psi_2 &=\int_0^\theta \int_{\R^d} G(t-s,x-y)  \rho'(u(s,y)) Y_0(s,y) W(\ud s\ud y)\,,\\
 \Psi_3 &= \int_\theta^t\int_{\R^d} G(t-s,x-y)\rho'(u(s,y)) \left(Y_\theta(s,y)-Y_0(s,y)\right)W(\ud s\ud y)\,.
\end{align*}
By the Lipschitz continuity of $\rho$, we have that
\begin{align*}
\E\left[|\Psi_1|^{2q}\right] &\leq C\: \E 
\left[\left( \int_{\RR^d} \psi(\xi) |u(\theta, \xi)-u(0, \xi)| G(t-\theta, x-\xi)\ud \xi \right)^{2q}\right]\\
&\qquad+ C\:\left( \int_{\RR^d} \psi(\xi) \left| G(t-\theta, x-\xi) - G(t, x-\xi) \right| (1+ |u(0, \xi)|) \ud \xi \right)^{2q}\\
&:= \Psi_{11} + \Psi_{12}\,.
\end{align*}
By part (2) of Lemma \ref{L:LimPmom}, we have that 
\begin{equation}
\Psi_{11}\leq C 
\left[
\int_{\R^d}\psi(\xi)G(t-\theta,x-\xi) \Norm{u(\theta,\xi)-u(0,\xi)}_{2q}\ud \xi
\right]^{2q}\le
C \theta^{q (\alpha \wedge \beta)}\,.
\end{equation}
As for $\Psi_{12}$, by the mean-value theorem, we see that
for some $s\in [t-\theta,t]$,
\begin{align*}
\Psi_{12} &\le 
C \sup_{x\in Q} (1+|g(x)|)^{2q}\left( \int_{\RR^d} \psi(\xi) \left|\frac{\partial}{\partial t} G(s, x-\xi)\right|  \ud \xi \right)^q \theta^q\le C \,\theta^q\,.
\end{align*}
For $\Psi_2$, 
\begin{align*}
\Norm{\Psi_2}_{2q}^2\leq& \int_0^{\theta}\ud s \int_{\RR^{2d}} \ud y\ud y' \: G(t-s, x-y)G(t-s, x-y')\Norm{Y_0(s,y) Y_0(s,y')}_q f(y-y')\\
\leq& C \sup_{(s,y)\in [0,t]\times \RR^d} \Norm{Y_0(s,y)}_{2q}^2 \int_0^{\theta}\ud s \int_{\RR^d} e^{-(t-s)|\xi|^2} \widehat{f}(\ud \xi).
\end{align*}
Because 
\[
e^{-t|x|^2} \le \frac{C}{(1+t|x|^2)^{1-\alpha}}
\le \frac{C}{t^{1-\alpha} ((1/T)+|x|^2)^{1-\alpha}}\quad\text{for all $(t,x)\in [0,T]\times\R^d$,}
\]
by \eqref{E:Dalang2}, we see that 
\[
\Norm{\Psi_2}_{2q}^2 \le  C \int_0^\theta \frac{1}{(t-s)^{1-\alpha}} \ud s \le C \theta^{\alpha},\quad\text{for all $\theta\in [0,t]$.}
\]
Applying the Burkholder-Davis-Gundy inequality to $\Psi_3$ in  \eqref{E:Y-Y} yields
\begin{align*}
\Norm{Y_{\theta}(t,x)-Y_0(t,x)}_{2q}^2
\leq &
C \left(  \Norm{\Psi_1}_{2q}^2 + \Norm{\Psi_2}_{2q}^2\right) \\
&+ C \int_{\theta}^{t}\ud s\int_{\RR^{2d}}\ud y\ud y'\: G(t-s, x-y) G(t-s, x-y') f(y-y')\\
&\hspace{6em} \times\sup_{z\in \RR^d} \Norm{Y_{\theta}(s,z) - Y_0(s,z)}_{2q}^2  \\
&\leq C \theta^{\alpha \wedge \beta} + C \int_{\theta}^t (t-s)^{\alpha-1} \sup_{z\in \RR^d} \Norm{Y_{\theta}(s,z) - Y_0(s,z)}_{2q}^2 \ud s\,.
\end{align*}
Then Gronwall's Lemma (see, e.g., Lemma A.2 in \cite{CHN16Density}) implies that 
\[
\sup_{z\in \RR^d} \Norm{Y_{\theta}(s,z) - Y_0(s,z)}_{2q}^2
\le  C \theta^{\alpha\wedge\beta}.
\]
Therefore,
\begin{equation}
\sup_{0 \leq \theta \leq \epsilon^r, x\in \RR^d} \E \left[ \left(Y_{\theta}(t,x)-Y_0(t,x)\right)^{2q}\right] \leq C \epsilon^{r (\alpha \wedge \beta) q}\,,
\end{equation}
which implies that 
\begin{equation}
\bbP(A_2)\leq C 
\epsilon^{q (r-1)}\epsilon^{\frac{1}{2}r (\alpha \wedge \beta) q} =C 
\epsilon^{q(r-1+\frac{\alpha\wedge\beta}{2})}\,.
\end{equation}
Thus we can choose any $r\in \left(1- \frac{\alpha\wedge\beta}{2}, 1\right) 
\subseteq (0,1)$.
Theorem \ref{T_:Single} is proved by applying Lemma A.1 of \cite{CHN16Density}. 
\end{proof}

\bigskip
Now we are ready to prove Theorem \ref{T:Single}.
\begin{proof}[Proof of Theorem \ref{T:Single}]
Recall the mild solution $u(t,x)=J_0(t,x)+I(t,x)$ in \eqref{E:WalshSI} and the critical time $t_0$ is defined in \eqref{E:IFF},that is,
\[
t_0:=\inf\left\{s>0,\: \sup_{y\in\R^d}\left|\rho\left(s,y,(G(s,\cdot)*\mu)(y)\right)\right|\ne 0\right\}.
\]
If condition \eqref{E:IFF} is not satisfied, that is, $t\le t_0$, then $I(t,x)\equiv 0$. In this case, $u(t,x)\equiv J_0(t,x)$ is deterministic. Hence, $u(t,x)$ doesn't have a density. This proves one direction for both parts (a) and (b). 

On the other hand, if condition \eqref{E:IFF} is satisfied, that is, $t>t_0$, then by the continuity of the function
\[
(0,\infty)\times\R^d\times\R^d \ni (t,x,z) \mapsto \rho\left(t,x, (G(t,\cdot)*\mu)(x)+z\right) \in\R,
\]
we know that for some $\epsilon_0\in (0,t-t_0)$ and some $x_0\in\R^d$, it holds that 
\begin{align}\label{E:rho=0}
\rho\left(t_0+\epsilon, y, (G(t_0+\epsilon,\cdot)*\mu)(y)+z\right)\ne 0
\end{align}
for all $(\epsilon,y,z)\in (0,\epsilon_0)\times B(x_0,\epsilon_0)\times [-\epsilon_0,\epsilon_0]$,
where $B(x,r):= \left\{y\in\R^d: \Norm{x-y}\le r\right\}$.
Let 
\[
\tau:= \left(t_0+\epsilon_0\right)\wedge \inf\left\{t>t_0, \: \sup_{y\in B(x_0,\epsilon_0)} |I(t,y)|\ge \epsilon_0\right\}.
\] 
Denote $\W_*(t,x):=\W(t+\tau,x)$ and $\rho_*(t,x,z) := \rho(t+\tau,x,z)$. Let $u_*(t,x)$ be the solution to the following stochastic heat equation 
\begin{align}\label{E:FracHt*}
 \begin{cases}
  \left(\displaystyle\frac{\partial}{\partial t} - \frac{1}{2}\Delta \right) u_*(t,x) =
\rho_*(t,x,u_*(t,x))  \dot{W}_*(t,x),& t>0\;,\: x\in\R^d,\cr
u_*(0,x) = u(\tau,x),\:& x\in\R^d.
 \end{cases}
\end{align}
By the construction, $\sup_{y\in B(x_0,\epsilon_0)}|I(\tau,y)|\le \epsilon_0$ and thus, property \eqref{E:rho=0} implies that
\[
\rho_*\left(0,y,u(\tau,y)\right)=
\rho\left(\tau,y,J_0(\tau,y)+I(\tau,y)\right)\ne 0, \quad\text{for all $y\in B(x_0,\epsilon_0)$.}
\]
Because $y\mapsto u(\tau,y)$ is $\beta$-H\"older continuous a.s. for any 
$\beta\in (0,\alpha)$, we can apply Theorem \ref{T_:Single} to SPDE 
\eqref{E:FracHt*} to see that if $\rho$ is differentiable in the third argument 
with bounded Lipschitz continuous derivative, then $u_*(t,x)$ has a conditional 
density, denoted as $\psi_t(x)$, that is absolutely continuous with respect to 
the Lebesgue measure. Moreover, if $\rho$ is infinitely differentiable in the 
third argument with bounded derivatives, then this conditional density 
$x\mapsto \psi_t(x)$ is smooth a.s.
For any nonnegative continuous function $g$ on $\R$ with compact support, 
\begin{align*}
\E\left[g(u(t,x))\right] &=\E\left[\E \left[g(u(t,x))| \calF_\tau\right]\right]\\ 
&=\E\left[\E \left[g(u_*(t-\tau,x))| \calF_\tau\right]\right]\\
&=\E\left[\int_\R g(x) \psi_{t-\tau}(x)\ud x\right] \\
&= \int_\R g(x)\: \E\left[\psi_{t-\tau}(x)\right]\ud x.
\end{align*}
Therefore, if $\rho$ is differentiable in the third argument with bounded Lipschitz continuous derivative, then $u(t,x)$ has a density, which is equal to $\E\left[\psi_{t-\tau}(x)\right]$. It is clear that this density is absolutely continuous with respect to the Lebesgue measure. Moreover, if $\rho$ is infinitely differentiable in the third argument with bounded derivatives, this density is smooth.
This completes the proof of Theorem \ref{T:Single}.
\end{proof}

\subsection{Assumption \ref{A:Rate-Sharp} (Properties of \texorpdfstring{$k(t)$}{})}
\label{SS:Rate}
In this part, we study properties of $k(t)$ defined in \eqref{E:k}, which is closely related to the following function 
\begin{align}\label{E:Def-Vd1}
 V_d(t)&:=\int_0^{t} \ud s\iint_{\R^{2d}} \ud y\ud y' \: G(s,y)G(s,y')f(y-y').
\end{align}
By Fourier transform, we see that
\begin{align}\label{E:Def-Vd2}
 V_d(t) =&(2\pi)^{-d}\int_0^t\ud s \int_{\R^d} e^{-s |\xi|^2}\widehat{f}(\ud 
\xi)
 =\int_0^t k(2s)\ud s.
\end{align}


\begin{lemma}\label{L:Rate}
If 
\begin{align}\label{E:B}
\begin{aligned}
\overline{B}:= &\left\{\beta\in [0,1):\: \limsup_{t\downarrow 0} 
t^{\beta}k(t)<\infty\right\},\\
\underline{B}:=& \left\{\beta\in [0,1):\: \liminf_{t\downarrow 0} 
t^{\beta}k(t)>0\right\},
\end{aligned}
\end{align}
then the following properties hold:
\begin{enumerate}[(1)]
 \item If $\overline{B}\ne \emptyset$, then $\inf \overline{B}\ge 
\sup\underline{B}$.
 \item If, for some $\beta\in [0,1)$, $\lim_{t\downarrow 0} t^\beta k(t)$ exists 
and belongs to $(0,\infty)$ , then $\inf \overline{B}= \sup\underline{B}=\beta$.
 \item If $\overline{B}\ne \emptyset$, then for any $\beta\in \overline{B}$, 
 \[
 \sup_{t\in [0,T]} t^\beta k(t)<\infty\quad\text{and}\quad
 V_d(t) \le C t^{1-\beta}.
 \]
 \item $\underline{B}$ is never an empty set since $0\in \underline{B}$. 
Moreover,  for any $\beta\in\underline{B}$, 
 \[
 \inf_{t\in [0,T]} t^{\beta} k(t)>0\quad\text{and}\quad
 V_d(t) \ge C t^{1-\beta}.
 \]
 \end{enumerate}
\end{lemma}
\begin{proof}
Notice that $g(t)$ is a strictly positive and nonincreasing function on 
$(0,\infty)$, and $k(t)$ may blow up at $t=0$, from which part (1) is clear. 

As for (2), for any $\beta'<\beta$, we have that $\limsup_{t\downarrow 
0}t^{\beta'}k(t)=\infty$, which implies that $\beta=\inf \overline{B}$. 
Similarly, for any $\beta''>\beta$, we have that $\liminf_{t\downarrow 
0}t^{\beta''}k(t)=0$, which implies that $\beta=\sup \underline{B}$.  

{\medskip\noindent (3)} Fix $t\in [0,T]$. 
Denote $h(t) = \int_{\R^d}e^{-s(1+|\xi|^2)/2}\widehat{f}(\ud\xi)$.  It is clear 
that 
\[
h(t)\le k(t)\le e^T h(t).
\]
The function $h(t)$ is a smooth function for $t\in (0,T]$ because, by Dalang's 
condition \eqref{E:Dalang},  
\[
h^{(n)}(t) =\int_{\R^d}\frac{e^{-s(1+|\xi|^2)/2}}{(1+|\xi|^2)^n} \widehat{f}(\ud 
\xi)<\infty, \quad\text{for all $n\ge 1$.}
\]
Hence, for some $\beta>0$, 
\begin{align*}
\sup_{t\in[0,T]} t^\beta k(t) < \infty &\quad\Longleftrightarrow\quad
\sup_{t\in[0,T]} t^\beta h(t) < \infty \\
& \quad\Longleftrightarrow\quad
\limsup_{t\downarrow 0} t^\beta h(t)<\infty\\
& \quad\Longleftrightarrow\quad
\limsup_{t\downarrow 0} t^\beta k(t)<\infty. 
\end{align*}
Finally, notice that 
\[
V_d(t) = (2\pi)^{-d} \int_0^t k(2s)\ud s \le  C \left[\sup_{s\in [0,T]} s^\beta 
k(s)\right] \int_0^t s^{-\beta}\ud s = C t^{1-\beta}.
\]
This proves (3).

{\medskip\noindent (4)} Since $k(t)$ is a strictly positive and nonincreasing 
function, we see that $0\in\underline{B}$. Since $h(t)>0$. We have that 
\[
e^{-T} h(t)^{-1}\le k(t)^{-1}\le h(t)^{-1}.
\]
By the same arguments as in part (3), noticing that $h^{-1}(t)$ is a smooth 
function on $(0,\infty)$, we see that
\begin{align*}
\inf_{t\in[0,T]} t^\beta k(t) >0 &\quad\Longleftrightarrow\quad
\sup_{t\in[0,T]} t^{-\beta} k(t)^{-1} <\infty\\
& \quad\Longleftrightarrow\quad 
\sup_{t\in[0,T]} t^{-\beta} h(t)^{-1} < \infty \\
& \quad\Longleftrightarrow\quad
\limsup_{t\downarrow 0} t^{-\beta} h(t)^{-1}<\infty\\
& \quad\Longleftrightarrow\quad
\limsup_{t\downarrow 0} t^{-\beta} k(t)^{-1}<\infty\\
& \quad\Longleftrightarrow\quad
\liminf_{t\downarrow 0} t^{\beta} k(t)>0, 
\end{align*}
and 
\[
V_d(t) = (2\pi)^{-d} \int_0^t k(2s)\ud s \ge  C \left[\inf_{s\in [0,T]} s^\beta 
k(s)\right] \int_0^t s^{-\beta}\ud s = C t^{1-\beta},
\]
which proves (4). This completes the proof of Lemma \ref{L:Rate}.
\end{proof}

The following proposition shows that many common correlation functions satisfy 
Assumption \ref{A:Rate-Sharp}.
\begin{proposition}\label{P:Rate}
We have that
\begin{enumerate}[(1)]
 \item For the space-time white noise case, that is, $f(x)=\delta_0(x)$, $\lim_{t\downarrow 0} t^{d/2} k(t) 
=(2\pi)^{-d/2}$. In particular, when $d=1$, $\inf \overline{B}= 
\sup\underline{B}=1/2$, 
 and when $d\ge 2$, $\overline{B} = \underline{B} = \phi$.
 \item For the Riesz kernel $f(x)=|x|^{-\beta'}$ with $\beta'\in (0,2\wedge 
d)$, 
 $\lim_{t\downarrow 0} t^{\beta'/2} g(t) =C\in(0,\infty)$ and hence, 
 $\inf \overline{B}= \sup\underline{B}=\beta'/2$.
 \item For the fractional noise case, that is, $f(x) = \prod_{j=1}^d |x_j|^{2H_j-2}$ with $H_j\in (1/2,1)$ and 
$d-\sum_{j=1}^d H_i<1$, we have that $\inf 
\overline{B}= \sup\underline{B}=d-H$, where $H:=\sum_{j=1}^d H_j$.
 \item If $f(0)<\infty$ (or equivalently $\hat{f}\in L^1(\R^d)$), then $\inf 
\overline{B}= \sup\underline{B}=0$.
\item If $f$ is the Bessel kernel of order $\alpha'>0$ (see \eqref{E:Bessel}), then we have that 
\begin{enumerate}
 \item If $\alpha'>d-2$, $f$ satisfies Dalang's condition \eqref{E:Dalang2} for any $\alpha\in (0,(\alpha'+2-d)/2)$.
 \item As $t\rightarrow 0,$ the function $k(t)$ defined in \eqref{E:k} satisfies the following property: 
 \[
 k(t) \asymp 
 \begin{cases}
  \displaystyle t^{\frac{\alpha'-d}{2}} & \text{if $\alpha'\in (0,d)$,}\\[0.5em]
  \displaystyle \log(1/t) & \text{if $\alpha'=d$,}\\[0.5em]
  \displaystyle 1 & \text{if $\alpha'>d$.}
 \end{cases}
 \]
\end{enumerate}
\end{enumerate}
 \end{proposition}
\begin{proof}
 {\medskip\noindent (1)} For the space-time white noise, $k(t) = G(t,0)=(2\pi 
t)^{-d/2}$. 
 It is clear that $d/2\in [0,1)$ if and only if $d=1$. This proves (1).

{\medskip\noindent (2)} 
For the Riesz kernel case, $\widehat{f}(\xi) = C |\xi|^{\beta'-d}$ and hence, 
\[
k(t) = C\int_{\R^d} e^{-t|\xi|^2/2} |\xi|^{\beta'-d} \ud\xi =
C \int_0^\infty e^{-t r} r^{\beta'-d} r^{d-1}\ud r= C t^{-\beta'/2},
\]
This case is proved by an application of part (2) of Lemma \ref{L:Rate}.

{\medskip\noindent (3)} In this case, 
\[
\widehat{f}(\ud \xi) = \prod_{j=1}^d C_j |\xi_j|^{1-2H_j}\ud \xi_j\quad 
\text{with} \quad C_j=\frac{\Gamma(2H_j+1)\sin(\pi H_j)}{2\pi}>0.
\]
Therefore, 
\[
k(t) = \prod_{j=1}^d C_j \int_\R e^{-t \xi_j^2 /2} |\xi_j|^{1-2H_j}\ud\xi_j
=\prod_{j=1}^d C_j \Gamma(1-H_j) t^{H_j-1} = C t^{\sum_{j=1}^d H_j-d}.
\]
Then an application of part (2) of Lemma \ref{L:Rate} proves part (3).

{\medskip\noindent (4)}
If $f(0)<\infty$, then $k(t)\le f(0)<\infty$, where 
the first inequality is due to the fact that both $f(\cdot)$ and $G(t,\cdot)$ 
are nonnegative definite. Hence, $\inf \overline{B}=0$.

{\medskip\noindent (5)} Now we study the Bessel kernel \eqref{E:Bessel}. Notice that 
\[
f(x) = (4\pi)^{d/2} \int_0^\infty w^{\frac{\alpha'-2}{2}}e^{-w}G(2w,x) \ud w,
\]
where $G(w,x)$ is the heat kernel. Hence, the Fourier transform of $f$ is equal to 
\begin{align*}
\widehat{f}(x) 
&= (4\pi)^{d/2} \int_0^\infty w^{\frac{\alpha'-2}{2}}e^{-w} e^{-w |\xi|^2}\ud w\\
&=(4\pi)^{d/2} (1+|\xi|^2)^{-\alpha'/2} \int_0^\infty e^{-w}w^{\frac{\alpha'-2}{2}}\ud w\\
&=(4\pi)^{d/2} \Gamma(\alpha'/2)(1+|\xi|^2)^{-\alpha'/2}.
\end{align*}
Hence, Dalang's condition \eqref{E:Dalang2} becomes
\[
\int_{\R^d}\frac{\ud\xi}{(1+|\xi|^2)^{\frac{\alpha'}{2}+1-\alpha}}<\infty,
\]
which implies that $\alpha'+2-2\alpha>d$. Therefore, if $\alpha'>d-2$, Dalang's condition \eqref{E:Dalang2} is satisfied for any $\alpha\in (0,(\alpha'+2-d)/2)$. This proves part (a). 

As for (b), notice that from \eqref{E:k}, by the spherical coordinates, 
\[
k(t) = C \int_{\R^d}\frac{e^{-t|\xi|^2/2}}{(1+|\xi|^2)^{\alpha'/2}}\ud \xi
= C \int_0^\infty \frac{e^{-tr^2/2}}{(1+r^2)^{\alpha'/2}}r^{d-1}\ud r
= C U\left(\frac{d}{2},\frac{2+d-\alpha'}{2},\frac{t}{2}\right),
\]
where $U(a,b,z)$ is the {\it confluent hypergeometric function} (see \cite[Chapter 13]{NIST2010}) and the last equation is due to \cite[Eq. 13.4.4 on p. 326]{NIST2010}. 
Therefore, by the seven cases from 13.2.16 to 13.3.22 of \cite{NIST2010}, we see that as $t\rightarrow 0$,
\[
k(t) = 
\begin{cases}
\displaystyle C  t^{\frac{\alpha'-d}{2}}  + O(t^{1+\frac{\alpha'-d}{2}})&  \alpha' \in (0,d-2),\\[0.5em]
 \displaystyle C  t^{-1}  + O(\log t) &  \alpha' =d-2, \\[0.5em]
 \displaystyle C  t^{\frac{\alpha'-d}{2}}  + C'+ O(t^{1+\frac{\alpha'-d}{2}}) &  \alpha' \in (d-2,d),\\[0.5em]
 \displaystyle C \log(1/t) + C' + O(t \log t) &  \alpha' =d, \\[0.5em]
 \displaystyle C+ O(t^{\frac{\alpha'-d}{2}}) &  \alpha' \in (d,d+2),
 \\[0.5em]
 \displaystyle C + O(t\log t) &  \alpha' =d+2, \\[0.5em]
 \displaystyle C + O(t) &  \alpha' > d+2.
\end{cases}
\]
Combining the above cases proves part (b).
This completes the proof of Proposition \ref{P:Rate}.
\end{proof}

\medskip
We will need the following lemma.
\begin{lemma}\label{L:V/V}
 For all $\alpha>0$, it holds that
 \[
 \sup_{t>0} \frac{V_d(\alpha t)}{V_d(t)} <\infty.
 \]
\end{lemma}
\begin{proof}
Simple calculations show that
\[
1=\inf_{x\in\R}  \frac{1+x^2}{x^2} (1-e^{-x^2}) < \sup_{x\in\R} 
\frac{1+x^2}{x^2} (1-e^{-x^2}):=C' <\infty.
\]
Noticing that by \eqref{E:Def-Vd2}, 
\[
V_d(t) = (2\pi)^{-d}\int_{\R^d}\frac{1-e^{-t|\xi|^2}}{|\xi|^2}\widehat{f}(\ud 
\xi).
\]
we see that for all $t>0$, 
\[
(2\pi)^{-d} t \int_{\R^d}\frac{1}{1+t |\xi|^2}\widehat{f}(\ud \xi) \le  V_d(t)  
\le 
C' (2\pi)^{-d} t \int_{\R^d}\frac{1}{1+ t |\xi|^2}\widehat{f}(\ud \xi).
\]
In case of $\alpha<1$, 
\begin{align}\label{E1:V/V}
\frac{V_d(\alpha t)}{V_d(t)}\le \frac{\displaystyle C' 
\alpha t\int_{\R^d}\frac{\widehat{f}(\ud 
\xi)}{1+\alpha t|\xi|^2}}{\displaystyle t 
\int_{\R^d}\frac{\widehat{f}(\ud \xi)}{1+t|\xi|^2}}
= C' \frac{\displaystyle \int_{\R^d}\frac{\widehat{f}(\ud 
\xi)}{1/\alpha+t|\xi|^2}}{\displaystyle\int_{\R^d}\frac{\widehat{f}
(\ud \xi)}{1+t|\xi|^2}} \le C',
\end{align}
and in case of $\alpha \ge 1$, 
\begin{align}\label{E2:V/V}
\frac{V_d(\alpha t)}{V_d(t)}\le \frac{\displaystyle C' 
\alpha t\int_{\R^d}\frac{\widehat{f}(\ud 
\xi)}{1+\alpha t|\xi|^2}}{\displaystyle t 
\int_{\R^d}\frac{\widehat{f}(\ud \xi)}{1+t|\xi|^2}}
= C' \frac{\displaystyle \alpha  \int_{\R^d}\frac{\widehat{f}(\ud 
\xi)}{1+t|\xi|^2}}{\displaystyle\int_{\R^d}\frac{\widehat{f}
(\ud \xi)}{1+t|\xi|^2}} \le C'\alpha.
\end{align}
This proves the lemma.
\end{proof}

\subsection{Density at multiple points (Proof of Theorem \ref{T:Mult})}
\label{S:Mult}
We start the proof of Theorem \ref{T:Mult} by denoting 
\begin{align}\label{E:f_gb}
f_{\gamma,\beta}(x):=\exp\left\{-2\beta
\left[\log\frac{1}{|x|\wedge 1}\right]^\gamma
\right\},\quad\text{for $x\in\R$,}
\end{align}
where $\gamma\in (0,1+\alpha)$ and $\beta>0$ are the constants given in the condition \eqref{E:lip}.
Fix $t>0$ and $m$ distinct points $\{x_1,\dots,x_m\}\subseteq\R^d$. Let $\epsilon_0$ be any positive  constant such that 
\[
\epsilon_0 \le \min(t/2,1)\quad\text{and}\quad
2(2\epsilon_0)^{1/2} \le \min_{i\ne j}|x_i-x_j|.
\]

We begin by writing 
\begin{equation}\label{E:Du}
D_{r, z} u(t,x) = \rho(u(r,z))G(t-r, x-z) + Q_{r, z} (t,x)\,,
\end{equation}
where $r\in (0,t]$, $z\in\R^d$, and 
\begin{equation}
Q_{r,z}(t,x) =\int_0^t \int_{\RR^d} G(t-s, x-y) \rho'(u(s,y)) D_{r,z}u(s,y)W(\ud s\ud y)\,.
\end{equation}
Denote 
\[
\psi_{r,z}(t,x):=D_{r,z}u(t,x).
\]
Let $S_{r, z}(t,x)$ be the solution to the equation
\begin{equation}
S_{r,z}(t,x) = G(t-r, x-z) + \int_r^t \int_{\RR^d} G(t-s, x-y) \rho'(u(s,y))S_{r, z}(s,y) W(\ud s\ud y)\,.
\end{equation}
Using the uniqueness of the solution to the SPDE, we can write
\begin{equation}
\psi_{r,z}(t,x)= S_{r,z}(t,x)\rho(u(r,z))\,.
\end{equation}
Note that this also shows that the Malliavin derivative of the solution 
$(r,z)\mapsto D_{r,z}u(t,x)$ is a function in $L^2(\R_+;\calH)$.
Set $\calT:=[t/2,t]$. 
For some $R\ge 2A$ large enough where $A$ is the constant in Assumption \ref{A:TailBlowup}, set $\calS =\{x\in\R^d, 
|x|\le R\}$  such that $x_i\in \calS$ for all $i=1,\cdots, m$.
Define
\begin{align}\label{E:Vd}
\widetilde{V}_d(\epsilon) := \int_0^{\epsilon}\ud r \iint_{\calS^2}\ud z\ud 
z'\:  
G(r, z) G(r,z') f(z-z')\:.
\end{align}

Recall that $V_d(\cdot)$ is defined in \eqref{E:Def-Vd1}. The following lemma 
shows that both $V_d(\epsilon)$ and $\widetilde{V}_d(\epsilon)$ has the same 
order as $\epsilon$ goes to zero. 

\begin{lemma}\label{L:VTildeV}
Under Assumption \ref{A:TailBlowup}, it holds that 
\[
\lim_{\epsilon\rightarrow 0_+} \frac{V_d(\epsilon)}{\widetilde{V}_d(\epsilon)} 
=1.
\]
\end{lemma}
\begin{remark}
 We first remark that if $f$ satisfies some scaling property, such as the Riesz kernel, then this property can be easily proved. 
 Let us see this through Riesz kernel case.
 By the l'Hopital rule,
 \begin{align*}
 \lim_{\epsilon\rightarrow 0_+} \frac{V_d(\epsilon)}{\widetilde{V}_d(\epsilon)} 
&= 
\lim_{\epsilon\rightarrow 0_+} 
\frac{
\iint_{\R^{2d}}G(\epsilon,z)G(\epsilon,z') f(z-z') \ud z\ud z'
}{
\int_{|z|\le R}\int_{|z'|\le R}G(\epsilon,z)G(\epsilon,z') f(z-z') \ud z\ud z'
}\\
&=
\lim_{\epsilon\rightarrow 0_+} 
\frac{
\iint_{\R^{2d}}G(1,z)G(1,z') f(\sqrt{\epsilon} (z-z')) \ud z\ud z'
}{
\int_{|z|\le \frac{R}{\sqrt{\epsilon}}}\int_{|z'|\le \frac{R}{\sqrt{\epsilon}}}G(1,z)G(1,z') f(\sqrt{\epsilon} (z-z')) \ud z\ud z'
}\\
&=
\lim_{\epsilon\rightarrow 0_+} 
\frac{
\iint_{\R^{2d}}G(1,z)G(1,z') f(z-z') \ud z\ud z'
}{
\int_{|z|\le \frac{R}{\sqrt{\epsilon}}}\int_{|z'|\le \frac{R}{\sqrt{\epsilon}}}G(1,z)G(1,z') f(z-z') \ud z\ud z'
}=1,
\end{align*}
where the last step is due to the dominated convergence theorem. 
However, for general $f$, to prove this property is much less straightforward. Indeed, we need to impose some conditions on $f$, namely, Assumption \ref{A:TailBlowup}.
\end{remark}

\begin{proof}[Proof of Lemma \ref{L:VTildeV}]
Let $A>1$ be the constant in Assumption \ref{A:TailBlowup}.
Throughout the proof, we assume that $\epsilon\in (0,1/(A^2R^2))$.
For any $\epsilon>0$ and $H\subseteq\R^{2d}$, denote 
\[
I_{H}(\epsilon):=\iint_{H}\ud z\ud z'\:  
G(\epsilon, z) G(\epsilon,z') f(z-z').
\]
By the l'Hopital rule, 
\begin{align*}
 \lim_{\epsilon\rightarrow 0_+} \frac{V_d(\epsilon)}{\widetilde{V}_d(\epsilon)} 
&= \lim_{\epsilon\rightarrow 0_+} 
\frac{I_{\R^{2d}}(\epsilon)
}{I_{\calS^{2}}(\epsilon)}.
\end{align*}
Notice that by \eqref{E:GGGG}, $G(\epsilon,z)G(\epsilon,z')=G(2\epsilon,z-z')G(\epsilon/2,(z+z')/2)$. Hence,
by change of variables $y=z-z'$ and $y'=(z+z')/2$, we see that 
\begin{align*}
I_{\R^{2d}}(\epsilon)
& =
\iint_{\R^{2d}}\ud y\ud y' G(2\epsilon, y)G(\epsilon/2, y') f(y)\\
&=
\int_{\R^{d}} G(2\epsilon, y) f(y)\ud y\\
&= \int_{|y|\le \frac{R}{2\sqrt{\epsilon}}} G(2,y)f(\sqrt{\epsilon} y)\ud y 
+ \int_{|y|> \frac{R}{2\sqrt{\epsilon}}} G(2,y)f(\sqrt{\epsilon} y)\ud y.
\end{align*}
Recall that $R\ge 2A$. By Assumption \ref{A:TailBlowup}, 
\[
I_{\R^{2d}}(\epsilon)\le 
\int_{|y|\le \frac{R}{2\sqrt{\epsilon}}} G(2,y)f(\sqrt{\epsilon} y)\ud y 
+ C_R \Theta(\epsilon),
\]
where 
\[
C_R :=\sup_{|x|\ge R/2} f(x)\quad\text{and}\quad
\Theta(\epsilon):=
\int_{|y|> \frac{R}{2\sqrt{\epsilon}}} G(2,y)\ud y.
\]
Similarly,
\begin{align*}
I_{\calS^{2}}(\epsilon)& \ge 
\int_{|y|\le R}\ud y\: f(y)G(2\epsilon, y) \int_{|y'|\le R/2} \ud y' \: 
G(\epsilon/2, y') \\
& =
\int_{|y|\le \frac{R}{\sqrt{\epsilon}}}\ud y\: f(\sqrt{\epsilon} y)G(2, y) \int_{|y'|\le \frac{R}{2\sqrt{\epsilon}}} \ud y' \: 
G(1/2, y').
\end{align*}
For any $\delta \in (0,1)$, as $\epsilon$ is small enough, we can always ensure that 
\[
\int_{|y'|\le \frac{R}{2\sqrt{\epsilon}}} \ud y' \: 
G(1/2, y')\ge \delta,
\]
which implies that
\[
I_{\calS^{2}}(\epsilon)\ge
\delta \int_{|y|\le \frac{R}{\sqrt{\epsilon}}} f(\sqrt{\epsilon} y)G(2, y)\ud y.
\]
Therefore, for $\epsilon$ small enough,
\begin{align*}
\frac{I_{\R^{2d}}(\epsilon)}{I_{\calS^2}(\epsilon)}
&\le \delta^{-1}
\frac{\int_{|y|\le \frac{R}{2\sqrt{\epsilon}}} G(2,y)f(\sqrt{\epsilon} y)\ud y 
+ C_R \Theta(\epsilon)}{ \int_{|y|\le \frac{R}{\sqrt{\epsilon}}} G(2,y)f(\sqrt{\epsilon} y)\ud y}\\
&\le \delta^{-1}
\left(1+\frac{C_R \Theta(\epsilon)}{
\int_{|y|\le \frac{R}{\sqrt{\epsilon}}} G(2,y)f(\sqrt{\epsilon} y)\ud y}\right).
\end{align*}
Because $\epsilon<1/(A^2R^2)$, that is $\sqrt{\epsilon} R \le \sqrt{\epsilon} A < 1/A$, by Assumption \ref{A:TailBlowup}, 
\[
\int_{|y|\le \frac{R}{\sqrt{\epsilon}}} G(2,y)f(\sqrt{\epsilon} y)\ud y
\ge 
\int_{|y|\le R} G(2,y)f(\sqrt{\epsilon} y)\ud y
\ge C_{R,f},
\]
where 
\[
C_{R,f}:=\inf_{|z|<1/A} f(z) \int_{|y|\le R}G(2,y)\ud y >0. 
\]
Therefore, 
\[
\frac{I_{\R^{2d}}(\epsilon)}{I_{\calS^2}(\epsilon)}\le \delta^{-1}\left(
1+ C_{R,f}^{-1} C_R \Theta(\epsilon) 
\right) \rightarrow \delta^{-1},\quad \text{as $\epsilon\downarrow 0$.}
\]
Finally, since $\delta$ can be arbitrarily close to $1$, this proves the lemma. 
\end{proof}

\medskip

Let us  continue our proof of Theorem \ref{T:Mult}.
Define
\begin{align*}
\sigma_{i,j} &:=  \langle D u(t,x_i), D u(t, 
x_j)\rangle_{\mathcal{H}_{\calT,\calS}}\\
&= \int_{\calT}\ud r \iint_{\calS^2} \ud z\ud z' D_{r, z} u(t,x_i)D_{r, z'} u(t,x_j) f(z-z') \,.
\end{align*}
Let $\sigma$ be the matrix with entries $\sigma_{i,j}$, $1\leq i,j \leq m$. For 
any $\xi\in \RR^m$ and any $\epsilon\in (0,\epsilon_0)$, consider the inner 
product in the Euclidean space
\begin{align*}
\InPrd{\sigma \xi, \xi}  &= \sum_{i,j=1}^m \xi_i \xi_j \int_{\calT} \ud 
r\iint_{\calS^2} \ud z\ud z'\: \psi_{r, z}(t, x_i) \psi_{r, z'}(t, x_j) f(z-z') 
\\
&\geq  \int_{t-\epsilon}^t \ud r 
\iint_{\calS^2}\ud z\ud z' 
\left( \sum_{i=1}^m \psi_{r, z}(t, x_i)\xi_i\right) \left(\sum_{i=1}^m \psi_{r, 
z'} (t, x_j)\xi_j \right)f(z-z')\\
&\geq \sum_{j=1}^m \xi_j^2\int_{t-\epsilon}^t \ud r \iint_{\calS^2}\ud z\ud 
z'\:  \psi_{r, z}(t,x_j)\psi_{r, z'}(t,x_j)f(z-z')\\
&\quad + \sum_{j=1}^m \sum_{i \neq j} \xi_i \xi_j\int_{t-\epsilon}^t\ud r 
\iint_{\calS^2}
\ud z\ud z' \:
\psi_{r, z}(t, x_i) \psi_{r,z'}(t, x_j)f(z-z') \\
&= I_{\epsilon}^* + I_{\epsilon}^{(1)}(\xi)\,.
\end{align*}
Apply twice the following inequality 
\begin{align}\label{E:ab32}
\Norm{a+b}^2  \ge \left(\Norm{a} -\Norm{b}\right)^2\ge 
\frac{2}{3}\Norm{a}^2 -2\Norm{b}^2
\end{align}
to see that
\begin{align*}
I_{\epsilon}^* &\geq \frac{2}{3} \sum_{j=1}^m \xi_j^2
\int_{t-\epsilon}^t \ud r 
\iint_{\calS^2}  
\ud z \ud z'\:
\rho(u(r,z))\rho(u(r,z'))\\
&\hspace{6em}\times G(t-r, x_j-z)G(t-r, x_j-z')f(z-z') \\
 &\qquad -2 \sum_{j=1}^m \xi_j^2\int_{t-\epsilon}^t \ud r\iint_{\calS^2} \ud 
z\ud z'\: Q_{r,z}(t,x_j)Q_{r,z'}(t,x_j)f(z-z') \\
&\geq \frac{4}{9} \sum_{j=1}^m \xi_j^2\int_{t-\epsilon}^t \ud r \iint_{\calS^2}
\ud z \ud z'\:
\rho(u(t,x_j))^2 G(t-r, x_j-z)G(t-r, x_j-z')f(z-z') \\
&\qquad -\frac{4}{3} \sum_{j=1}^m \xi_j^2\int_{t-\epsilon}^t  \ud 
r\iint_{\calS^2} \ud z\ud z' \left[\rho(u(t, x_j)) - 
\rho(u(r,z))\right]\left[\rho(u(t, x_j)) - \rho(u(r,z'))\right] \\
 &\hspace{6em} \times G(t-r, x_j-z)G(t-r, x_j-z') f(z-z')  \\
 &\qquad -2 \sum_{j=1}^m \xi_j^2\int_{t-\epsilon}^t \ud r\iint_{\calS^2} \ud 
z\ud z'\: Q_{r,z}(t,x_j)Q_{r,z'}(t,x_j)f(z-z') \\
 &=: \: \frac{4}{9} I_{\epsilon}^{(0)}(\xi) - \frac{4}{3} I_{\epsilon}^{(2)} 
(\xi) - 2 I_{\epsilon}^{(3)}(\xi)\,.
\end{align*}
So
\begin{equation}
(\det\sigma)^{1/d} \geq \inf_{|\xi|=1}\langle \sigma \xi, \xi\rangle \geq 
\frac{4}{9} \inf_{|\xi|=1} I_{\epsilon}^{(0)}(\xi) - 2 \sum_{i=1}^3 
\sup_{|\xi|=1}|I_{\epsilon}^{(i)}(\xi)|\,.
\end{equation}

\bigskip

By the same arguments as those in the proof of Theorem 1.2 of 
\cite{CHN16Density}, we see that 
\begin{align*}
I^{(0)}_{\epsilon}(\xi) &\geq \inf_{x\in K} \rho(u(t,x))^2 \int_{0}^{\epsilon} 
\ud r \iint_{\calS^2} \ud z \ud z'\: G(r, z) G(r,z') f(z-z')  \\
&= V_
d(\epsilon)\inf_{x\in K} \rho(u(t,x))^2.
\end{align*}
For $I^{(i)}_{\epsilon}$, $i = 1,2,3$, we will estimate their upper bound of the 
$L^p$ norm.
By Minkowski's inequality and the Cauchy-Schwartz inequality, we see that 
\begin{align*}
\Norm{\sup_{|\xi|=1} I_{\epsilon}^{(1)}}_{p} &\leq \sum_{j=1}^m \sum_{i\ne j} 
\int_{t-\epsilon}^t\ud r \iint_{ \calS^2} \ud z\ud z' 
\:\left\|\psi_{r,z}(t,x_i) 
\psi_{r, z'}(t,x_j)\right\|_{p} f(z-z') \\
&\leq \sum_{j=1}^m \sum_{i\ne j} \int_{t-\epsilon}^t\ud r \iint_{\calS^2} \ud 
z\ud z' 
\left\|\psi_{r, z}(t,x_i)\right\|_{2p} \left\|\psi_{r, z'}(t,x_j)\right\|_{2p} 
f(z-z') \\
&\leq C_{t, K} \sum_{j=1}^m \sum_{i\ne j}
\int_{t-\epsilon}^t  \ud r
\iint_{\calS^2} 
\ud z\ud z' 
\left\|S_{r,z}(t, x_i) \right\|_{4p}\left\|S_{r,z'}(t, x_j) \right\|_{4p} 
f(z-z'),
\end{align*}
where 
\[
C_{t,K} := \sup_{(s,y)\in [t/2,t]\times K} \Norm{\rho(u(s,y))^2}_{2p}.
\]
The next lemma gives a moment bound for $S_{r,z}(t,x)$. 

\begin{lemma}\label{L:MomS}
For $t\in (0,T]$, $x\in\R^d$ and $p\ge 2$, there exists a constant $C=C(T,p)>0$ 
such that 
\[
\Norm{S_{r,z}(t,x)}_p\le C G(t-r,x-z),\qquad\text{for all $r\in (0,t]$ and 
$x,z\in\R^d$.}
\]
\end{lemma}
\begin{proof}
Because $\rho'$ is bounded, by the Burkholder-Davis-Gundy inequality,
\begin{align*}
\Norm{S_{r,z}(t,x)}_p^2 \le C_p G(t-r,x-z)^2 
+ C_p \int_r^t\ud s\iint_{\R^{2d}}& \ud y\ud y'\: 
G(t-s,x-y)G(t-s,x-y')\\
\times &
\Norm{S_{r,z}(s,y)}_p\Norm{S_{r,z}(s,y')}_p f(y-y').
\end{align*}
By setting $\theta=t-r$, $\eta=x-z$ and 
$g(\theta,\eta)=\Norm{S_{r,z}(\theta+r,\eta+z)}_p$, we see that 
\begin{align*}
g(\theta,\eta)^2 \le C_p G(\theta,\eta)^2 +C_p
\int_0^\theta \ud r \iint_{\R^{2d}}&  \ud y\ud y'\: 
G(t-s,\eta-y)G(t-s,\eta-y')\\
&\times g(s,y)g(s,y')f(y-y').
\end{align*}
Therefore, this lemma is proved by an application of Lemma \ref{L:IntIneq} with 
$\mu = \sqrt{C_p} \: \delta_0$.
\end{proof}
By Lemma \ref{L:MomS} and the property of $\Psi_n(T,x;0)$ in Proposition 
\ref{P:Psi}, we have that
\begin{align*}
\Norm{\sup_{|\xi|=1} I_{\epsilon}^{(1)}}_{p} \leq
& C \sum_{j=1}^m \sum_{i\ne j}
\int_{t-\epsilon}^t \ud r \iint_{ \calS^2}
\ud z \ud z'\:
G(t-r, x_i-z) G(t-r, x_j-z') f(z-z')\\
\le & 
C \sum_{j=1}^m \sum_{i\ne j}
C_{x_i-x_j} \widetilde{V}_d(\epsilon) \epsilon^{\beta}
\\
= & C' \widetilde{V}_d(\epsilon) \epsilon^{\beta},
\end{align*}
where the constant $C'$ in the last expression depends on $x_i, i=1,\dots, m$.

To estimate $I^{(2)}_{\epsilon}(\xi)$, we note that by Theorem \ref{T:Holder},
\begin{equation}
\Norm{u(r, z)-u(s,y)}_{p} \leq C_p \left( |r-s|^{\frac{\alpha}{2}} + 
|y-z|^{\alpha}\right)\quad \text{for all } \ r,s, \in [0,T]\,, y, z \in K\ \,,
\end{equation}
for some constant $C_p$ which depends on $T$ and $K$. Thus we have
\begin{align*}
\Norm{\sup_{|\xi|=1}I_{\epsilon}^{(2)}(\xi)}_{p} &\leq \sum_{j=1}^m 
\int_{t-\epsilon}^t \ud r 
\iint_{\calS^2}\ud z\ud z'\:
\Norm {u(r, z)-u(t, x_j)}_{2p}\Norm {u(r, z')-u(t, x_j)}_{2p}\\
&\hspace{14em}\times G(t-r, x_j-z)G(t-r, x_j-z') f(z-z') \\
&\leq C \int_0^{\epsilon} \ud r \iint_{\calS^2}
\ud z\ud z'  \left( |r|^{\frac{\alpha}{2}} + |z|^{\alpha} \right)\left( 
|r|^{\frac{\alpha}{2}} + |z'|^{\alpha} \right) G(r, z) G(r, z') f(z-z') 
\\
&\leq C \left( |\epsilon|^{\frac{\alpha}{2}} + |\sqrt{\epsilon}|^{\alpha} 
\right)^2 \int_0^{\epsilon} \ud r \iint_{\calS^2}
\ud z\ud z' \:  G(r, z) G(r, z') f(z-z') \\
&= C \epsilon^\alpha \widetilde{V}_d(\epsilon).
\end{align*}

To estimate $I_{\epsilon}^{(3)}(\xi)$, we first claim that 
\begin{equation}\label{E:Q Lp}
\Norm{Q_{r, z}(t, x)}_{2p}^2 
\leq C \:G(t-r, x-z)^2 (t-r)^\alpha,\qquad \text{for all $z \in K$}\,.
\end{equation}
Indeed, by the Burkholder-Davis-Gundy inequality, we see that
\begin{align*}
\Norm{Q_{r,z}(t,x)}_{2p}^2 &\leq 4p \int_r^t \ud s \iint_{\RR^{2d}} \ud y\ud y' 
\: G(t-s, x-y) G(t-s, x-y')\\
 &\hspace{5em} \times \Norm{D_{r, z}u(s, y)}_{2p} \Norm{D_{r, z}u(s, y')}_{2p} 
f(y-y') \\
 &\leq  C_{t, K}  p \int_0^t\ud s \iint_{\RR^{2d}} \ud y\ud y' \: G(t-s, x-y) 
G(t-s, x-y')\\
 &\hspace{5em}\times \Norm{S_{r, z}u(s, y)}_{4p} \Norm{S_{r, z}u(s, y')}_{4p} 
f(y-y') \\
 & \leq  C \int_r^t \ud s\iint_{\RR^d}\ud y\ud y'\: 
 G(t-s, x-y) G(t-s, x-y') \\
 &\hspace{5em} \times G(s-r, y-z) G(s-r, y'-z)f(y-y')\,,
\end{align*}
where we have applied Lemma \ref{L:MomS} and have used the inequality that for 
$z \in K$ and $r\in (t-\epsilon,t)$, 
\begin{align*}
\Norm {D_{r,z}u(s,y)}_{2p} \leq \Norm{S_{r,z}u(s,y)}_{4p} 
\Norm{\rho(u(r,z))}_{4p} \leq C_{t,K} \Norm{S_{r,z}u(s,y)}_{4p}\,.
\end{align*}
Next we use identity \eqref{E:GGGG} to see that that 
\begin{align*}
\Norm{Q_{r,z}(t,x)}_{2p}^2 &\leq C  \int_r^t\ud s \iint_{\RR^{2d}} \ud y\ud y'\: 
G(t-r, x-z)G(t-r, x-z) \\
&\qquad \times G\left( \frac{(t-s)(s-r)}{t-r}, 
y-\frac{t-s}{t-r}z-\frac{s-r}{t-r}x\right)\\
&\qquad \times G\left( \frac{(t-s)(s-r)}{t-r}, 
y'-\frac{t-s}{t-r}z-\frac{s-r}{t-r}x\right) f(y -y') \\
&\leq C\: G(t-r, x-z)^2 \int_r^t\ud s \int_{\RR^d} 
e^{-\frac{2(t-s)(s-r)}{t-r}|\xi|^2} \widehat{f}(\ud\xi)\,.
\end{align*}
Notice that for $\beta>0$, $x\ge 0$ and $\theta\in (0,\Theta]$, 
\begin{align}\label{E:Exp-Poly}
(1+x^\beta ) e^{-\theta x^2} = \theta^{-\beta/2}e^{-(\sqrt{\beta}x)^2}\left(
\theta^{\beta/2} + (\sqrt{\theta}x)^\beta
\right)
\le \theta^{-\beta/2}
\max_{y>0} e^{-y^2}(\Theta^{\beta/2}+y^\beta).
\end{align}
Apply the above inequality with $\beta=2(1-\alpha)$ and
\[
\theta := \frac{2(t-s)(s-r)}{t-r}\le t-r <T =:\Theta
\]
to see that 
\begin{equation}\label{E:Exp-Poly2}
(1+|\xi|^2)^{1-\alpha} e^{-\frac{2(t-s)(s-r)}{t-r} |\xi|^2} \leq 
C\left[\frac{t-r}{(t-s)(s-r)}\right]^{1-\alpha}.
\end{equation} 
Hence, by the Beta integral and condition \eqref{E:Dalang2},
\begin{align*}
 \Norm{Q_{r,z}(t,x)}_{2p}^2 &\le 
 C\: G(t-r, x-z)^2 \int_r^t\ud s 
\left[\frac{t-r}{(t-s)(s-r)}\right]^{1-\alpha}\int_{\RR^d} 
\frac{\widehat{f}(\ud\xi)}{(1+|\xi|^2)^{1-\alpha}}\\
 &= C G(t-r, x-z)^2 (t-r)^\alpha,
\end{align*}
which proves \eqref{E:Q Lp}. 
Therefore, by Minkowski's inequality, we have that
\begin{align*}
\Norm {\sup_{|\xi|=1}I_{\epsilon}^{(3)}(\xi)}_{p}&\le
 C \int_{t-\epsilon}^t \ud r\iint_{\calS^2} \ud z\ud z'\: 
 \Norm{Q_{r,z}(t,x_j)}_{2p}
 \Norm{Q_{r,z'}(t,x_j)}_{2p} f(z-z')\\
&\leq  C \int_{t-\epsilon}^t \ud r\: (t-r)^{\alpha} \iint_{ 
\calS^2}
\ud z\ud z' \: 
G(t-r,z) G(t-r,z') f(z-z') \\
&\le C \epsilon^{\alpha} \widetilde{V}_d(\epsilon).
\end{align*}

\bigskip
Finally, by choosing $\eta\in (0,\alpha\wedge \beta)$,
we have that
\begin{align*}
\bbP \left( \left(\det \sigma \right)^{1/d} < 
\widetilde{V}_d(\epsilon)\epsilon^{\eta}\right)
\leq & \bbP \left(\frac{4}{9} \inf_{|\xi|=1} |I_{\epsilon}^{(0)}(\xi)|< 
2\widetilde{V}_d(\epsilon)\epsilon^{\eta} \right) + \sum_{i=1}^{3} \bbP 
\left(2\sup_{|\xi|=1} |I_{\epsilon}^{(i)}(\xi)| > \frac{1}{3} 
\widetilde{V}_d(\epsilon)\epsilon^{\eta}\right)\\
\leq & \bbP 
\left(\inf_{x\in K} f_{\beta,\gamma}(u(t,x))< 5\epsilon^{\eta} \right) + C 
\epsilon^{p (\alpha-\eta)}\,,
\end{align*}
where $f_{\beta,\gamma}$ is defined in \eqref{E:f_gb}.
Notice that for any $\theta$ and $x\in (0,1)$,
\[
\exp\left\{-2\beta\left[\log \frac{1}{x}\right]^\gamma\right\}
< \theta \quad\Longleftrightarrow\quad
x <
\exp\left\{-(2\beta)^{-1/\gamma}\left[\log 
\frac{1}{\theta}\right]^{1/\gamma}\right\}.
\]
Hence, as $\epsilon$ is small enough, for some constant $C_0>0$,
\begin{align*}
\bbP\left(\inf_{x\in K} f_{\beta,\gamma}(u(t,x))< 5 \: \epsilon^{\eta} \right)
&= \bbP\left(\left(\inf_{x\in K} u(t,x) \wedge 1\right)< 
\exp\left\{-C_0\left(\frac{\eta}{2\beta}\right)^{\frac{1}{\gamma}}\left[
\log\frac{1}{\epsilon}\right]^{\frac{1}{\gamma}}\right\}\right)\\
&= \bbP\left(\inf_{x\in K} u(t,x) < 
\exp\left\{-C_0\left(\frac{\eta}{2\beta}\right)^{\frac{1}{\gamma}}\left[
\log\frac{1}{\epsilon}\right]^{\frac{1}{\gamma}}\right\}\right)\\
&\le \exp\left(-C [\log(1/\epsilon)]^{\frac{1+\alpha}{\gamma}}\right),
\end{align*}
where in the last step we have applied Theorem \ref{T:NegMom}.
Therefore,
\begin{align}\label{E:multiPoint}
 \bbP \left( \left(\det \sigma \right)^{1/d} < 
\widetilde{V}_d(\epsilon)\epsilon^{\eta}\right)\le
 \exp\left(-C [\log(1/\epsilon)]^{\frac{1+\alpha}{\gamma}}\right) + C 
\epsilon^{p (\alpha-\eta)}.
\end{align}

Finally, by Lemma \ref{L:VTildeV} and part (4) of Lemma \ref{L:Rate}, $\widetilde{V}_d \ge C V_d(\epsilon)\ge C' \epsilon^{1-\beta}$.
From \eqref{E:multiPoint} we see that as $\epsilon$ small enough,
\begin{align*}
\bbP \left( \left(\det \sigma \right)^{1/d} <  
\epsilon^{1+\eta-\beta}\right)&\le 
 \bbP \left( \left(\det \sigma \right)^{1/d} < R(\epsilon)\epsilon^{\eta}\right)
 \le
 \exp\left(-C [\log(1/\epsilon)]^{\frac{1+\alpha}{\gamma}}\right) + C \epsilon^{p (\alpha-\eta)}.
\end{align*}
Because $\gamma< 1+\alpha$, an application of Lemma A.1 of \cite{CHN16Density} shows that $\E\left[(\det \sigma)^{-p}\right]< \infty$ for all $p >0$. 
Hence, Theorem \ref{T:Density2}, together with Proposition \ref{P:D1}, implies that for the choice of 
$\calT$ and $\calS$, both parts (a) and (b) of Theorem \ref{T:Mult} hold. 
This completes the whole proof of Theorem \ref{T:Mult}.\myEnd

\section{Strict positivity of density}\label{S:Pos}
The aim of this section is to prove the positivity of the joint density as 
stated in Theorem \ref{T:Pos}. Throughout this section, we will fix a set of 
arbitrary $m$ disjoint points $\{x_1,\dots,x_m\}\subseteq\R^d$.
We will assume that the initial data $\mu$ is nonnegative, and if not, one may 
simply replace $\mu$ by  $|\mu|$. All arguments go through. 

The outline of the proof of Theorem \ref{T:Pos} is given in Section \ref{SS:Pos}.
All technical details are given in the subsequent sections. 

\subsection{Proof of Theorem \ref{T:Pos}}
\label{SS:Pos}

The proof of Theorem \ref{T:Pos} follows the same arguments as those in the proof of Theorem 1.4 of \cite{CHN16Density}. 
For completeness, we present this proof below.

\begin{proof}[Proof of Theorem \ref{T:Pos}]
Choose and fix an arbitrary final time $T$. We will prove Theorem \ref{T:Pos}
for $t=T$.
Throughout the proof, we fix $\kappa>0$ and assume that $|\mathbf{z}|\le \kappa$ where $\mathbf{z} = (z_1, \dots, z_m)\in\R^m$ and $t\in (0,T]$.
Without loss of generality, one may assume that $T>1$ in order that $T-2^{-n}>0$ for all $n\ge 1$.
Otherwise we simply replace all ``$n\ge 1$" in the proof below by ``$n\ge N$" for some large $N>0$.
The proof consists of three steps.

{\bigskip\bf\noindent Step 1.~}
For $n\ge 1$, define $\mathbf{h}_n$ as follows:
\begin{align}\label{E:hni}
h_{n}^i(s,y):= c_n \one_{[T-2^{-n},T]}(s) G(T-s,x_i-y), \qquad\text{for $1\le i\le m$,}
\end{align}
where
\begin{align}
\label{E:cni-1}
\begin{aligned}
c_n^{-1} & := \int_{T-2^{-n}}^T \ud s \iint_{\R^{2d}} \ud y\ud y'\: G(T-s,x_i-y)
G(T-s,x_i-y')f(y-y')\\
&= \int_0^{2^{-n}} \ud s \iint_{\R^{2d}} \ud y\ud y' \: G(s,y)G(s,y')f(y-y').
\end{aligned}
\end{align}
Under Assumption \ref{A:Rate-Sharp}, 
\begin{align}\label{E:cn}
c_n =V(2^{-n})^{-1}\asymp 2^{(1-\beta)n},
\end{align}
or equivalently,
\begin{align}\label{E:Rate-CV}
c_n2^{-n}\le  C 2^{-\beta n} \quad \text{and}\quad V_d(2^{-n}) \le C 
2^{-(1-\beta)n}\qquad \text{for all $n\in\bbN$,} 
\end{align}
see parts (2) and (3) of Lemma \ref{L:Rate}.

Let $\widehat{W}^n_{\mathbf{z}}$ be the cylindrical Wiener process translated 
by $\mathbf{h}_{n}$
and $\mathbf{z}$.
Let $\{\widehat{u}_{\mathbf{z}}^n(t,x), \:(t,x)\in (0,T]\times\R^d\}$ be
the random field shifted with respect to $\widehat{W}^n_{\mathbf{z}}$,
that is, $ \widehat{u}_{\mathbf{z}}^n(t,x) $   satisfies the following equation:
\begin{align}\label{E:hatUn}
\begin{aligned}
 \widehat{u}_{\mathbf{z}}^n(t,x) = &J_0(t,x)+ \int_0^t\int_{\R^d} G(t-s,x-y) \rho(\widehat{u}_{\mathbf{z}}^n(s,y))W(\ud s\ud y)\\
 &+\int_{0}^t\iint_{\R^{2d}} G(t-s,x-y)\rho(\widehat{u}_{\mathbf{z}}^n(s,y)) \InPrd{\mathbf{z},\mathbf{h}_n(s,y')}f(y-y')\ud s\ud y\ud y'.
\end{aligned}
\end{align}

For $x\in\R^d$, denote the gradient vector and the Hessian matrix of $\widehat{u}_{\mathbf{z}}^n(t,x)$ by
\begin{align}\label{E:uij}
\widehat{u}^{n,i}_{\mathbf{z}}(t,x) := \partial_{z_i} \:\widehat{u}_{\mathbf{z}}^n (t,x)
\quad\text{and}\quad
\widehat{u}^{n,i,k}_{\mathbf{z}}(t,x) := \partial^2_{z_i z_k} \:\widehat{u}_{\mathbf{z}}^n (t,x),
\end{align}
respectively.
From \eqref{E:hatUn}, we see that
\[
\widehat{u}^n_{\mathbf{z}}(s,y)=u(s,y) \quad\text{for $s\le T-2^{-n}$ and $y\in\R^d$.}
\]
Hence, $\{\widehat{u}^{n,i}_{\mathbf{z}}(t,x),(t,x)\in(0,T]\times \R^d\}$ satisfies
\begin{align}
\label{E:hatUni} 
\widehat{u}^{n,i}_{\mathbf{z}}(t,x) =&
\theta_{\mathbf{z}}^{n,i}(t,x) \\
& +
\notag
\one_{\{t>T-2^{-n}\}}\int_{T-2^{-n}}^t \int_{\R^d} G(t-s,x-y)\rho'(\widehat{u}_{\mathbf{z}}^n(s,y))
\widehat{u}^{n,i}_{\mathbf{z}}(s,y) W(\ud s\ud y)\\
 &
 +\int_{0}^t \iint_{\R^{2d}} G(t-s,x-y)\rho'(\widehat{u}_{\mathbf{z}}^n(s,y))
\widehat{u}^{n,i}_{\mathbf{z}}(s,y)\InPrd{\mathbf{z},\mathbf{h}_n(s,y')} f(y-y')\ud s\ud y\ud y',
\notag
\end{align}
where
\begin{align}\label{E:thetazni}
\theta_{\mathbf{z}}^{n,i}(t,x) = \int_0^t \iint_{\R^{2d}} G(t-s,x-y)\rho(\widehat{u}_{\mathbf{z}}^n(s,y)) h_n^i(s,y')f(y-y')\ud s\ud y\ud y'.
\end{align}
Similarly, $\{\widehat{u}^{n,i,k}_{\mathbf{z}}(t,x),(t,x)\in(0,T]\times \R^d\}$ satisfies
\begin{align}
\label{E:hatUnik}
 \widehat{u}^{n,i,k}_{\mathbf{z}}(t,x) =&\theta^{n,i,k}_{\mathbf{z}}(t,x)
 \\
 &\notag
 +\int_0^t\iint_{\R^{2d}} G(t-s,x-y)\rho'(\widehat{u}^{n}_{\mathbf{z}}(s,y)) \widehat{u}^{n,i}_{\mathbf{z}}(s,y) h_n^k(s,y')f(y-y')\ud s\ud y\ud y'\\
\notag
 &+\one_{\{t>T-2^{-n}\}}\int_{T-2^{-n}}^t\int_{\R^d} G(t-s,x-y)\rho''(\widehat{u}^{n}_{\mathbf{z}}(s,y)) \widehat{u}^{n,i}_{\mathbf{z}}(s,y)\widehat{u}^{n,k}_{\mathbf{z}}(s,y) W(\ud s\ud y)\\
\notag
 &+\int_{0}^t\iint_{\R^d} G(t-s,x-y)\rho''(\widehat{u}^{n}_{\mathbf{z}}(s,y)) \widehat{u}^{n,i}_{\mathbf{z}}(s,y)\widehat{u}^{n,k}_{\mathbf{z}}(s,y) \InPrd{\mathbf{z},\mathbf{h}_n(s,y')}f(y-y')\ud s\ud y\ud y'\\
\notag
 &+\one_{\{t>T-2^{-n}\}}\int_{T-2^{-n}}^t\int_{\R^d} G(t-s,x-y)\rho'(\widehat{u}^{n}_{\mathbf{z}}(s,y)) \widehat{u}^{n,i,k}_{\mathbf{z}}(s,y) W(\ud s\ud y)\\
\notag
 &+\int_{0}^t\iint_{\R^{2d}} G(t-s,x-y)\rho'(\widehat{u}^{n}_{\mathbf{z}}(s,y)) \widehat{u}^{n,i,k}_{\mathbf{z}}(s,y) \InPrd{\mathbf{z},\mathbf{h}_n(s,y')}f(y-y')\ud s\ud y\ud y',
\end{align}
where
\begin{align}
\label{E:thetaznik}
\begin{aligned}
\theta^{n,i,k}_{\mathbf{z}}(t,x)&:=\partial_{z_k}\theta^{n,i}_{\mathbf{z}}(t,x)\\
&=
\int_{0}^t \iint_{\R^{2d}} G(t-s,x-y)\rho'(\widehat{u}_{\mathbf{z}}^n(s,y)) \widehat{u}_{\mathbf{z}}^{n,k}(s,y) h_n^i(s,y')f(y-y')\ud s\ud y\ud y'.
\end{aligned}
\end{align}
Note that the second term on the right-hand side of \eqref{E:hatUnik} is equal to
$\theta^{n,k,i}_{\mathbf{z}}(t,x)$.

Denote
\begin{align}
\label{E:C2Norm}
\begin{aligned}
 &\hspace{-3em}\Norm{\left\{\widehat{u}^{n}_{\mathbf{z}}(t,x_i)\right\}_{1\le i\le m}}_{C^2}\\
&=
\left|\left\{\widehat{u}^{n}_{\mathbf{z}}(t,x_i)\right\}_{1\le i\le m}\right|
+
\Norm{\left\{\widehat{u}^{n,i}_{\mathbf{z}}(t,x_j)\right\}_{1\le i,j\le m}}
+
\Norm{\left\{\widehat{u}^{n,i,k}_{\mathbf{z}}(t,x_j)\right\}_{1\le i,j,k\le m}}.
\end{aligned}
\end{align}

Suppose that
$\mathbf{y}\in\R^m$ belongs to the interior
of the support of the joint law of
\[
(u(T,x_1),\dots, u(T,x_m))
\]
and $\rho(y_i)\ne 0$ for all $i=1,\dots,m$.
By Theorem \ref{T:Criteria},
Theorem \ref{T:Pos}  is proved once
we show that
there exist some positive constants $c_1$, $c_2$, $r_0$, and $\kappa$ such that
the following two conditions are satisfied:
\begin{align}\label{E:(i)}
 \liminf_{n\rightarrow\infty}\bbP\Bigg(\left| \left\{u(T,x_i)-y_i\right\}_{1\le i\le m}\right| \le r\:\:\text{and}\:\:
 \left| \det\left[\left\{\widehat{u}_0^{n,i}(T,x_j)\right\}_{1\le i,j\le m}\right]\right|\ge c_1\Bigg)>0,
\end{align}
for all $r\in (0,r_0]$, and
\begin{align}\label{E:(ii)}
\lim_{n\rightarrow\infty}\bbP\left(
\sup_{|\mathbf{z}|\le \kappa}
\Norm{\left\{\widehat{u}^{n}_{\mathbf{z}}(T,x_i)\right\}_{1\le i\le m}}_{C^2}
\le c_2
\:\:\Bigg|\:\:
\left|\left\{u(T,x_i)-y_i\right\}_{1\le i\le m}\right|\le r_0
\right) =1.
\end{align}

These two conditions are verified in the following two steps.
\bigskip

{\bigskip\noindent\bf Step 2.~} Let $\mathbf{y}$ be a point in the intersection 
of $\{\rho\ne 0\}^m$  and
the interior of the support of the joint law of $(u(T,x_1),\dots,u(T,x_m))$.
Then there exists $r_0\in (0,1)$ such that for all $0<r\le r_0$,
\[
\mathbb{P}\left(
\Big\{(u(T,x_1),\dots, u(T,x_d))\in B(\mathbf{y},r)\Big\}
\cap
\Big\{\prod_{i=1}^m \left|\rho(u(T,x_i))\right|\ge 2c_0\Big\}\right)>0,
\]
where
\[
c_0 = \frac{1}{2}\inf_{(z_1,\dots,z_d) \in B(\mathbf{y},r_0)}\prod_{i=1}^m \left|\rho(z_i)\right|.
\]
Due to \eqref{E2:UniU-Close} below, it holds that
\[
\lim_{n\rightarrow\infty}
\det\left[ \left\{\widehat{u}_{0}^{n,i}(T,x_j)\right\}_{1\le i,j\le m}\right]
= \prod_{i=1}^m \rho\left(u(T,x_i)\right)
\qquad\text{a.s.}
\]
Hence, by denoting
\begin{align*}
A&:=
\Big\{(u(T,x_1),\dots, u(T,x_m))\in B(\mathbf{y},r)\Big\},\\
D &:=\Big\{\prod_{i=1}^m \left|\rho(u(T,x_i))\right|\ge 2c_0\Big\},\\
E_n&:=\left\{
\left|\det\left[ \left\{\widehat{u}_{0}^{n,i}(T,x_j)\right\}_{1\le i,j\le m}\right]
- \prod_{i=1}^m \rho\left(u(T,x_i)\right)\right| <c_0\right\},\\
G_n & :=\Big\{\left|\det\left[\{\widehat{u}_{0}^{n,i}(T,x_j)\}_{1\le i,j\le m}\right]\right|
\ge c_0\Big\},
\end{align*}
we see that
\begin{align*}
 \bbP\left(A\cap G_n\right)
 &  \ge
 \bbP\left(A\cap D \cap E_n\right) \rightarrow \bbP(A\cap D)>0,\quad\text{as $n\rightarrow\infty$.}
\end{align*}
Therefore,
\begin{align*}
& \liminf_{n\rightarrow\infty}\mathbb{P}\left(\Big\{(u(T,x_1),\dots, u(T,x_m))\in B(\mathbf{y},r)\Big\}
\cap \Big\{\left|\det\left[\{\widehat{u}_{0}^{n,i}(T,x_j)\}_{1\le i,j\le m}\right]\right|
\ge c_0\Big\}\right)>0,
\end{align*}
which proves condition \eqref{E:(i)}.

{\bigskip\noindent\bf Step 3.~}
From \eqref{E:C2Norm}, we see that
\begin{align*}
\Norm{\left\{\widehat{u}^{n}_{\mathbf{z}}(T,x_i)\right\}_{1\le i\le m}}_{C^2}
&\le
\sum_{i=1}^m |\widehat{u}^{n}_{\mathbf{z}}(T,x_i)|
+
\sum_{i,j=1}^m |\widehat{u}^{n,i}_{\mathbf{z}}(T,x_j)|
+
\sum_{i,j,k=1}^m |\widehat{u}^{n,i,k}_{\mathbf{z}}(T,x_j)|.
\end{align*}
By Proposition \ref{P:thetaBdd} below,
there exists some constant $K_{r_0}$ independent of $n$ such that
\[
\lim_{n\rightarrow\infty}\bbP\left(
\sup_{|\mathbf{z}|\le\kappa}\Norm{\left\{\widehat{u}^{n}_{\mathbf{z}}(T,x_i)\right\}_{1\le i\le m}}_{C^2} \le K_{r_0} \:\Bigg|\:
|\{u(T,x_i)-y_i\}_{1\le i\le m}|\le r_0
\right) =1,
\]
where $\kappa$ is fixed as at the beginning of the proof.
Therefore, condition \eqref{E:(ii)} is also satisfied.
This completes the proof of Theorem \ref{T:Pos}.
\end{proof}

\subsection{Properties of the function \texorpdfstring{$\Psi_n(t,x;\ell)$}{}}
\label{SS:Psin}

For $t\in [0,T]$, $x\in\R^d$ and $k\in\bbN$ define 
\begin{align}\label{E:PsiN}
\begin{aligned}
\Psi_n^i(t,x;k) &:=
\int_0^t\iint_{\R^{2d}}  G(t-s, x-y) J_0^k(s,y) h_n^i(s,y')f(y-y')\ud s\ud y\ud 
y',\\
\Psi_n(t,x;k) &:=
\int_0^t\iint_{\R^{2d}}  G(t-s, x-y) J_0^k(s,y) \InPrd{\mathbf{1}, 
\mathbf{h}_n(s,y')} f(y-y')\ud s\ud y\ud y',
\end{aligned}
\end{align}
where $\mathbf{1}=(1,\dots,1)\in\R^m$.
We will use the convention that 
\[
\Psi_n^i(t,x):= \Psi_n^i(t,x;0)\quad\text{and}\quad
\Psi_n(t,x) := \Psi_n(t,x;0).
\]

The aim of this subsection is to prove the following proposition for the 
properties of $\Psi_n(t,x;k)$ and we will use the convention that 
\begin{align}
\label{E:J000}
J_0^k(kt,x)\equiv 1 \quad\text{if $k=0$.}
\end{align}

\begin{proposition}\label{P:Psi}
Under Assumption \ref{A:Continuity}, for all $(t,x)\in (0,T]\times\R^d$, 
$n\in\bbN$, $k\in\{0,1,2,3\}$ and $i\in\{1,\dots,m\}$, the following properties 
hold:
\begin{enumerate}
 \item[(1)] $\Psi_n^i(t,x;k)$ are nonnegative. When $k=0$, it holds that
 \begin{align}\label{E:hG-J0}
\Psi_n^i(t,x)\le \Psi_n^i(T,x_i)\one_{\{t> T-2^{-n}\}} = \one_{\{t> T-2^{-n}\}}.
\end{align}
For $k\in\{1,2,3\}$, there exists some constant $C>0$ independent of $n$ such 
that 
\begin{align}\label{E:hG-J123}
\Psi_n(t,x;k) \le C J_0^k(kt,x)\one_{\{t> T-2^{-n}\}}.
\end{align}
\item[(2)] Under Assumption \ref{A:Continuity}, for any $x\ne x_i$, there exists 
some finite constant $C_x>0$ such that 
\begin{align}
\label{E:hG-Small}
\Psi_n^i(T,x;k)\le C_x c_n 2^{-n}.
\end{align}
Moreover, under Assumption \ref{A:Rate-Sharp}, $\Psi_n^i(T,x;k) \le C_x 2^{-\beta 
n}\rightarrow0$ as $n\rightarrow\infty$.
\item[(3)] For all $n\in\bbN$, $(t,x)\in [T-2^{-n},T]\times\R^d$ and 
$k\in\{0,1,2,3\}$,
there exists some constant $C>0$ independent of $n$ such that 
\begin{multline}\label{E:hG-GGJJ}
 \int_{T-2^{-n}}^t\ud s  \iint_{\R^{2d}} \ud y\ud y'\: G(t-s, x-y)G(t-s, 
x-y')f(y-y')\\
 \times 
 J_0^k(s,y)J_0^k(s,y')  \le C J_0^{2k}(kt,x) V_d(2^{-n}).\qquad
 \end{multline}
\end{enumerate}
\end{proposition}

The following lemma will also be used in order to apply our Picard iterations as 
those in the proof of Lemma \ref{L:Moment-U} below.
\begin{lemma}\label{L:J^k->J}
 For all $n\in\bbN$, $T>1$, $t>0$, $x\in\R^d$ and $k\in\{2,3\}$, it holds that 
 \begin{align}\label{E:J^k->J}
 J_0^k(kt,x)\one_{[T-2^{-n},T]}(t) \le C \one_{\{t> T-2^{-n}\}} \int_{\R^d} 
G(t,x-y) J_0^k(kT,y)\ud y<\infty.
 \end{align}
\end{lemma} 

We need to define the augmented initial measure as follows:
\begin{definition}
Let $\mu$ be a nonnegative measure that satisfies \eqref{E:J0finite}. The {\it 
augmented initial measure}, or the {\it star version} of $\mu$ is a nonnegative 
measure defined as 
\begin{align}\label{E:mu*}
 \mu^*(\ud x)  := & \mu(\ud x) + \left[1+J_0^2(2T,x)+J_0^3(3T,x)\right]\ud x,
\end{align}
where $J_0(t,x)$ is the solution to the homogeneous equation (see \eqref{E:J0}).
Let $J_0^*(t,x)$ and $\Psi_n^*(t,x;k)$ denote the corresponding star versions of 
$J_0(t,x)$ and $\Psi_n(t,x;k)$, respectively: 
\begin{align}\label{E:J*}
 J_0^*(t,x) := &\int_{\R^d} G(t,x-y)\mu^*(\ud y),\\
 \label{E:Psi*}
 \Psi_n^*(t,x;k) :=&
\int_0^t\iint_{\R^{2d}}  G(t-s, x-y) J_0^*(s,y)^k \InPrd{\mathbf{1}, 
\mathbf{h}_n^i(s,y')} f(y-y')\ud s\ud y\ud y'.
\end{align}
\end{definition}

By Lemma \ref{L:J^k->J}, $\mu^*$ is a legal initial measure (that is, it 
satisfies \eqref{E:J0finite}).
For a given initial measure, we may augment it twice, namely, $\mu^{**}$. 
The following facts will be often used, the proofs of which are apparent and 
left for the interested reader.
\begin{lemma}\label{L:Star}
 The following properties hold:
\begin{enumerate}[(1)]
 \item Clearly, $\Psi_n^*(t,x;0)\equiv \Psi_n(t,x;0)$;
 \item $1+J_0(t,x)\le J_0^*(t,x)$;
 \item $\Psi_{n}(t,x;k) \le \Psi_{n}^*(t,x;k)$, for all $k\in\{1,2,3\}$;
 \item $\sum_{k=0}^3 \left( \Psi_n(t,x;k)+J_0^k(kt,x) 2^{-(1-\beta) 
n/2}\one_{\{t> T-2^{-n}\}} \right)\le  C J_0^*(t,x)\one_{\{t> T-2^{-n}\}}$ (due 
to \eqref{E:hG-J0}, \eqref{E:hG-J123} and Lemma \ref{L:J^k->J});
 \item $\Psi_{n}(t,x;k) \le  C \Psi_{n}^*(t,x;1)$ for all $k\in\{0,1,2,3\}$ (due 
to part (4).
\end{enumerate}
\end{lemma}

In the rest part of this subsection, we will prove Proposition \ref{P:Psi} and 
Lemma \ref{L:J^k->J}.

\subsubsection{Proof of part (1) of Proposition \ref{P:Psi}}
The proof consists of the following four steps for $k=0,\dots, 3$. Set 
$\epsilon:=2^{-n}$. \\

{\noindent\bf Step 0.~} In this step, we study the case when $k=0$.
The nonnegativity of $\Psi_n^i$ is clear. 
It is clear that $\Psi_n^i(t,x) \equiv 0$ for $t\le T-\epsilon$. When $t\in 
[T-\epsilon,T]$, we have that
\begin{align*}
 \Psi_n^i(t,x)&=c_n \int_0^{\epsilon-(T-t)}\ud s 
\iint_{\R^{2d}} \ud y\ud y' \: G(s, x-y)G(s+T-t, x_i-y')f(y-y')\\
&=c_n (2\pi)^{-d}\int_0^{\epsilon-(T-t)}\ud s 
\int_{\R^{d}} \widehat{f}(\ud \xi) 
\exp\left(-\left(s+\frac{T-t}{2}\right)|\xi|^2-i(x-x_i)\cdot\xi\right)\\
&\le c_n (2\pi)^{-d}\int_0^{\epsilon-(T-t)}\ud s 
\int_{\R^{d}} \widehat{f}(\ud \xi) 
\exp\left(-\left(s+\frac{T-t}{2}\right)|\xi|^2\right)\\
&\le c_n (2\pi)^{-d}\int_0^{\epsilon-(T-t)}\ud s 
\int_{\R^{d}}  
\exp\left(-s|\xi|^2\right)\widehat{f}(\ud \xi)\\
 &= c_n \int_0^{\epsilon-(T-t)}\ud s 
 \iint_{\R^{2d}} \ud y\ud y' \: G(s,y)G(s,y')f(y-y').
\end{align*}
Therefore, 
\begin{align}\label{E:hG0}
0\le \Psi_n^i(t,x)\le V_d(\epsilon)^{-1} 
\: V_d(\max\left(\epsilon-(T-t), 0)\right).
\end{align}
In particular, the above two inequalities become equalities when $x=x_i$ and 
$t=T$, respectively. This proves \eqref{E:hG-J0}.

{\medskip\noindent\bf Step 1.~} In this step, we prove \eqref{E:hG-J123} for 
$k=1$. We need only prove the case when $t\ge T-\epsilon$. In this case, using 
\eqref{E:GGGG} in the following form
\begin{equation}\label{E:GG}
G(t-s, x-y)G(s, y-z) = G(t, x-z)G\left(\frac{(t-s)s}{t}, y - 
z-\frac{s}{t}(x-z)\right),
\end{equation}  
we can apply similar arguments as above to see that
\begin{align}
\notag
\Psi_n^i(t,x;1) =& \int_0^t \ud s \iiint_{\R^{3d}}\ud y\ud y' \mu(\ud z)\: 
h_n^i(s,y') f(y-y') \\ \notag
&\times G(t,x-z) G\left(\frac{s(t-s)}{t},y-z-\frac{s}{t}(x-z) \right)\\
\notag
=& c_n (2\pi)^{-d}\int_{\R^d}\mu(\ud z) G(t,x-z)
\int_0^{\epsilon+t-T}\ud s \int_{\R^{d}} \widehat{f}(\ud \xi)
\\
&\times  
\exp\bigg(-\frac{1}{2}\left[\frac{s(t-s)}{t}+T-t+s\right]|\xi|^2 \notag \\
& \hspace{4em} - i \left[x_i-z-\frac{t-s}{t}(x-z)\right]\cdot \xi\bigg).
\label{E:TildePsiBd}
\end{align}
Hence, 
\begin{align*}
\Psi_n^i(t,x;1)
\le & c_n (2\pi)^{-d}\int_{\R^d}\mu(\ud z) G(t,x-z)
\int_0^{\epsilon}\ud s 
\int_{\R^{d}} \widehat{f}(\ud \xi) \\
& \times \exp\left(-\frac{1}{2}\left(\frac{s(t-s)}{t}+T-t+s\right)|\xi|^2 
\right)\\
 \le  & c_n (2\pi)^{-d} J_0(t,x) 
\int_0^{\epsilon}\ud s 
\int_{\R^{d}} \widehat{f}(\ud \xi) 
\exp\left(-\frac{1}{2}\left(\frac{s(t-s)}{t}+s\right)|\xi|^2 \right).
\end{align*}
Because $T>1$, when $n$ is sufficiently large, say $n\ge 2$ (which implies 
$\epsilon\in (0,1/4]$), the $\ud s$-integral satisfies that 
\begin{align}\label{E_:FracToLin}
s\le \epsilon < \frac{T-\epsilon}{2}\le \frac{t}{2} \quad\Longrightarrow\quad 
\frac{s(t-s)}{t}\ge \frac{s}{2}.
\end{align}
Hence, 
\begin{align*}
\Psi_n^i(t,x;1)\le & 
c_n (2\pi)^{-d} J_0(t,x) \int_0^{\epsilon}\ud s 
\int_{\R^{d}} \widehat{f}(\ud \xi) \exp\left(-\frac{3}{4}s|\xi|^2 
\right)=\frac{4}{3} c_n J_0(t,x) V_d(3\epsilon/4),\notag
\end{align*}
where the last equality is due to \eqref{E:Def-Vd2}.
Finally, an application of Lemma \ref{L:V/V} proves 
\eqref{E:hG-J123} for $k=1$.

{\medskip\noindent\bf Step 2.~} Now we study the case $k=2$.
In this case, 
\begin{align*}
\Psi_n^i(t,x;2) =& \int_0^t\ud s\iint_{\R^{2d}}\mu(\ud z_1)\mu(\ud 
z_2)\iint_{\R^{2d}}\ud y\ud y' h_n^i(s,y')f(y-y')\\
&\times G(t-s,x-y) G(s,y-z_1)G(s,y-z_2).
\end{align*}
Then we apply the following bounds 
\begin{align}\label{E:GG2}
\begin{aligned}
G(s,y-z) & \le 2^{d/2}G(2s,y-z),\quad\text{and}\\
G(t-s,x-y) &= 2^{3d/2} [\pi(t-s)]^{d/2} G(2(t-s),x-y)^2
\end{aligned}
\end{align}
to turn the three $G$'s into two pairs of $G$'s: 
\begin{align*}
\Psi_n^i(t,x;2) \le& C \int_0^t\ud s (t-s)^{d/2}\iint_{\R^{2d}}\mu(\ud 
z_1)\mu(\ud z_2)\iint_{\R^{2d}}\ud y\ud y' h_n^i(s,y')f(y-y')\\
&\times G(2(t-s),x-y) G(2s,y-z_1)\cdot G(2(t-s),x-y) G(2s,y-z_2).
\end{align*}
Then apply \eqref{E:GG} for these two pairs of $G$'s to see that 
\begin{align*}
\Psi_n^i(t,x;2) \le& C \int_0^t\ud s (t-s)^{d/2}\iint_{\R^{2d}}\mu(\ud 
z_1)\mu(\ud z_2)\iint_{\R^{2d}}\ud y\ud y' h_n^i(s,y')f(y-y')\\
&\times G(2t,x-z_1) G\left(\frac{2s(t-s)}{t},y-z_1-\frac{s}{t}(x-z_1)\right)\\
&\times G(2t,x-z_2) G\left(\frac{2s(t-s)}{t},y-z_2-\frac{s}{t}(x-z_2)\right).
\end{align*}
Then apply \eqref{E:GGGG} and the definition of $h_n^i$ in \eqref{E:hni} to see 
that 
\begin{align*}
\Psi_n^i(t,x;2) \le& C c_n \int_{T-\epsilon}^t \ud s 
(t-s)^{d/2}\iint_{\R^{2d}}\mu(\ud z_1)\mu(\ud z_2) G(2t,x-z_1) G(2t,x-z_2)\\
&\times 
\iint_{\R^{2d}}\ud y\ud y' f(y-y') G(T-s,x_i-y')\\
&\times  
G\left(\frac{s(t-s)}{t},y-\frac{z_1+z_2}{2}-\frac{s}{2t}(2x-z_1-z_2)\right)\\
&\times  G\left(\frac{4s(t-s)}{t},z_1-z_2+\frac{s}{t}(z_1-z_2)\right).
\end{align*}
Then bound the last $G(t,x)$ by $C t^{-d/2}$ to see that
\begin{align*}
\Psi_n^i(t,x;2) \le& C c_n \int_{T-\epsilon}^t \ud s 
(t/s)^{d/2}\iint_{\R^{2d}}\mu(\ud z_1)\mu(\ud z_2) G(2t,x-z_1) G(2t,x-z_2)\\
&\times 
\iint_{\R^{2d}}\ud y\ud y' f(y-y') G(T-s,x_i-y')\\
&\times  
G\left(\frac{s(t-s)}{t},y-\frac{z_1+z_2}{2}-\frac{s}{2t}(2x-z_1-z_2)\right).
\end{align*}
For $T>1$ and $n$ large enough, we have that $T-2^{-n}\ge 1/2$ and hence, we can 
bound $s^{-d/2}$ from above by $2^{d/2}$,  so that we can get rid of the factor 
$(t/s)^{d/2}$. Then by Fourier transform, we see that 
\begin{align}
\notag
\Psi_n^i(t,x;2) \le& C c_n \int_{T-\epsilon}^t \ud s \iint_{\R^{2d}}\mu(\ud 
z_1)\mu(\ud z_2) G(2t,x-z_1) G(2t,x-z_2) \int_{\R^{d}} \widehat{f}(\ud \xi)  \\ 
\notag
&\times 
\exp\left(-\frac{1}{2}\left[\frac{s(t-s)}{t}+T-s\right]|\xi|^2-i\left(\frac{
z_1+z_2}{2}+\frac{s}{2t}(2x-z_1-z_2)\right)\cdot \xi\right)\\ \notag
=& C c_n \int_{0}^{\epsilon+t-T} \ud s \iint_{\R^{2d}}\mu(\ud z_1)\mu(\ud z_2) 
G(2t,x-z_1) G(2t,x-z_2) \int_{\R^{d}} \widehat{f}(\ud \xi)  \\ \notag
&\times 
\exp\bigg(-\frac{1}{2}\left[\frac{s(t-s)}{t}+T-t+s\right]|\xi|^2\\
&\hspace{4em}-i\left(x_i-\frac{z_1+z_2}{2}-\frac{t-s}{2t}
(2x-z_1-z_2)\right)\cdot \xi\bigg).
\label{E2:TildePsiBd}
\end{align}
By \eqref{E_:FracToLin},
\begin{align*}
\Psi_n^i(t,x;2) \le
& C c_n \int_{0}^{\epsilon+t-T} \ud s \iint_{\R^{2d}}\mu(\ud z_1)\mu(\ud z_2) 
G(2t,x-z_1) G(2t,x-z_2) 
\\
&\times \int_{\R^{d}} \widehat{f}(\ud \xi)  
\exp\left(-\frac{3}{4}s |\xi|^2\right)\\
\le & C c_nJ_0^2(2t,x) V_d(3\epsilon/4).
\end{align*}
Finally, this case (that is, $k=2$) is proved by an application of 
Lemma \ref{L:V/V}.

{\medskip\noindent\bf Step 3.~} Now we prove the case $k=3$. In this case, 
\begin{align*}
\Psi_n^i(t,x;3) =& \int_0^t\ud s\iiint_{\R^{3d}}\mu(\ud z_1)\mu(\ud z_2)\mu(\ud 
z_3)\iint_{\R^{2d}}\ud y\ud y' h_n^i(s,y')f(y-y')\\
&\times G(t-s,x-y) G(s,y-z_1)G(s,y-z_2)G(s,y-z_3).
\end{align*}
Then we apply the following bounds 
\begin{align}\label{E:GG3}
\begin{aligned}
G(s,y-z) & \le 3^{d/2}G(3s,y-z),\quad\text{and}\\
G(t-s,x-y) &= 3^{3d/2} [2\pi(t-s)]^{d} G(3(t-s),x-y)^3
\end{aligned}
\end{align}
to turn the four $G$'s into three pairs of $G$'s: 
\begin{align*}
\Psi_n^i(t,x;3) \le& C \int_0^t\ud s (t-s)^{d}\iiint_{\R^{3d}}\mu(\ud 
z_1)\mu(\ud z_2)\mu(\ud z_3) 
\iint_{\R^{2d}}\ud y\ud y' h_n^i(s,y')f(y-y')\\
&\times \prod_{i=1}^3   G(3(t-s),x-y) G(3s,y-z_i) .
\end{align*}
Then apply \eqref{E:GG} for these three pairs of $G$'s to see that 
\begin{align*}
\Psi_n^i(t,x;3) \le& C \int_0^t\ud s (t-s)^{d}\iiint_{\R^{3d}}\mu(\ud 
z_1)\mu(\ud z_2)\mu(\ud z_3)\iint_{\R^{2d}}\ud y\ud y' h_n^i(s,y')f(y-y')\\
&\times \prod_{i=1}^3 G(3t,x-z_i) 
G\left(\frac{3s(t-s)}{t},y-z_i-\frac{s}{t}(x-z_i)\right).
\end{align*}
Then apply \eqref{E:GGGG} and the definition of $h_n^i$ in \eqref{E:hni} to see 
that 
\begin{align*}
\Psi_n^i(t,x;3) \le& C c_n \int_{T-\epsilon}^t \ud s 
(t-s)^{d}\iiint_{\R^{3d}}\mu(\ud z_1)\mu(\ud z_2) \mu(\ud z_3)  \prod_{i=1}^3 
G(3t,x-z_i) \\
&\times 
\iint_{\R^{2d}}\ud y\ud y' f(y-y') G(T-s,x_i-y')\\
&\times \prod_{i=1}^3 G\left(\frac{3s(t-s)}{t},y-z_i-\frac{s}{t}(x-z_i)\right).
\end{align*}
Notice by \eqref{E:GG2} that
\begin{align}\notag
&\prod_{i=1}^3 G\left(\frac{3s(t-s)}{t},y-z_i-\frac{s}{t}(x-z_i)\right)\\ \notag
=&G\left(\frac{3s(t-s)}{2t},y-\frac{z_1+z_2}{2}-\frac{s}{2t}
(2x-z_1-z_2)\right)\\ \notag
&\times  G\left(\frac{6s(t-s)}{t},z_1-z_2+\frac{s}{t}(z_1-z_2)\right)\\ \notag
&\times G\left(\frac{3s(t-s)}{t},y-z_3-\frac{s}{t}(x-z_3)\right)\\ \notag
\le & C \left(\frac{s(t-s)}{t}\right)^{-d/2} 
G\left(\frac{3s(t-s)}{t},y-\frac{z_1+z_2}{2}-\frac{s}{2t}(2x-z_1-z_2)\right)\\ 
\notag
&\times G\left(\frac{3s(t-s)}{t},y-z_3-\frac{s}{t}(x-z_3)\right)\\
\le & C \left(\frac{s(t-s)}{t}\right)^{-d} 
G\left(\frac{3s(t-s)}{2t},y-\frac{z_1+z_2+2z_3}{4}-\frac{s}{4t}
(4x-z_1-z_2-2z_3)\right).
\label{E:Prod3G's}
\end{align}
Hence,
\begin{align*}
\Psi_n^i(t,x;3) \le& C c_n \int_{T-\epsilon}^t \ud s (t/s)^{d} 
\iiint_{\R^{3d}}\mu(\ud z_1)\mu(\ud z_2)\mu(\ud z_3)\prod_{i=1}^d G(3t,x-z_i) \\
&\times 
\iint_{\R^{2d}}\ud y\ud y' f(y-y') G(T-s,x_i-y')\\
&\times  
G\left(\frac{3s(t-s)}{2t},y-\frac{z_1+z_2+2z_3}{4}-\frac{s}{2t}
(4x-z_1-z_2-2z_3)\right).
\end{align*}
By the same reasoning as before, we can get rid of the factor $(t/s)^d$ and then 
by Fourier transform 
we see that 
\begin{align}
\notag
\Psi_n^i(t,x;3) \le& C c_n \int_{T-\epsilon}^t \ud s \iiint_{\R^{2d}}\mu(\ud 
z_1)\mu(\ud z_2)\mu(\ud z_3)\prod_{i=1}^3 G(3t,x-z_i) \int_{\R^{d}} 
\widehat{f}(\ud \xi)  \\ \notag
&\times 
\exp\bigg(-\frac{1}{2}\left[\frac{3s(t-s)}{2t}+T-s\right]|\xi|^2\\ \notag
&\qquad -i\left(\frac{z_1+z_2+2z_3}{4}+\frac{s}{4t}(4x-z_1-z_2-2z_3)\right)\cdot 
\xi\bigg)\\ \notag
=& C c_n \int_{0}^{\epsilon+t-T} \ud s \iiint_{\R^{3d}}\mu(\ud z_1)\mu(\ud 
z_2)\mu(\ud z_3)\prod_{i=1}^d G(3t,x-z_i) \int_{\R^{d}} \widehat{f}(\ud \xi)  \\ 
\notag
&\times 
\exp\bigg(-\frac{1}{2}\left[\frac{3s(t-s)}{2t}+T-t+s\right]|\xi|^2\\
&\hspace{4em}-i\left(x_i-\frac{z_1+z_2+2z_3}{4}-\frac{t-s}{4t}
(4x-z_1-z_2-2z_3)\right)\cdot \xi\bigg).
\label{E3:TildePsiBd}
\end{align}
By \eqref{E_:FracToLin},
\begin{align*}
\Psi_n^i(t,x;3) \le
& C c_n \int_{0}^{\epsilon+t-T} \ud s \iiint_{\R^{3d}}\mu(\ud z_1)\mu(\ud 
z_2)\mu(\ud z_3)\prod_{i=1}^3 G(3t,x-z_i) 
\\
&\times \int_{\R^{d}} \widehat{f}(\ud \xi)  
\exp\left(-\frac{7}{8}s |\xi|^2\right)\\
\le & C c_nJ_0^3(3t,x) V_d(7\epsilon/8).
\end{align*}
Finally, this case (that is, $k=3$) is proved by an application of 
Lemma \ref{L:V/V}. \qed

\subsubsection{Proof of part (2) of Proposition \ref{P:Psi}}
The proof relies on the following lemma. 

\begin{lemma}\label{L:f-LocalBdd}
Suppose $f$ satisfies Assumption \ref{A:Continuity}, that is, $f$ is locally 
bounded on $\R^d\setminus\{0\}$. Then 
for any $x\in\R^d\setminus\{0\}$,
there exists some constant $C_{x}>0$, 
\begin{align}
 \limsup_{\epsilon\downarrow 0} \int_{\R^d}G(\epsilon,y)f(y+ x) < 
C_{x}.
\end{align}
\end{lemma}
\begin{proof}
Notice that 
\begin{align*}
\int_{\R^d} G(\epsilon, y) f(y+x)\ud y
&=  
\int_{|y|\le |\Delta x|/2} G(\epsilon, y) f(y+x)\ud y 
+
\int_{|y|> |\Delta x|/2} G(\epsilon, y) f(y+x)\ud y\\
&=: I_{0,1}(\epsilon) + 
I_{0,2}(\epsilon).
\end{align*}
Since $f$ is a locally bounded function on 
$\R^d\setminus\{0\}$, we see that
\[
\int_{\R^d} G(1, z) 
\one_{\{|z|\le |x|/(2\sqrt{\epsilon})\}}f\left(\sqrt{\epsilon} \: z + 
x \right) \ud z \le 
\left(\sup_{\frac{|x|}{2}\le |z| \le \frac{3|x|}{2} }
f(z)\right)
\int_{\R^d} G(1, 
z)  \ud z =: C_{x}.
\]
Hence,
\[
\limsup_{\epsilon \downarrow  0 } I_{0,1}(\epsilon) = 
\limsup_{\epsilon\downarrow 
0} \int_{\R^d} G(1, z) 
\one_{\{|z|\le |x|/(2\sqrt{\epsilon})\}}f\left(\sqrt{\epsilon} \: z + 
x \right) \ud z \le  C_{x}.
\]
As for $I_{0,2}(\epsilon)$, notice that
\begin{align*}
I_{0,2}(\epsilon) 
&= \int_{|y|>|\Delta x|/2} (2\pi\epsilon)^{-d/2} 
\exp\left(-\frac{|y|^2}{4\epsilon} -\frac{|y|^2}{4\epsilon} \right) f(y+\Delta 
x)\ud y\\
&\le 
\int_{\R^d} (2\pi\epsilon)^{-d/2} \exp\left(-\frac{|\Delta x|^2}{ 16\epsilon} 
-\frac{|y|^2}{8\epsilon} \right) f(y+\Delta x)\ud y\\
&= C\exp\left(-\frac{|\Delta x|^2}{ 16 \epsilon}\right)
\int_{\R^d} \exp\left(-2\epsilon|\xi|^2-i\Delta x\cdot \xi\right) 
\widehat{f}(\ud\xi)\\
&\le C\exp\left(-\frac{|\Delta x|^2}{32\epsilon}\right)
\int_{\R^d} \exp\left(-2\epsilon|\xi|^2-\frac{|\Delta x|^2}{ 32\epsilon}\right) 
\widehat{f}(\ud\xi).
\end{align*}
Because $g(\epsilon):= 2\epsilon|\xi|^2+|\Delta x|^2/(32\epsilon)$ achieves its 
global minimum at $\epsilon = \frac{|\Delta x|}{8|\xi|}$ with 
\begin{align}\label{E:min-g}
\min_{\epsilon>0 } g(\epsilon) = g\left(\frac{|\Delta x|}{8|\xi|}\right) = 
\frac{|\Delta x||\xi|}{2} ,
\end{align}
we see that
\begin{align*}
I_{0,2}(\epsilon) &\le 
C \exp\left(-\frac{|\Delta x|^2}{32\epsilon}\right)
\int_{\R^d} \exp\left(-\frac{|\Delta x|}{2} |\xi| \right) 
\widehat{f}(\ud\xi)\rightarrow 0, \quad\text{as $\epsilon\rightarrow 0$,}
\end{align*}
where the integral is finite thanks to Dalang's condition \eqref{E:Dalang}.
Combining the above two terms proves the lemma. 
\end{proof}

We will prove \eqref{E:hG-Small} in four steps. Throughout the proof, $x\ne 
x_i$ and denote $\Delta x:=x-x_i$.

{\bigskip\bf\noindent Step 0.~} We first study the case when $k=0$. Denote 
\[
I_0(\epsilon) := \int_0^{\epsilon}\ud s 
\iint_{\R^{2d}}G(s,x-y)G(s,x_i-y')f(y-y')\ud y \ud y'.
\]
Then  $\Psi_n^i(T,x;0)= c_n I_0(2^{-n})$.
By Fourier transform, we see that
\begin{align*}
I_0(\epsilon) =& (2\pi)^{-d} \int_0^\epsilon \ud s \int_{\R^d} e^{-s|\xi|^2 - i 
(x-x_i)\cdot \xi}\widehat{f}(\ud \xi)\\
&=\int_0^\epsilon \ud s \int_{\R^d} G(2s, y) f(y+\Delta x)\ud y.
\end{align*}
By l'H\^ospital's rule,
\begin{align*}
\limsup_{\epsilon\downarrow  0} \frac{I_0(\epsilon)}{\epsilon} 
&=  \limsup_{\epsilon\downarrow  0 }
\int_{\R^d} G(2\epsilon, y) f(y+\Delta x)\ud y\le C_{\Delta x},
\end{align*}
where the last step is due to Lemma \ref{L:f-LocalBdd}.
Therefore, 
\begin{align}\label{E_:I/epsilon}
\limsup_{\epsilon\downarrow 0} \frac{I_0(\epsilon)}{\epsilon} \le  C_{\Delta x}.
\end{align}
In particular, for some constant $C_x>0$, 
$I_0(\epsilon)\le C_x \epsilon$ for all $\epsilon\in (0,1)$. Clearly, from the 
above limit we can see that if $f$ blows up at $x=0$, this constant $C_x$ will 
blow up at $x=x_i$.
Finally, \eqref{E:hG-Small} is proved by setting $\epsilon = 2^{-n}$.  \\

{\medskip\noindent\bf Step 1.~} 
In this step, we will prove \eqref{E:hG-Small} for $k=1$.  
Denote 
\[
\widetilde{s}:=\frac{s(T-s)}{T}+s\quad\text{and}\quad 
\widetilde{\epsilon}:= \epsilon(T-\epsilon)/T + \epsilon. 
\]
We first prove the case $k=1$. From \eqref{E:TildePsiBd} with $t=T$, we denote 
\begin{align*}
 I_1(\epsilon)&:= 
(2\pi)^{-d}\int_{\R^d}\mu(\ud z) G(T,x-z)
\int_0^{\epsilon}\ud s \int_{\R^{d}} \widehat{f}(\ud \xi)  
\exp\left(-\frac{1}{2}\widetilde{s} |\xi|^2 - i \Delta_{z}^s x \cdot \xi\right)
\end{align*}
with 
\[
\Delta_{z}^s x:= x_i-z-\frac{T-s}{T}(x-z),
\]
so that $\Psi_n^i(T,x;1)= c_n I_1(2^{-n})$.
By Fourier transform, we see that 
\[
I_1(\epsilon) =\int_{\R^d} \mu(\ud z) G(T,x-z) \int_0^\epsilon \ud s 
\int_{\R^d} \ud y \: G\left(\widetilde{s}, y\right) f\left(y+\Delta_{z}^s 
x\right)
\]
By L'H\^ospital's rule,
\begin{align*}
\limsup_{\epsilon\downarrow  0} \frac{I_1(\epsilon)}{\epsilon}  = 
&\limsup_{\epsilon \downarrow 0} I_1'(\epsilon) \\
=& 
\limsup_{\epsilon\downarrow0}\int_{\R^d} \mu(\ud z) G(T,x-z) 
\int_{\R^d} \ud y \: G\left(\widetilde{\epsilon}, y\right) 
f\left(y+\Delta_{z}^\epsilon x\right)\\
=:&\limsup_{\epsilon \downarrow 0} I_1^*(\epsilon).
\end{align*}
Noticing that $\Delta_z^\epsilon x = (x_i-x) -\frac{\epsilon}{T}(z-x)$, by 
change of variable 
$y\leftrightarrow y'=y-\frac{\epsilon}{T}(z-x)$, we see that
\[
I_1^*(\epsilon) = 
\int_{\R^d} \mu(\ud z) G(T,x-z) 
\int_{\R^d} \ud y' \: G\left(\widetilde{\epsilon}, 
y'+\frac{\epsilon}{T}(z-x)\right) 
f\left(y'+x_i-x\right).
\]
Because
\begin{align}\label{E_:GTildeEpsilon}
G\left(\widetilde{\epsilon}, y'+\frac{\epsilon}{T}(z-x)\right) \le
 C G\left(2\widetilde{\epsilon}, y'\right) 
\exp\left(\frac{\epsilon^2 |x-z|^2}{2\widetilde{\epsilon} T^2}\right),
\end{align}
we see that 
\begin{align*}
I_1^*(\epsilon) \le &
\int_{\R^d} \mu(\ud z) G(T,x-z) \exp\left(\frac{\epsilon^2 
|x-z|^2}{2\widetilde{\epsilon} T^2}\right)
\int_{\R^d} \ud y' \: G\left( 2 \widetilde{\epsilon}, y'\right) 
f\left(y'+x_i-x\right).
\end{align*}
Because $\frac{\epsilon^2}{2\tilde{\epsilon}} \to 0$ as $\epsilon \to 0$, when 
$\epsilon$ is sufficiently small, we have that
\begin{align*}
 G(T,x-z) \exp\left(\frac{\epsilon^2 |x-z|^2}{2\widetilde{\epsilon} T^2}\right) 
 \leq C G(2T, x-z),
\end{align*}
which implies that 
\begin{align*}
\limsup_{\epsilon\downarrow 0}
I_1^*(\epsilon) \le & \limsup_{\epsilon\downarrow 0}
C J_0(2T,x) \int_{\R^d} \ud 
y' \: G\left(\widetilde{\epsilon}, y'\right) 
f\left(y'+x_i-x\right)
\\
\le & C_{\Delta x} J_0(2T,x),
\end{align*}
where the last step is due to Lemma \ref{L:f-LocalBdd}.
Therefore, for some constant $C_x>0$, $I_1(\epsilon)\le C_x \epsilon$.  This 
proves \eqref{E:hG-Small} for $k=1$.

{\medskip\noindent\bf Step 2.~} 
Now we study the case $k=2$. From \eqref{E2:TildePsiBd} with $t=T$, we denote
\begin{align*}
 I_2(\epsilon):= &
(2\pi)^{-d}\iint_{\R^{2d}}\mu(\ud z_1)\mu(\ud z_2) G(2T,x-z_1)G(2T,x-z_2)\\
&\times 
\int_0^{\epsilon}\ud s \int_{\R^{d}} \widehat{f}(\ud \xi)  
\exp\left(-\frac{1}{2}\widetilde{s} |\xi|^2 - i \Delta_{z_1,z_2}^s x \cdot 
\xi\right)
\end{align*}
with 
\[
\Delta_{z_1,z_2}^s x:= x_i-\frac{z_1+z_2}{2}-\frac{T-s}{2T}(2x-z_1-z_2),
\]
so that $\Psi_n^i(T,x;2)= c_n I_2(2^{-n})$.
Let $\bar{z}=(z_1+z_2)/2$. Then we can apply the same arguments for $I_1$ with 
$z$ replaced by $\bar{z}$. Indeed, by L'H\^ospital's rule, 
\begin{align*}
\limsup_{\epsilon\downarrow  0} \frac{I_2(\epsilon)}{\epsilon}  = &
\limsup_{\epsilon \downarrow 0} I_2'(\epsilon) \\
=&
\limsup_{\epsilon\downarrow0}\iint_{\R^{2d}} \mu(\ud z_1)\mu(\ud z_2) 
G(2T,x-z_1) G(2T,x-z_2) \\
&\qquad\qquad\times \int_{\R^d} \ud y \: G\left(\widetilde{\epsilon}, 
y\right) f\left(y+\Delta_{z_1,z_2}^\epsilon x\right)\\
=:& \limsup_{\epsilon \downarrow 0} I_2^*(\epsilon).
\end{align*}
Noticing that $\Delta_{z_1,z_2}^\epsilon x = (x_i-x) 
-\frac{\epsilon}{T}\left(\bar{z}-x\right)$, 
by change of variable 
$y\leftrightarrow y'=y-\frac{\epsilon}{T}(\bar{z}-x)$, we see that
\begin{align*}
I_2^*(\epsilon) = &
\iint_{\R^{2d}} \mu(\ud z_1)\mu(\ud z_2) G(2T,x-z_1)  G(2T,x-z_2) \\
&\times \int_{\R^d} \ud y' \: G\left(\widetilde{\epsilon}, 
y'+\frac{\epsilon}{T}(\bar{z}-x)\right) 
f\left(y'+x_i-x\right).
\end{align*}
Then by \eqref{E_:GTildeEpsilon}, we see that
\begin{align*}
I_2^*(\epsilon) \le &
\iint_{\R^{2d}} \mu(\ud z_1)\mu(\ud z_2) G(2T,x-z_1)G(2T,x-z_2) 
\exp\left(\frac{\epsilon^2 |x-\bar{z}|^2}{2\widetilde{\epsilon} T^2}\right)\\
&\times \int_{\R^d} \ud y' \: G\left(\widetilde{\epsilon}, y'\right) 
f\left(y'+x_i-x\right).
\end{align*}
Notice that 
\[
\exp\left(\frac{\epsilon^2 |x-\bar{z}|^2}{2\widetilde{\epsilon} T^2}\right)
\le \exp\left(\frac{\epsilon^2 |x-z_1|^2}{4\widetilde{\epsilon} T^2}\right)
\exp\left(\frac{\epsilon^2 |x-z_2|^2}{4\widetilde{\epsilon} T^2}\right).
\]
Then by the same arguments as for $I_1$, when $\epsilon$ is sufficiently small, 
we have that
\begin{align*}
 G(2T,x-z_1)G(2T,x-z_2) 
\exp\left(\frac{\epsilon^2 |x-\bar{z}|^2}{2\widetilde{\epsilon} T^2}\right)
 \leq C G(4T, x-z_1)G(4T, x-z_2),
\end{align*}
which implies that 
\begin{align*}
\limsup_{\epsilon\downarrow 0}I_2^*(\epsilon) \le & 
\limsup_{\epsilon\downarrow 0}C J_0^2(4T,x) \int_{\R^d} 
\ud y' \: G\left(\widetilde{\epsilon}, y'\right) 
f\left(y'+x_i-x\right)
\\
\le &
C_{\Delta x} J_0^2(4T,x),
\end{align*}
where the last step is due to Lemma \ref{L:f-LocalBdd}.
Therefore, for some constant $C_x>0$, $I_2(\epsilon)\le C_x \epsilon$.  This 
proves \eqref{E:hG-Small} for $k=2$.

{\medskip\noindent\bf Step 3.~} 
Now we study the case $k=3$. From \eqref{E3:TildePsiBd} with $t=T$, we denote
\begin{align*}
 I_3(\epsilon):= &
(2\pi)^{-d}\iiint_{\R^{3d}}\prod_{i=1}^3 \mu(\ud z_i) G(3T,x-z_i)
\int_0^{\epsilon}\ud s \int_{\R^{d}} \widehat{f}(\ud \xi)  
\exp\left(-\frac{1}{2}\widetilde{s} |\xi|^2 - i \Delta_{z's}^s x \cdot 
\xi\right)
\end{align*}
with 
\[
\Delta_{z's}^s x:= x_i-\frac{z_1+z_2+2z_3}{4}-\frac{T-s}{4T}(4x-z_1-z_2-2z_3),
\]
so that $\Psi_n^i(T,x;3)= c_n I_3(2^{-n})$.
Let $\bar{z}=(z_1+z_2+2z_3)/4$. Then we can apply the same arguments for $I_1$ 
with $z$ replaced by $\bar{z}$. Indeed, by L'H\^ospital's rule, 
\begin{align*}
\limsup_{\epsilon\downarrow  0} \frac{I_3(\epsilon)}{\epsilon}  =& 
\limsup_{\epsilon \downarrow 0} I_3'(\epsilon) \\
=&
\limsup_{\epsilon\downarrow0}\iint_{\R^{2d}} \prod_{i=1}^3 \mu(\ud z_i)  
G(3T,x-z_i)  \int_{\R^d} \ud y \: G\left(\widetilde{\epsilon}, y\right) 
f\left(y+\Delta_{z's}^\epsilon x\right)\\
=:&\limsup_{\epsilon\downarrow0} I_3^*(\epsilon).
\end{align*}
Noticing that $\Delta_{z's}^\epsilon x = (x_i-x) 
-\frac{\epsilon}{T}\left(\bar{z}-x\right)$, 
by change of variable 
$y\leftrightarrow y'=y-\frac{\epsilon}{T}(\bar{z}-x)$, we see that
\begin{align*}
I_3^*(\epsilon) = &
\iiint_{\R^{3d}} \prod_{i=1}^3 \mu(\ud z_i) G(3T,x-z_i) \int_{\R^d} \ud y' \: 
G\left(\widetilde{\epsilon}, y'+\frac{\epsilon}{T}(\bar{z}-x)\right) 
f\left(y'+x_i-x\right).
\end{align*}
Then by \eqref{E_:GTildeEpsilon}, we see that
\begin{align*}
I_3^*(\epsilon) \le &
\iiint_{\R^{2d}}  \exp\left(\frac{\epsilon^2 
|x-\bar{z}|^2}{2\widetilde{\epsilon} T^2}\right) \prod_{i=1}^3 \mu(\ud z_i) 
G(2T,x-z_i) \int_{\R^d} \ud y' \: G\left(\widetilde{\epsilon}, y'\right) 
f\left(y'+x_i-x\right).
\end{align*}
Notice that 
\begin{align*}
\left|x-\frac{z_1+z_2+2z_3}{4}\right|^2 
= &\left|\frac{x-z_1}{4} +  \frac{x-z_2}{4} + \frac{x-z_3}{2}\right|^2 \\
\le &\left|x-z_1\right|^2+\left|x-z_2\right|^2 +\left|x-z_3\right|^2,
\end{align*} 
which implies that
\[
\exp\left(\frac{\epsilon^2 |x-\bar{z}|^2}{2\widetilde{\epsilon} T^2}\right)
\le \prod_{i=1}^3 \exp\left(\frac{\epsilon^2 |x-z_i|^2}{2\widetilde{\epsilon} 
T^2}\right).
\]
Then by the same arguments as for $I_1$, when $\epsilon$ is sufficiently small, 
we have that
\begin{align*}
 \prod_{i=1}^3 G(3T,x-z_i)
\exp\left(\frac{\epsilon^2 |x-z_i|^2}{2\widetilde{\epsilon} T^2}\right)
 \leq C \prod_{i=1}^3 G(6T,x-z_i),
\end{align*}
which implies that 
\begin{align*}
\limsup_{\epsilon\downarrow 0}I_3^*(\epsilon) \le & 
\limsup_{\epsilon\downarrow 0} C J_0^3(6T,x) \int_{\R^d} 
\ud y' \: G\left(\widetilde{\epsilon}, y'\right) 
f\left(y'+x_i-x\right)
\\
\le & 
C_{\Delta x} J_0^3(6T,x),
\end{align*}
where the last step is due to Lemma \ref{L:f-LocalBdd}.
Therefore, for some constant $C_x>0$, $I_3(\epsilon)\le C_x \epsilon$.  This 
proves \eqref{E:hG-Small} for $k=3$. \qed

\subsubsection{Proof of part (3) of Proposition \ref{P:Psi}}
In this part, we will prove \eqref{E:hG-GGJJ}. It is clear that the case of 
$k=0$ is a direct consequence of the definition of $c_n^{-1}$ in 
\eqref{E:cni-1}.
In the following, we need only to prove the cases for $k=1,2,3$.  Denote the 
triple integral in \eqref{E:hG-GGJJ} by $I$. 

{\medskip\bf\noindent Step 1.~} we first prove the case when $k=1$.
By \eqref{E:GG}, we see that
\begin{align*}
I = &  \int_{T-2^{-n}}^t \ud s\iint_{\R^{2d}}\mu(\ud z)\mu(\ud z')G(t,x-z) 
G(t,x-z')\iint_{\R^{2d}}\ud y\ud y' f(y-y') \\
&\times G\left(\frac{(t-s)s}{t},y-z-\frac{s}{t}(x-z)\right)
G\left(\frac{(t-s)s}{t},y'-z'-\frac{s}{t}(x-z')\right)\\
\le &  C \int_{T-2^{-n}}^t \ud s\iint_{\R^{2d}}\mu(\ud z)\mu(\ud z')G(t,x-z) 
G(t,x-z')\int_{\R^d} \widehat{f}(\ud 
\xi)\exp\left(-\frac{(t-s)s}{t}|\xi|^2\right)\\
\le & C J^2_0(t,x) \int_{0}^{2^{-n}+t-T}\ud s \int_{\R^d}\widehat{f}(\ud 
\xi)\exp\left(-\frac{(t-s)s}{t}|\xi|^2\right).
\end{align*}
Then by \eqref{E_:FracToLin} and Lemma \ref{L:V/V}, we have that
\begin{align*} 
I &\le C J_0^2(t,x)\int_{0}^{2^{-n}}\ud s \int_{\R^d}\widehat{f}(\ud 
\xi)\exp\left(-\frac{s}{2}|\xi|^2\right) \\
&= CJ_0^2(t,x) V_d(2\times 2^{-n})\\
&\le CJ_0^2(t,x) V_d(2^{-n}),
\end{align*}

{\medskip\bf\noindent Step 2.~} The case $k=2$ is more involved. First we see 
that
\begin{align*}
I= &\int_{T-2^{-n}}^t\ud s\mathop{\int\dots\int}_{\R^{4d}} \mu(\ud z_1)\mu(\ud 
z_2)\mu(\ud z_1')\mu(\ud z_2')\iint_{\R^{2d}}\ud y\ud y' \: f(y-y')\\
&\times G(t-s,x-y)G(s,y-z_1)G(s,y-z_2)\\
&\times G(t-s,x-y')G(s,y'-z_1')G(s,y'-z_2').
\end{align*}
Then we use \eqref{E:GG2} to turn the above six $G$'s to four pairs of $G$'s:
\begin{align*}
I\le & C \int_{T-2^{-n}}^t\ud s\: (t-s)^d \mathop{\int\dots\int}_{\R^{4d}} 
\mu(\ud z_1)\mu(\ud z_2)\mu(\ud z_1')\mu(\ud z_2')\iint_{\R^{2d}}\ud y\ud y' \: 
f(y-y')\\
&\times G(2(t-s),x-y)G(2s,y-z_1)\cdot  G(2(t-s),x-y) G(2s,y-z_2)\\
&\times G(2(t-s),x-y')G(2s,y'-z_1')\cdot  G(2(t-s),x-y') G(2s,y'-z_2').
\end{align*}
Then we apply the relation \eqref{E:GGGG} several times to see that the 
right-hand side of the above inequality is equal to 
\begin{align*}
=& C \int_{T-2^{-n}}^t\ud s\: (t-s)^d \mathop{\int\dots\int}_{\R^{4d}} \mu(\ud 
z_1)\mu(\ud z_2)\mu(\ud z_1')\mu(\ud z_2')\iint_{\R^{2d}}\ud y\ud y' \: 
f(y-y')\\
&\times G(2t,x-z_1)G(2t,x-z_2) G(2t,x-z_1')G(2t,x-z_2') \\
&\times 
G\left(\frac{2s(t-s)}{t},y-z_1-\frac{s}{t}(x-z_1)\right)G\left(\frac{2s(t-s)}{t}
,y-z_2-\frac{s}{t}(x-z_2)\right)\\
&\times  
G\left(\frac{2s(t-s)}{t},y'-z_1'-\frac{s}{t}(x-z_1')\right)G\left(\frac{2s(t-s)}
{t},y'-z_2'-\frac{s}{t}(x-z_2')\right)\\
= & C \int_{T-2^{-n}}^t\ud s\: (t-s)^d \mathop{\int\dots\int}_{\R^{4d}} \mu(\ud 
z_1)\mu(\ud z_2)\mu(\ud z_1')\mu(\ud z_2')\iint_{\R^{2d}}\ud y\ud y' \: 
f(y-y')\\
&\times G(2t,x-z_1)G(2t,x-z_2) G(2t,x-z_1')G(2t,x-z_2') \\
&\times 
G\left(\frac{s(t-s)}{t},y-\frac{z_1+z_2}{2}-\frac{s}{2t}
(2x-z_1-z_2)\right)G\left(\frac{4s(t-s)}{t},z_1-z_2+\frac{s}{t}
(z_1-z_2)\right)\\
&\times  
G\left(\frac{s(t-s)}{t},y'-\frac{z_1'+z_2'}{2}-\frac{s}{2t}
(2x-z_1'-z_2')\right)G\left(\frac{4s(t-s)}{t},z_1'-z_2'+\frac{s}{t}
(z_1'-z_2')\right).
\end{align*}
Then by bounding $G(t,x)$ by $(2\pi t)^{-d/2}$, we can get rid of two $G$'s: 
\begin{align*}
I \le & C \int_{T-2^{-n}}^t\ud s\: \frac{t^d}{s^d} 
\mathop{\int\dots\int}_{\R^{4d}} \mu(\ud z_1)\mu(\ud z_2)\mu(\ud z_1')\mu(\ud 
z_2')\iint_{\R^{2d}}\ud y\ud y' \: f(y-y')\\
&\times G(2t,x-z_1)G(2t,x-z_2) G(2t,x-z_1')G(2t,x-z_2') \\
&\times 
G\left(\frac{s(t-s)}{t},y-\frac{z_1+z_2}{2}-\frac{s}{2t}(2x-z_1-z_2)\right)\\
&\times  
G\left(\frac{s(t-s)}{t},y'-\frac{z_1'+z_2'}{2}-\frac{s}{2t}
(2x-z_1'-z_2')\right).
\end{align*}
For $T>1$ and $n$ large enough, we have that $T-2^{-n}\ge 1/2$ and hence, we 
can 
bound $t^d/s^d$ from above by $2^d$,  so that we can get rid of the factor 
$t^d/s^d$. Then by Fourier transform, we see that 
\begin{align*}
I \le & C \int_{T-2^{-n}}^t\ud s\mathop{\int\dots\int}_{\R^{4d}} \mu(\ud 
z_1)\mu(\ud z_2)\mu(\ud z_1')\mu(\ud z_2')\\
&\times G(2t,x-z_1)G(2t,x-z_2) G(2t,x-z_1')G(2t,x-z_2') \\
&\times \int_{\R^{d}}\ud \xi \: \widehat{f}(\ud 
\xi)\exp\left(-\frac{2s(t-s)}{t}|\xi|^2\right)\\
=&C J_0^4(2t,x)\int_{T-2^{-n}}^t\ud s
\int_{\R^{d}}\ud \xi \: \widehat{f}(\ud 
\xi)\exp\left(-\frac{2s(t-s)}{t}|\xi|^2\right)\\
\le &C J_0^4(2t,x)\int_{0}^{2^{-n}}\ud s
\int_{\R^{d}}\ud \xi \: \widehat{f}(\ud 
\xi)\exp\left(-\frac{2s(t-s)}{t}|\xi|^2\right).
\end{align*}
Then by \eqref{E_:FracToLin},
\begin{align*}
I \le &C J_0^4(2t,x)\int_{0}^{2^{-n}}\ud s
\int_{\R^{d}}\ud \xi \: \widehat{f}(\ud \xi)\exp\left(-s|\xi|^2\right)\\
= &C  J_0^4(2t,x) V_d(2^{-n}),
\end{align*}
which proves \eqref{E:hG-GGJJ} for $k=2$.

{\medskip\bf\noindent Step 3.~} The case when $k=3$ can be proved in a similar 
way. First we see that
\begin{align*}
I= &\int_{T-2^{-n}}^t\ud s\mathop{\int\dots\int}_{\R^{6d}} \prod_{i=1}^3 \mu(\ud 
z_i)\mu(\ud z_i')\iint_{\R^{2d}}\ud y\ud y' \: f(y-y')\\
&\times G(t-s,x-y)G(s,y-z_1)G(s,y-z_2)G(s,y-z_3)\\
&\times G(t-s,x-y')G(s,y'-z_1')G(s,y'-z_2')G(s,y'-z_3').
\end{align*}
Then we use \eqref{E:GG3} to turn the above eight $G$'s to six pairs of $G$'s:
\begin{align*}
I\le & C \int_{T-2^{-n}}^t\ud s\: (t-s)^{2d} \mathop{\int\dots\int}_{\R^{6d}} 
\prod_{i=1}^3\mu(\ud z_i)\mu(\ud z_i') \iint_{\R^{2d}}\ud y\ud y' \: f(y-y')\\
&\times \prod_{i=1}^3 \left[G(3(t-s),x-y)G(3s,y-z_i)\right]\\
&\times \prod_{i=1}^3 \left[ G(3(t-s),x-y')G(3s,y'-z_i')\right].
\end{align*}
Then we apply the relation \eqref{E:GGGG} several times to see that the 
right-hand side of the above inequality is equal to 
\begin{align*}
=& C \int_{T-2^{-n}}^t\ud s\: (t-s)^{2d} \mathop{\int\dots\int}_{\R^{6d}} 
\prod_{i=1}^3 \mu(\ud z_i)\mu(\ud z_i') \iint_{\R^{2d}}\ud y\ud y' \: f(y-y')\\
&\times \prod_{i=1}^3 G(3t,x-z_i) G(3t,x-z_i')\\
&\times \prod_{i=1}^3 
G\left(\frac{3s(t-s)}{t},y-z_i-\frac{s}{t}(x-z_i)\right)\prod_{i=1}^3  
G\left(\frac{3s(t-s)}{t},y'-z_i'-\frac{s}{t}(x-z_i')\right).
\end{align*}
By \eqref{E:Prod3G's}, we see that the product of the last six $G$'s is bounded 
by  
\begin{align*}
\le C \left(\frac{s(t-s)}{t}\right)^{-2d} & 
G\left(\frac{3s(t-s)}{2t},y-\frac{z_1+z_2+2z_3}{4}-\frac{s}{4t}
(4x-z_1-z_2-2z_3)\right)\\
\times & 
G\left(\frac{3s(t-s)}{2t},y'-\frac{z_1'+z_2'+2z_3'}{4}-\frac{s}{4t}
(4x-z_1'-z_2'-2z_3')\right).
\end{align*}
By a similar argument as above we can remove the factor $(t/s)^{2d}$, hence, 
 by Fourier transform, we see that 
\begin{align*}
I \le & C \int_{T-2^{-n}}^t\ud s\mathop{\int\dots\int}_{\R^{6d}} \prod_{i=1}^3 
\bigg\{\mu(\ud z_i)\mu(\ud z_i')  G(3t,x-z_i)G(3t,x-z_i') \bigg\}\\
&\times \int_{\R^{d}}\ud \xi \: \widehat{f}(\ud 
\xi)\exp\left(-\frac{3s(t-s)}{2t}|\xi|^2\right)\\
=&C J_0^6(3t,x)\int_{T-2^{-n}}^t\ud s
\int_{\R^{d}}\ud \xi \: \widehat{f}(\ud 
\xi)\exp\left(-\frac{3s(t-s)}{2t}|\xi|^2\right)\\
\le &C J_0^6(3t,x)\int_{0}^{2^{-n}}\ud s
\int_{\R^{d}}\ud \xi \: \widehat{f}(\ud 
\xi)\exp\left(-\frac{3s(t-s)}{2t}|\xi|^2\right).
\end{align*}
Then by \eqref{E_:FracToLin} and Lemma \ref{L:V/V},
\begin{align*}
I \le &C J_0^6(3t,x)\int_{0}^{2^{-n}}\ud s
\int_{\R^{d}}\ud \xi \: \widehat{f}(\ud \xi)\exp\left(-\frac34
s|\xi|^2\right)\\
\le &C  J_0^6(3t,x) V_d(2^{-n}),
\end{align*}
which proves \eqref{E:hG-GGJJ} for $k=3$.
With this we have completed the whole proof of Proposition \ref{P:Psi}.
\qed

\subsubsection{Proof of Lemma \ref{L:J^k->J}}
In this part, we will prove Lemma \ref{L:J^k->J} in two steps. 
Denote the integral in \eqref{E:J^k->J} by $I_k$.

{\bf\medskip\noindent Step 1 ($k=2$).~} We first prove that $I_2$ is finite so that $J_0^2(2T,x)$ can be viewed as a legal initial data (see \eqref{E:J0finite}). Fix $t\in [T-2^{-n},T]$.
Denote $\bar{z}=(z+z')/2$ below. 
Notice that by \eqref{E:GGGG}, 
\begin{align*}
I_2 = &  \iiint_{\R^{3d}} G(t,x-y) G(2T,y-z)G(2T,y-z')\mu(\ud z)\mu(\ud z')\ud y\\
=  &  \iiint_{\R^{3d}} G(t,x-y) G(4T,z-z')G(T,y-\bar{z})\mu(\ud z)\mu(\ud z')\ud y.
\end{align*}
Then by semigroup property, we see that 
\begin{align}\label{E_:J^2->J}
 I_2 = \iint_{\R^{2d}}  G(4T,z-z')G(t+T,x-\bar{z})\mu(\ud z)\mu(\ud z').
\end{align}
Now, since $t\in [T-2^{-n},T]$, we can get rid of $t$ to see that
\begin{align*}
I_2\le  &  C \iint_{\R^{2d}} G(8T,z-z')G(2T,x-\bar{z}) \mu(\ud z)\mu(\ud z')\\
=  &  C \iint_{\R^{2d}} G(4T,x-z) G(4T,x-z')\mu(\ud z)\mu(\ud z')\\
= & C J_0^2(4T,x)<\infty,
\end{align*}
where we have applied \eqref{E:GGGG}. 

It remains to prove the first inequality in \eqref{E:J^k->J}. 
Actually, because $t\in [T-2^{-n},T]$,
\begin{align*}
J_0^2(2t,x) =&  \iint_{\R^{2d}} G(2t,x-z)G(2t,x-z')\mu(\ud z)\mu(\ud z')\\
=&  \iint_{\R^{2d}} G(4t,z-z')G(t,x-\bar{z})\mu(\ud z)\mu(\ud z')\\
\le & C\iint_{\R^{2d}}G(4T,x-z') G(t+T,x-\bar{z})\mu(\ud z)\mu(\ud z') =C I_2,
\end{align*}
where the last equality is due to \eqref{E_:J^2->J}. 

{\bf\medskip\noindent Step 2 ($k=3$).~} As the previous case we first prove that $I_3$ is finite so that $J_0^3(3T,x)$ can be viewed as a legal initial data (see \eqref{E:J0finite}). Fix $t\in [T-2^{-n},T]$.
Denote $\bar{z}=(z_1+z_2+z_3)/3$ below. 
Notice that 
\begin{align*}
I_3 = &  \mathop{\int\cdots\int}_{\R^{4d}} G(t,x-y) \left[\prod_{i=1}^3 G(3T,y-z_i)\mu(\ud z_i)\right] \ud y.
\end{align*}
Applying \ref{E:GGGG} twice gives that 
\begin{align}\label{E:GGG}
 \prod_{i=1}^3 G(3T,y-z_i) = 
 G(6T,z_1-z_2)G\left(\frac{9T}{2},z_3-\frac{z_1+z_2}{2}\right)
 G\left(T,y-\frac{z_1+z_2+z_3}{3}\right).
\end{align}
Hence, by semigroup property to integrate over $\ud y$, we see that 
\begin{align}\label{E_:J^3->J}
 I_3 = \iiint_{\R^{3d}}  G(6T,z_1-z_2)G\left(\frac{9T}{2},z_3-\frac{z_1+z_2}{2}\right)G(t+T,x-\bar{z}) \prod_{i=1}^3\mu(\ud z_i).
\end{align}
Now, since $t\in [T-2^{-n},T]$, we see that 
\begin{align*}
I_3 \le  &  \iiint_{\R^{3d}} G(12T,z_1-z_2)G\left(9T,z_3-\frac{z_1+z_2}{2}\right)G(2T,x-\bar{z}) \prod_{i=1}^3\mu(\ud z_i)\\
=  &   \iiint_{\R^{3d}} \prod_{i=1}^3 G(6T,y-z_i)\mu(\ud z_i)\\ 
= & C J_0^3(6T,x)<\infty,
\end{align*}
where we have applied \eqref{E:GGG} on the second line. 

It remains to prove the first inequality in \eqref{E:J^k->J}. 
Actually, because $t\in [T-2^{-n},T]$, by \eqref{E:GGG},
\begin{align*}
J_0^3(3t,x) =&  \iiint_{\R^{3d}}  G(6t,z_1-z_2)G\left(\frac{9t}{2},z_3-\frac{z_1+z_2}{2}\right)
 G\left(t,x-\bar{z}\right)\prod_{i=1}^3 \mu(\ud z_i)\\
\le & C\iiint_{\R^{3d}}  
G(6T,z_1-z_2)
G\left(\frac{9T}{2},z_3-\frac{z_1+z_2}{2}\right)
G(t+T,x-\bar{z}) 
\prod_{i=1}^3\mu(\ud z_i)\\
= & C I_3,
\end{align*}
where the last equality is due to \eqref{E_:J^3->J}. 
This proves Lemma \ref{L:J^k->J}.
\qed

\subsection{Moments of \texorpdfstring{$\widehat{u}_{\mathbf{z}}^n(t,x)$}{}  and its first two derivatives}

The aim of this subsection is to prove the following proposition.

\begin{proposition}\label{P:SupU}
For all $\kappa>0$, $1\le i, k\le m$, $n\in\bbN$, 
$p\ge 2$, $t\in [0,T]$ and $x\in\R^d$, we have that
\begin{align}\label{E:SupU}
\Norm{\sup_{|\mathbf{z}|\le \kappa}\left|\widehat{u}^n_{\mathbf{z}}(t,x)\right|}_p
&\le C (1+J_0(t,x)),\\
\label{E2:SupU}
\Norm{\sup_{|\mathbf{z}|\le \kappa} \left|\widehat{u}^{n,i}_{\mathbf{z}}(t,x)\right|
}_p
&\le C \left(\Psi_n^*(t,x;1)+2^{-(1-\beta) n/2}J_0^*(t,x)\right),\\
\label{E3:SupU}
\Norm{\sup_{|\mathbf{z}|\le \kappa} \left|\widehat{u}^{n,i,k}_{\mathbf{z}}(t,x)\right|
}_p
&\le C 
\left(\Psi_n^{**}(t,x;1)+2^{-(1-\beta) n/2}J_0^{**}(t,x)\right),\\
\label{E4:SupU}
\Norm{\sup_{|\mathbf{z}|\le \kappa}\left| \theta^{n,i}_{\mathbf{z}}(t,x)\right|
}_p
&\le C \left(\Psi_n(t,x)+\Psi_n(t,x;1)\right),\\
\label{E5:SupU}
\Norm{\sup_{|\mathbf{z}|\le \kappa}\left| \theta^{n,i,k}_{\mathbf{z}}(t,x)\right|
}_p
&\le C \Psi_n^*(t,x;1).
\end{align}
\end{proposition}
\begin{proof}
Notice that 
\[
\Norm{\sup_{|\mathbf{z}|\le \kappa}
\left|\widehat{u}^n_{\mathbf{z}}(t,x)\right|}_p
\le
\Norm{\sup_{|\mathbf{z}|\le \kappa}\left|\widehat{u}^n_{\mathbf{z}}(t,x)-\widehat{u}^n_{0}(t,x)\right|}_p
+\Norm{\widehat{u}^n_{0}(t,x)}_p.
\]
Then we apply the Kolmogorov continuity theorem and \eqref{E:IncHatU} to the first term
and apply \eqref{E:uPmnt} to the second term to see that
\[
\Norm{\sup_{|\mathbf{z}|\le \kappa}
\left|\widehat{u}^n_{\mathbf{z}}(t,x)\right|
}_p
\le
C\sum_{\ell=0,1}\left[\Psi_n(t,x;\ell)+J_0^\ell(\ell t,x )2^{-(1-\beta) n/2}\right]
+C (1 + J_0(t,x)),
\]
which proves \eqref{E:SupU}.
Similarly, \eqref{E:UniLp0} and \eqref{E2:IncHatU} imply \eqref{E2:SupU};
\eqref{E:UnikLp0} and \eqref{E3:IncHatU} imply \eqref{E3:SupU}.

As for \eqref{E4:SupU}, from \eqref{E:thetazni}, we see that
\begin{align*}
&\Norm{\sup_{|\mathbf{z}|\le \kappa}
\left|
\theta^{n,i}_{\mathbf{z}}(t,x)
\right|
}_p
\\
&\le C \int_0^t\iint_{\R^{2d}} G(t-s,x-y)
\Norm{\sup_{|\mathbf{z}|\le \kappa}
\left|
\widehat{u}^{n}_{\mathbf{z}}(s,y)
\right|
}_p h_n^i(s,y')f(y-y')\ud s\ud y\ud y'.
\end{align*}
Then we can apply \eqref{E:SupU} to obtain \eqref{E4:SupU}.
Similarly, from \eqref{E:thetaznik}, since $\rho'$ is bounded,
\begin{align*}
&\Norm{\sup_{|\mathbf{z}|\le \kappa} \left|\theta^{n,i,k}_{\mathbf{z}}(t,x)
\right|
}_p\\
&\le C \int_0^t\iint_{\R^{2d}} G(t-s,x-y)
\Norm{\sup_{|\mathbf{z}|\le \kappa}
\left| \widehat{u}^{n,i}_{\mathbf{z}}(s,y)
\right| }_p h_n^i(s,y')f(y-y')\ud s\ud y\ud y'.
\end{align*}
Then use the following bound,
\begin{align}
 \label{E2_:SupU}
\Norm{\sup_{|\mathbf{z}|\le \kappa}
\left| \widehat{u}^{n,i}_{\mathbf{z}}(s,y)
\right| }_p \le C J^*_0(t,x),
\end{align}
which is a consequence of \eqref{E:hG-J123} and \eqref{E2:SupU}, to obtain \eqref{E5:SupU}.
Recall that $J_0^*(t,x)$ is defined in \eqref{E:J*}.
This completes the proof of Proposition \ref{P:SupU}.
\end{proof}

\subsubsection{Moments of \texorpdfstring{$\widehat{u}_{\mathbf{z}}^n(t,x)$}{}}
In the next lemma, we study the moments of $\widehat{u}^n_{\mathbf{z}}(t,x)$.

\begin{lemma}\label{L:Moment-U}
For any $p\ge 2$, $T>1$, $(t,x)\in [0,T]\times\R^d$ and $\kappa>0$,
there exists some constant $\beta>0$ independent of $n$ such that
\begin{align}\label{E:uPmnt}
\sup_{n\in\bbN}\sup_{|\mathbf{z}|\le \kappa}\Norm{\widehat{u}_{\mathbf{z}}^n(t,x)}_p \le  C (1 + J_0(t,x)),
\end{align}
where $\tau:=\rho(0)/\LIP_\rho$. As a consequence, 
\begin{align}\label{E:theta-Lp0}
&\sup_{|\mathbf{z}|\le \kappa}\Norm{
\theta_{\mathbf{z}}^{n,i}(t,x)
}_p \le C \left(\Psi_n^i(t,x)+\Psi_n^i(t,x;1)\right),\\
\label{E:theta-bd}
&\max_{1\le i\le m}\sup_{|\mathbf{z}|\le \kappa}\Norm{\theta_{\mathbf{z}}^{n,i}(t,x)}_p < C (1+J_0(t,x))\one_{\{t> T-2^{-n}\}},
\end{align}
and under Assumption \ref{A:Rate-Sharp}, for $x\ne x_i$, 
\begin{align}\label{E:theta-AS0}
\lim_{n\rightarrow\infty}
\theta_{\mathbf{z}}^{n,i}(t,x)=0 \:\: \text{a.s. for all $t\in[0,T]$ and $|\mathbf{z}|\le\kappa$}.
\end{align}
\end{lemma}
\begin{proof}
We prove this Lemma in two steps. 

{\bigskip\bf\noindent Step 1.~} We first prove \eqref{E:uPmnt}.
Recall that $\widehat{u}^n_{\mathbf{z}}(t,x)$ satisfies \eqref{E:hatUn}. 
We will use Picard iteration to show its existence and uniqueness and moment bounds. 
Since $m$ is used to denote the number of preselected points $x_i$, we will use $m'$ for the Picard iteration.  
Define 
\begin{equation}
\widehat{u}^n_{\mathbf{z},0}(t,x) = J_0(t,x)\,.
\end{equation}
and
\begin{align*}
\widehat{u}^n_{\mathbf{z},m'+1}(t,x) = & J_0(t,x) + \int_0^t \int_{\RR^d} G(t-s, x-y)\rho(\widehat{u}^n_{\mathbf{z},m'}(s,y)) W(\ud s\ud y) \\
&+ \int_0^t \iint_{\R^{2d}} G(t-s, x-y) \rho(\widehat{u}^n_{\mathbf{z},m'}(s,y)) \langle 
\mathbf{z}, \mathbf{h}_n(s,y')\rangle f(y-y'))\ud y\ud y' \ud s
\end{align*}
for $m'\ge 1$. Recall that $\tau=\rho(0)/\LIP_\rho$. Then we have that
\begin{align*}
\frac{\tau+|\widehat{u}^n_{\mathbf{z},m'+1}(t,x)|}{\tau+J_0(t,x)} \leq 1&+\left| \int_0^t \int_{\RR^d} \frac{G(t-s, x-y)[\tau+J_0(s,y)]}{\tau+J_0(t,x)}\frac{\rho(\widehat{u}^n_{\mathbf{z},m'}(s,y))}{\tau+J_0(s,y)} W(\ud s\ud y)\right| \\
&+ \bigg|\int_0^t\ud s \iint_{\RR^{2d}} \ud y \ud y' \: \frac{G(t-s, 
x-y)[\tau+J_0(s,y)]}{\tau+J_0(t,x)} 
\frac{\rho(\widehat{u}^n_{\mathbf{z},m'}(s,y))}{\tau+J_0(s,y)}
\\
&\qquad \times  \InPrd{\mathbf{z}, \mathbf{h}_n(s,y')} f(y-y') \bigg|\,.
\end{align*}
Set 
\[
\Theta_{T,m',n}:=
\sup_{0 \leq s \leq T, y \in \RR^d}e^{-\beta s}\Norm{ \frac{\tau + |\widehat{u}^n_{\mathbf{z},m'}(s,y)|}{\tau+J_0(s,y)} }_p.
\]
Taking $L^p(\Omega)$-norm on both sides and multiplying by $e^{-\beta t}$, we have that
\begin{align*}
e^{-\beta t}&\left\|\frac{\tau+|\widehat{u}^n_{\mathbf{z},m'+1}(t,x)|}{\tau+J_0(t,x)}\right\|_{p}&\\ &\leq 1+e^{-\beta t}\left\| \int_0^t \int_{\RR^d} \frac{G(t-s, x-y)[\tau+J_0(s,y)]}{\tau+J_0(t,x)}\frac{\rho(\widehat{u}^n_{\mathbf{z},m'}(s,y))}{\tau+J_0(s,y)} W(\ud s\ud y)\right\|_{p} \\
&\quad + e^{-\beta t}\left\|\int_0^t \int_{\RR^d} \frac{G(t-s, x-y)[\tau+J_0(s,y)]}{\tau+J_0(t,x)} \frac{\rho(\widehat{u}^n_{\mathbf{z},m'}(s,y))}{\tau+J_0(s,y)}\langle \mathbf{z}, \mathbf{h}_n(s,y)\rangle \ud y \ud s\right\|_{p}\\
&\leq 1+ \Theta_{T,m',n} C_p \LIP_\rho I_1 +  \Theta_{T,m',n} \LIP_\rho I_2,
\end{align*} 
where 
\begin{align*}
I_1 := \bigg(\int_0^t\ud s \:
e^{-2\beta (t-s)} \iint_{\RR^{2d}}\ud y\ud y' 
& \frac{G(t-s, x-y)[\tau + J_0(s,y)]}{\tau + J_0(t,x)}\\
\times & \frac{G(t-s, x-y')[\tau + J_0(s,y')]}{\tau + J_0(t,x)}  f(y-y') \bigg)^{1/2}
\end{align*}
and 
\[
I_2 := \int_0^t \iint_{\R^{2d}} e^{-\beta(t-s)}\frac{G(t-s, x-y)[\tau+J_0(s,y)]}{\tau+J_0(t,x)} |\langle \mathbf{z}, \mathbf{h}_n(s,y')\rangle| f(y-y')\ud y \ud y' \ud s.
\]
For $I_1$, using Minkowski's inequality we obtain that
\begin{align*}
I_1 &\leq \left(\int_0^t \iint_{\R^{2d}} e^{-2\beta (t-s)} G(t-s, x-y)G(t-s, x-y')\left(\frac{\tau}{\tau+J_0(t,x)}\right)^2 f(y-y') \ud y \ud y' \ud s\right)^{1/2}\\
&\quad +\left( \int_0^t \iint_{\R^{2d}}
\frac{G(t-s, x-y)J_0(s,y)}{\tau+ J_0(t,x)}\frac{G(t-s, x-y')J_0(s,y')}{\tau+J_0(t,x)} e^{-2\beta (t-s)} f(y-y') \ud y \ud y' \ud s\right)^{1/2}\\
&\leq \left(\int_0^t \iint_{\R^{2d}} e^{-2\beta (t-s)} G(t-s, x-y)G(t-s, x-y')f(y-y') \ud y \ud y' \ud s\right)^{1/2}\\
&\quad +\left( \int_0^t \iint_{\R^{2d}}
\frac{G(t-s, x-y)J_0(s,y)}{J_0(t,x)}\frac{G(t-s, x-y')J_0(s,y')}{J_0(t,x)} e^{-2\beta (t-s)} f(y-y') \ud y \ud y' \ud s\right)^{1/2}\;.
\end{align*}
In the second summand of $I_1$, using the identity \eqref{E:GG} and the Fourier 
transform we obtain that
\begin{align*}
I_1\leq & (2\pi)^{-d/2} \left(\int_0^t \int_{\RR^d}  e^{-2\beta (t-s)} e^{-2(t-s)|\xi|^2} \widehat{f}(\ud\xi) \ud s\right)^{1/2}\\
& + (2\pi)^{-d/2} \left(\int_0^t \int_{\RR^d} e^{-2\beta(t-s)} e^{-\frac{2(t-s)s}{t}|\xi|^2}\widehat{f}(\ud\xi) \ud s \right)^{1/2}\,.
\end{align*}
The dominated convergence theorem shows that $I_1$ can be arbitrarily small if $\beta$ is sufficiently large. 

Similarly, $I_2$ can be bounded as the following
\begin{align*}
I_2 &\leq \int_0^t \iint_{\R^{2d}} e^{-\beta(t-s)} G(t-s, x-y) |\langle \mathbf{z}, \mathbf{h}_n(s,y')\rangle| f(y-y')\ud y\ud y' \ud s\\
&\quad  +\int_0^t \iint_{\R^{2d}} e^{-\beta(t-s)} \frac{G(t-s, x-y)J_0(s,y)}{J_0(t,x)} |\langle \mathbf{z}, \mathbf{h}_n(s,y')\rangle| f(y-y') \ud y\ud y' \ud s\\
&=: I_{21}(\beta) + I_{22}(\beta).
\end{align*}
By Proposition \ref{P:Psi}, we see that 
\[
I_{21}(\beta)\le I_{21}(0)\le \kappa \Psi_n(t,x) \le \kappa m,
\]
and 
\[
I_{22}(\beta)\le I_{22}(0)\le \kappa \Psi_n(t,x)\le \kappa \max_{i=1,\dots,m} \sup_{(t,x)\in (0,T]\times\R^d, n\in\bbN} \Psi_n^i(t,x)<\infty.
\]
Hence, we can again apply the dominated convergence theorem to show that $I_2$ can be 
arbitrarily small if $\beta$ is large enough. \\

Therefore, the above arguments show that
\[
\Theta_{T,m'+1,n} \le 1 +  \Theta_{T,m',n}   C_\beta  C_p\LIP_\rho,
\]
where $C_\beta$ can be arbitrarily small if $\beta$ is large enough.
Then the induction on $m'$ shows that for  some $\beta$ sufficiently large it holds that
\begin{equation}
\sup_{(t,x)\in (0,T]\times\R^d} \sup_{n\in \bbN} \sup_{|\mathbf{z}|\leq \kappa}e^{-\beta 
t}\left\|\frac{\tau+|\widehat{u}^n_{\mathbf{z},m'}(t,x)|}{\tau+J_0(t,x)}\right\|_{p
}\leq C,\quad\text{for all $m'\in\bbN$.}
\end{equation}
Next, by considering the difference 
\begin{equation}
\frac{|\widehat{u}^n_{\mathbf{z},m'+1}(t,x)-\widehat{u}^n_{\mathbf{z},m'}(t,x)|}
{\tau+J_0(t,x)} \,,
\end{equation}
it is easy to show the existence and uniqueness of the solution. We also have the moment bound
\[
\sup_{|\mathbf{z}|\leq \kappa}\sup_{n\in \bbN }\|\widehat{u}^n_{\mathbf{z}}(t,x)\|_{p} \leq \tau + (\tau + J_0(t,x))e^{\beta t}\,,
\]
for some $\beta$ sufficiently large. 
Finally, because $t\in[0,T]$, the above inequality is equivalent to \eqref{E:uPmnt}.

{\bigskip\noindent\bf Step 2.~} Now we study the rest properties that are related to $\theta_{\mathbf{z}}^{n,i}(t,x)$. By Minkowski's inequality and the Lipschitz continuity of $\rho$, we have that
\begin{align*}
 \Norm{\theta_{\mathbf{z}}^{n,i}(t,x)}_p 
 &\le \int_0^t\ud s\iint_{\R^{2d}}\ud y \ud y'\: G(t-s,x-y)
 \Norm{\rho\left(\widehat{u}^n_{\mathbf{z}}(s,y)\right)}_p h_n^i(s,y') f(y-y')\\
 &\le \LIP_\rho\int_0^t\ud s\iint_{\R^{2d}}\ud y \ud y'\: G(t-s,x-y)
 \left( \tau + \Norm{\widehat{u}^n_{\mathbf{z}}(s,y)}_p \right)h_n^i(s,y') f(y-y'),
\end{align*}
where we recall that $\tau = \rho(0)/\LIP_{\rho}$. Then by \eqref{E:uPmnt}, for some constant $C_T>0$, 
\begin{align*}
\Norm{\theta_{\mathbf{z}}^{n,i}(t,x)}_p 
 & \le C_T \LIP_\rho\int_0^t\ud s\iint_{\R^{2d}}\ud y \ud y'\: G(t-s,x-y)
 \left( 1+ J_0(s,y) \right)h_n^i(s,y') f(y-y')\\
 & = C_T \LIP_\rho \left[\Psi_n^i(t,x) +\Psi_n^i(t,x;1) \right],
\end{align*}
which proves \eqref{E:theta-Lp0}. Property \eqref{E:theta-bd} is a direct consequence of \eqref{E:theta-Lp0}, \eqref{E:hG-J0} and \eqref{E:hG-J123}.
Finally, By \eqref{E:theta-Lp0} and the bounds in parts (3) and (4) of Proposition \ref{P:Psi}, one can apply the Borel–Cantelli lemma to obtain \eqref{E:theta-AS0}.
This completes the proof of Lemma \ref{L:Moment-U}.
\end{proof}

\subsubsection{Moments of \texorpdfstring{$\widehat{u}_{\mathbf{z}}^{n,i}(t,x)$}{}}

In the next lemma, we study the moments of $\widehat{u}^{n,i}_{\mathbf{z}}(t,x)$.

\begin{lemma}\label{L:Moment-DU}
For any $p\ge 2$, $n\in\bbN$, $i=1,\dots, d$, and $\kappa>0$,
we have that
\begin{align}\label{E:UniLp0}
\sup_{|\mathbf{z}|\le\kappa}
\Norm{\widehat{u}^{n,i}_{\mathbf{z}}(t,x)}_p
\le C \sum_{\ell=0,1}\left[\Psi_n(t,x;\ell) + J_0^\ell(\ell t,x)2^{-(1-\beta) n/2} \one_{\{t> T-2^{-n}\}} \right]\,,
\end{align}
and as a consequence,
\begin{align}\label{E:thetaIK-bd}
 &\max_{1\le i, k\le d}
 \sup_{|\mathbf{z}|\le \kappa}
 \Norm{
 \theta_{\mathbf{z}}^{n,i,k}(t,x)}_p \le
 C \left(\Psi_n(t,x;0) + \Psi_n(t,x;1)\right)\,.
 \end{align}
\end{lemma}

\begin{proof}
Recall that $\widehat{u}^{n,i}_{\mathbf{z}}(t,x)$ satisfies \eqref{E:hatUni}. 
We claim that
\begin{align*}
\sup_{|\mathbf{z}|\le \kappa} \Norm{\widehat{u}_{\mathbf{z}}^{n,i}(t,x)}_p\le \sum_{i=1}^3 I_i,
\end{align*}
where the $I_i$ are defined and bounded as follows:
From \eqref{E:hatUni}, 
By \eqref{E:theta-Lp0},
\[
I_1 :=\sup_{|\mathbf{z}|\le \kappa} \Norm{\theta_{\mathbf{z}}^{n,i}(t,x)}_p \le C \left(\Psi_n(t,x;0)+\Psi_n(t,x;1)\right).
\]
By the boundedness of $\rho'$, \eqref{E:supxKui} and \eqref{E:hG-GGJJ} (together with Assumption \ref{A:Rate-Sharp}), we see that 
\begin{align*}
I_2 := &  \Norm{\rho'}_{L^\infty} \one_{(t > T- 2^{-n})}\bigg(\int_{T-2^{-n}}^t\ud s \iint_{\R^{2d}}\ud y\ud y'\:  G(t-s, x-y)G(t-s, x-y') f(y-y') \\
&\qquad\times
\sup_{|\mathbf{z}|\le \kappa} \Norm{\widehat{u}_{\mathbf{z}}^{n,i}(s,y)}_p
\sup_{|\mathbf{z}|\le \kappa} \Norm{\widehat{u}_{\mathbf{z}}^{n,i}(s,y')}_p \bigg)^{\frac{1}{2}}.
\end{align*}
By the boundedness of $\rho'$, \eqref{E:supxKui} and \eqref{E:PsiN}, we see that  
\begin{align*}
I_3:=&  \Norm{\rho'}_{L^\infty} \int_0^t \ud s\iint_{\RR^{2d}}\ud y\ud y'\: G(t-s, x-y)\InPrd{\mathbf{1}, \mathbf{h}_n(s,y')} f(y-y') \sup_{|\mathbf{z}|\le \kappa} \Norm{\widehat{u}_{\mathbf{z}}^{n,i}(s,y)}_p.
\end{align*}
Thanks \eqref{E:hG-J0} and \eqref{E:hG-J123}, when evaluating the upper moment bounds of $\widehat{u}^{n,i}_{\mathbf{z}}(t,x)$, we can treat $\widehat{u}^{n,i}_{\mathbf{z}}(t,x)$ as if it starts from the initial data $C(1+\mu)$. 
Hence, we can apply the same Picard iteration scheme as in the proof of Lemma \ref{L:Moment-U} to see that 
\begin{align}\label{E:supxKui}
\max_{1\le i\le d} 
\sup_{|\mathbf{z}|\le \kappa}\Norm{\widehat{u}^{n,i}_{\mathbf{z}}(t,x)}_p< 
C(1+J_0(t,x))\one_{\{t> T-2^{-n}\}}\,.
\end{align}
Then plugging the bound \eqref{E:supxKui} back to $I_2$ and $I_3$ shows that
\begin{align*}
I_2\le &C \one_{(t > T- 2^{-n})}\left(1+J_0(t,x)\right)2^{-(1-\beta) n/2},\\
I_3\le &C \left(\Psi_n(t,x;0)+\Psi_n(t,x;1)\right),
\end{align*}
which proves \eqref{E:UniLp0}.
Finally, the proof for \eqref{E:thetaIK-bd} is the same as those for $I_1$ above. 
This completes the whole proof of Lemma \ref{L:Moment-DU}.
\end{proof}

\subsubsection{Moments of \texorpdfstring{$\widehat{u}_{\mathbf{z}}^{n,i,k}(t,x)$}{}}
In the next lemma, we study the moments of $\widehat{u}^{n,i,k}_{\mathbf{z}}(t,x)$. 

\begin{lemma}\label{L:Moment-DDU}
For any $p \geq 2$, $n \in \mathbb{N}$, $1\leq i, k\leq d$, and $\kappa >0$, we have that 
\begin{equation}\label{E:UnikLp0}
\sup_{|\mathbf{z}|\leq \kappa} \Norm{\hat{u}_{\mathbf{z}}^{n,i,k}(t,x)}_{p}\leq 
C\left[\Psi_n^*(t,x;1) + J_0^*(t,x) 2^{-(1-\beta) n/2}\one_{\{t> T-2^{-n}\}} \right].
\end{equation}
\end{lemma}
\begin{proof}
We can write the six parts of $\widehat{u}^{n,i,k}_{\mathbf{z}}(t,x)$ in \eqref{E:hatUnik} as
\begin{align}\label{E:Unik6}
\widehat{u}^{n,i,k}_{\mathbf{z}}(t,x) =
\theta^{n,i,k}_{\mathbf{z}}(t,x) +\theta^{n,k,i}_{\mathbf{z}}(t,x)+ \sum_{\ell=1}^4 U_\ell^n(t,x).
\end{align}
Hence, we have that
\begin{align*}
\sup_{|\mathbf{z}|\le \kappa} \Norm{\widehat{u}_{\mathbf{z}}^{n,i,k}(t,x)}_p\le \sum_{i=0}^4 I_i,
\end{align*}
where the $I_i$ are defined and bounded as follows:
By \eqref{E:thetaIK-bd},
\begin{align*}
I_0 := &  \sup_{|\mathbf{z}|\le\kappa}\Norm{\theta^{n,i,k}_{\mathbf{z}}(t,x)}_p +
\sup_{|\mathbf{z}|\le\kappa}\Norm{\theta^{n,k,i}_{\mathbf{z}}(t,x)}_p \\
\le &C \left(\Psi_n(t,x) + \Psi_n(t,x;1)\right).
\end{align*}
By the boundedness of $\rho''$, the moments bound for $\widehat{u}^{n,i}_{\mathbf{z}}(t,x)$ in \eqref{E:supxKui} and \eqref{E:hG-GGJJ} (together with Assumption \ref{A:Rate-Sharp}), we see that 
\begin{align}\notag
I_1:= &\sup_{|\mathbf{z}|\le\kappa}\Norm{U_1^n(t,x)}_p \\ \notag
\le & C\Norm{\rho''}_{L^\infty}\one_{\{t> T-2^{-n}\}} \bigg(\int_{T-2^{-n}}^t\ud s\iint_{\R^{2d}} G(t-s,x-y) G(t-s,x-y')f(y-y')\\ \notag
&\times \sup_{|\mathbf{z}|\le \kappa} \left(
\Norm{\widehat{u}_{\mathbf{z}}^{n,i}(s,y)}_{2p}
\Norm{\widehat{u}_{\mathbf{z}}^{n,{k}}(s,y)}_{2p}
\Norm{\widehat{u}_{\mathbf{z}}^{n,i}(s,y')}_{2p}
\Norm{\widehat{u}_{\mathbf{z}}^{n,{k}}(s,y')}_{2p}
\right) \bigg)^{1/2}\\ \notag
\le & C\Norm{\rho''}_{L^\infty}\one_{\{t> T-2^{-n}\}} \bigg(\int_{T-2^{-n}}^t\ud s\iint_{\R^{2d}} G(t-s,x-y) G(t-s,x-y')f(y-y')\\ \notag
&\times \left(1+J_0(s,y)\right)^2\left(1+J_0(s,y')^2\right)
\bigg)^{1/2}
\\
 \le & C 2^{-(1-\beta) n/2}\left(1+J_0^2(2t,x)\right)\one_{\{t> T-2^{-n}\}}.
\label{E:U13}
\end{align}
By the boundedness of $\rho''$, \eqref{E:supxKui} and \eqref{E:PsiN}, we see that 
\begin{align*}
I_2:= & \sup_{|\mathbf{z}|\le\kappa}\Norm{U_2^n(t,x)}_p\\
\le & \Norm{\rho''}_{L^\infty} \int_0^t \ud s  \iint_{\R^{2d}}\ud y\ud y' G(t-s,x-y)\InPrd{\mathbf{1},\mathbf{h}_n(s,y')}f(y-y')\\
&\times 
\sup_{|\mathbf{z}|\le\kappa}\left(
\Norm{\widehat{u}_{\mathbf{z}}^{n,i}(s,y)}_{2p}\Norm{\widehat{u}_{\mathbf{z}}^{n
,{k}}(s,y)}_{2p}\right)\\
\le & C \left(\Psi_n(t,x)+\Psi_n(t,x;2)\right). 
\end{align*}
Similarly, 
\begin{align*}
I_3:= &\sup_{|\mathbf{z}|\le\kappa}\Norm{U_3^n(t,x)}_p \\ \notag
\le & C\Norm{\rho'}_{L^\infty}\one_{\{t> T-2^{-n}\}} \bigg(\int_{T-2^{-n}}^t\ud s\iint_{\R^{2d}} G(t-s,x-y) G(t-s,x-y')f(y-y')\\ \notag
&\times \sup_{|\mathbf{z}|\le \kappa} \left(
\Norm{\widehat{u}_{\mathbf{z}}^{n,i,k}(s,y)}_p
\Norm{\widehat{u}_{\mathbf{z}}^{n,i,k}(s,y')}_p
\right) \bigg)^{1/2}
\end{align*}
and 
\begin{align*}
I_4:= & \sup_{|\mathbf{z}|\le\kappa}\Norm{U_4^n(t,x)}_p\\
\le & \Norm{\rho'}_{L^\infty} \int_0^t \ud s  \iint_{\R^{2d}}\ud y\ud y' G(t-s,x-y)\InPrd{\mathbf{1},\mathbf{h}_n(s,y')}f(y-y')\\
&\times 
\sup_{|\mathbf{z}|\le\kappa} \Norm{\widehat{u}_{\mathbf{z}}^{n,i,k}(s,y)}_p.
\end{align*}
Notice that 
\begin{align*}
I_0+I_1+I_2
\le & C \left(\Psi_n(t,x)+\Psi_n(t,x;1)+\Psi_n(t,x;2)\right)+ 
C 2^{-(1-\beta) n/2}\left(1+J_0^2(2t,x)\right)\one_{\{t> T-2^{-n}\}}\\
\le &
C \left[1+J_0(t,x)+J_0^2(2t,x)\right] \one_{\{t> T-2^{-n}\}}\\
\le & C J_0^*(t,x)\one_{\{t> T-2^{-n}\}},
\end{align*}
where in the second inequality we have applied \eqref{E:hG-J0} and \eqref{E:hG-J123},
and the last inequality is due to Lemma \ref{L:J^k->J}.
Therefore, $\Norm{\widehat{u}^{n,i,k}_{\mathbf{z}}(t,x)}_p$ satisfies a similar integral inequality as that for $\Norm{\widehat{u}^{n}_{\mathbf{z}}(t,x)}_p$. Hence, we can carry out the same Picard iteration scheme as that in the proof of Lemma \ref{L:Moment-U} to conclude that 
\begin{align}\label{E:supxKuik}
\max_{1 \leq i, k\leq d} \sup_{|\mathbf{z}|\leq \kappa} \Norm{\hat{u}_{\mathbf{z}}^{n,i,k}(t,x)}_p < C(1+J_0^*(t,x))\one_{\{t> T-2^{-n}\}}
\le C J_0^*(t,x)\one_{\{t> T-2^{-n}\}}.
\end{align}
Then by plugging the above bounds back to the upper bounds for $I_3$ and $I_4$, we see that 
\begin{align*}
I_3\le&C 2^{-(1-\beta) n/2}J_0^*(t,x)\one_{\{t> T-2^{-n}\}}\quad\text{and}\quad
I_4\le C \Psi_n^*(t,x;1).
\end{align*}
Finally, we can use Lemma \ref{L:Star} to upgrade the bounds for $I_0$, $I_1$ and $I_2$ into either $C \Psi_n^*(t,x;1)$ or 
$C 2^{-(1-\beta) n/2}J_0^*(t,x)\one_{\{t> T-2^{-n}\}}$. 
This completes the proof of Lemma \ref{L:Moment-DDU}. 
\end{proof}

\subsubsection{Moment increments in \texorpdfstring{$\mathbf{z}$}{}}
Since we want to bring the ``$\sup_{|\mathbf{z}|\vee|\mathbf{z}'|\le \kappa}$'' inside the expectation, we need to study the moment increments in $\mathbf{z}$.

\begin{lemma}\label{L:HolderInZ}
For all $\kappa >0$, $1 \leq i, k \leq d$, $n \in \mathbb{N}$, $p\ge 2$, $t\in [0,T]$ and $x\in \RR^d$,
we have 
\begin{align}\label{E:IncHatU}
\sup_{|\mathbf{z}|\vee |\mathbf{z}'|\le \kappa}
\Norm{\widehat{u}^n_{\mathbf{z}}(t,x)-\widehat{u}^n_{\mathbf{z}'}(t,x)}_p
&\le C |\mathbf{z}-\mathbf{z}'| 
\sum_{\ell=0,1}\left(\Psi_n(t,x;\ell)+2^{-(1-\beta) n/2 }J_0^\ell(\ell 
t,x)\right),\\
\label{E2:IncHatU}
\sup_{|\mathbf{z}|\vee |\mathbf{z}'|\le \kappa}
\Norm{\widehat{u}^{n,i}_{\mathbf{z}}(t,x)-\widehat{u}^{n,i}_{\mathbf{z}'}(t,x)}_p
&\le C |\mathbf{z}-\mathbf{z}'| \left(\Psi_n^*(t,x;1)+2^{-(1-\beta) n/2 
}J_0^*(t,x)\right),\\
\label{E3:IncHatU}
\sup_{|\mathbf{z}|\vee |\mathbf{z}'|\le \kappa}
\Norm{\widehat{u}^{n,i,k}_{\mathbf{z}}(t,x)-\widehat{u}^{n,i,k}_{\mathbf{z}'}(t,x)}_p
&\le C |\mathbf{z}-\mathbf{z}'|\left(\Psi_n^{**}(t,x;1)+2^{-(1-\beta) n/2 
}J_0^{**}(t,x)\right)\,. 
\end{align}
\end{lemma}
\begin{proof}
We will prove these three inequalities in this lemma in three steps. 

{\bigskip\noindent\bf Step 1.~} 
In this step we prove \eqref{E:IncHatU}. Notice that 
\begin{align*}
&\widehat{u}_{\mathbf{z}}^n(t,x) - \widehat{u}_{\mathbf{z}'}^n(t,x)\\
&=\int_0^t \int_{\RR^d}G(t-s, x-y) \left[\rho(\widehat{u}_{\mathbf{z}}^n(s,y))-\rho(\widehat{u}_{\mathbf{z}'}^n(s,y))\right]W(\ud s\ud y)\\
&\quad + \int_0^t \iint_{\R^{2d}} G(t-s, x-y)\left[\rho(\widehat{u}_{\mathbf{z}}^n(s,y))-\rho(\widehat{u}_{\mathbf{z}'}^n(s,y))\right] \InPrd{\mathbf{z}, h_n(s,y')}f(y-y') \ud y\ud y' \ud s\\
&\quad + \int_0^t \iint_{\R^{2d}} G(t-s, x-y) \rho(\widehat{u}_{\mathbf{z}'}^n(s,y)) \InPrd{\mathbf{z}-\mathbf{z}', \mathbf{h}_n(s,y')}f(y-y') \ud y \ud y' \ud s.
\end{align*}
Hence, we have that 
\[
\sup_{|\mathbf{z}|\vee |\mathbf{z}'|\le \kappa}
\Norm{\widehat{u}^{n}_{\mathbf{z}}(t,x)-\widehat{u}^{n}_{\mathbf{z}'}(t,x)}_p
\le C \sum_{i=1}^{3} I_i,
\]
where $I_i$ are defined and bounded as follows: 
By the Lipschitz continuity of $\rho$, 
\begin{align*}
I_1^2 := &\LIP_\rho^2 \int_0^t \ud s\iint_{\R^{2d}} \ud y\ud y' \: G(t-s, x-y)G(t-s, x-y') f(y-y')\\
&\times \sup_{|\mathbf{z}|\vee |\mathbf{z}'|\le \kappa} \left(
\Norm{\widehat{u}_{\mathbf{z}}^n(s,y)-\widehat{u}_{\mathbf{z}'}^n(s,y)}_p
\Norm{\widehat{u}_{\mathbf{z}}^n(s,y')-\widehat{u}_{\mathbf{z}'}^n(s,y')}
_p\right),\\
I_2 := & \LIP_\rho\int_0^t \ud s\iint_{\R^{2d}}\ud y\ud y'\: G(t-s, x-y)
\InPrd{\mathbf{1}, \mathbf{h}_n(s,y')}f(y-y')\\
&\times \sup_{|\mathbf{z}|\vee |\mathbf{z}'|\le \kappa} \Norm{\widehat{u}_{\mathbf{z}}^n(s,y)-\widehat{u}_{\mathbf{z}'}^n(s,y)}_p.
\end{align*}
By the linear growth of $\rho$, \eqref{E:uPmnt} and \eqref{E:PsiN}, 
\begin{align}\notag
I_3: = & |\mathbf{z}-\mathbf{z}'| \int_0^t \ud s \iint_{\R^{2d}} \ud y \ud y'\: G(t-s, x-y) \InPrd{\mathbf{1}, \mathbf{h}_n(s,y')} f(y-y')\\ \notag
&\times \left(1+\sup_{|\mathbf{z}'|\le \kappa} \Norm{\widehat{u}_{\mathbf{z}'}^n(s,y)}_p\right)  \\ \notag
\le & C |\mathbf{z}-\mathbf{z}'| \int_0^t \ud s \iint_{\RR^{2d}} \ud y\ud y'\: G(t-s, x-y) \InPrd{\mathbf{1}, \mathbf{h}_n(s,y')}f(y-y')\\  \notag
&\times \left(1+J_0(s,y)\right)  \\
\le &C |\mathbf{z}-\mathbf{z}'| \left(\Psi_n(t,x)+\Psi_n(t,x;1)\right).
\label{E_:IncHatU-I1}
\end{align}
By \eqref{E:hG-J0} and \eqref{E:hG-J123}, we see that 
\[
I_3\le C |\mathbf{z}-\mathbf{z}'| \left(1+J_0(t,x)\right)\one_{\{t>T-2^{-n}\}}.
\]
Hence, we can apply the same Picard iterate scheme as in the proof of Lemma \ref{L:Moment-U} to see that 
\begin{align}\label{E_:HolderU}
\sup_{|\mathbf{z}|\vee |\mathbf{z}'|\le \kappa}\Norm{\widehat{u}_{\mathbf{z}}^n(t,x) - \widehat{u}_{\mathbf{z}'}^n(t,x)}_p\le C |\mathbf{z}-\mathbf{z}'| \left(1+J_0(t,x)\right)\one_{\{t>T-2^{-n}\}}.
\end{align}
Then plugging the moment bound \eqref{E_:HolderU} back to the upper bounds for $I_1$ and $I_2$ shows that 
\begin{align*}
I_1 &\le C |\mathbf{z}-\mathbf{z}'|\one_{\{t> T-2^{-n}\}}  \left(1+J_0(t,x)\right)2^{-(1-\beta) n/2},\\
I_2 &\le C |\mathbf{z}-\mathbf{z}'|\left(\Psi_n(t,x) + \Psi_n(t,x;1)\right), 
\end{align*}
which proves \eqref{E:IncHatU}.

{\bigskip\bf\noindent Step 2.~} Now we will prove \eqref{E2:IncHatU}.
Similar to the previous case, we have 
\[
\sup_{|\mathbf{z}|\vee |\mathbf{z}'|\le \kappa}
\Norm{\widehat{u}^{n,i}_{\mathbf{z}}(t,x)-\widehat{u}^{n,i}_{\mathbf{z}'}(t,x)}_p
\le C \sum_{i=1}^{6} I_i,
\]
with $I_i$ being defined and bounded as follows. 
By the Lipschitz continuity of $\rho$ and \eqref{E_:HolderU}, 
\begin{align*}
I_1:=& \sup_{|\mathbf{z}|\vee |\mathbf{z}'|\le \kappa} \Norm{\theta_{\mathbf{z}}^{n,i}(t,x)-\theta_{\mathbf{z}'}^{n,i}(t,x)}_p\\
\le&
\LIP_\rho\int_{0}^t\ud s\iint_{\R^{2d}} \ud y \ud y'\: G(t-s,x-y) f(y-y') h_n^i(s,y')
\\
& \times  \sup_{|\mathbf{z}|\vee |\mathbf{z}'|\le \kappa}  \Norm{\widehat{u}_{\mathbf{z}}^{n}(s,y)-\widehat{u}_{\mathbf{z}'}^{n}(s,y)}_p\\
\le& C|\mathbf{z}-\mathbf{z}'| \left( \Psi_n(t,x)+\Psi_n(t,x;1)\right).
\end{align*}
By the Lipschitz continuity of $\rho'$ and the Schwartz inequality,
\begin{align*}
I_2^2:=&
\LIP_{\rho'}^2 \one_{\{t>T-2^{-n}\}}  \int_{0}^t\ud s\iint_{\R^{2d}} \ud y \ud y'\: G(t-s,x-y)G(t-s,x-y')f(y-y')\\
&\times 
\sup_{|\mathbf{z}|\vee |\mathbf{z}'|\le \kappa}
\left[\Norm{\widehat{u}_{\mathbf{z}}^{n}(s,y)-\widehat{u}_{\mathbf{z}'}^{n}(s,y)}_{2p}
\Norm{\widehat{u}_{\mathbf{z}}^{n,i}(s,y)}_{2p}\right]\\
&\times  
\sup_{|\mathbf{z}|\vee |\mathbf{z}'|\le \kappa} 
\left[\Norm{\widehat{u}_{\mathbf{z}}^{n}(s,y')-\widehat{u}_{\mathbf{z}'}^{n}(s,y')}_{2p}
\Norm{\widehat{u}_{\mathbf{z}}^{n,i}(s,y')}_{2p}\right].
\end{align*}
Then by \eqref{E_:HolderU}, \eqref{E:supxKui} and \eqref{E:hG-GGJJ},
\begin{align*}
I_2\le&
C\one_{\{t>T-2^{-n}\}}\bigg(  \int_{0}^t\ud s\iint_{\R^{2d}} \ud y \ud y'\: G(t-s,x-y)G(t-s,x-y')f(y-y')\\
&\times (1+J_0(s,y))^2(1+J_0(s,y'))^2\bigg)^{1/2} |{\bf z} - {\bf z'}|\\
\le & C |\mathbf{z}-\mathbf{z}'|\one_{\{t> T-2^{-n}\}}  \left(1+J_0^2(2t,x)\right)2^{-(1-\beta) n/2}.
\end{align*}
By the boundedness of $\rho'$, 
\begin{align*}
I_3:=&
\Norm{\rho'}_{L^\infty} \one_{\{t> T-2^{-n}\}}  \bigg(\int_{0}^t\ud s\iint_{\R^{2d}} \ud y \ud y'\: G(t-s,x-y)G(t-s,x-y')f(y-y')\\
&\times 
\sup_{|\mathbf{z}|\vee |\mathbf{z}'|\le \kappa}
\Norm{\widehat{u}_{\mathbf{z}}^{n,i}(s,y)-\widehat{u}_{\mathbf{z}'}^{n,i}(s,y)}_{p}
\\
&\times  
\sup_{|\mathbf{z}|\vee |\mathbf{z}'|\le \kappa} 
\Norm{\widehat{u}_{\mathbf{z}}^{n,i}(s,y')-\widehat{u}_{\mathbf{z}'}^{n,i}(s,y')}_{p}\bigg)^{1/2}.
\end{align*} 
Similarly, we have that 
\begin{align*}
I_4:=&
\LIP_{\rho'} \int_{0}^t\ud s\iint_{\R^{2d}} \ud y \ud y'\: G(t-s,x-y)\InPrd{\mathbf{1},\mathbf{h}_n(s,y')}f(y-y')\\
&\times 
\sup_{|\mathbf{z}|\vee |\mathbf{z}'|\le \kappa}
\Norm{\widehat{u}_{\mathbf{z}}^{n}(s,y)-\widehat{u}_{\mathbf{z}'}^{n}(s,y)}_{2p}
\Norm{\widehat{u}_{\mathbf{z}}^{n,i}(s,y)}_{2p}\\
\le &C |\mathbf{z}-\mathbf{z}'|\one_{\{t> T-2^{-n}\}}  \left(\Psi_n(t,x;0)+\Psi_n(t,x;2)\right),
\end{align*}
and 
\begin{align*}
I_5:=&
\Norm{\rho'}_{L^\infty} \int_{0}^t\ud s\iint_{\R^{2d}} \ud y \ud y'\: G(t-s,x-y)\InPrd{\mathbf{1},\mathbf{h}_n(s,y')}f(y-y')\\
&\times 
\sup_{|\mathbf{z}|\vee |\mathbf{z}'|\le \kappa}
\Norm{\widehat{u}_{\mathbf{z}}^{n,i}(s,y)-\widehat{u}_{\mathbf{z}'}^{n,i}(s,y)}_p,
\end{align*}
and 
\begin{align*}
I_6:=&
\Norm{\rho'}_{L^\infty} \int_{0}^t\ud s\iint_{\R^{2d}} \ud y \ud y'\: G(t-s,x-y) f(y-y')\\
&\times 
\sup_{|\mathbf{z}|\vee |\mathbf{z}'|\le \kappa} \left(
\Norm{\widehat{u}_{\mathbf{z}'}^{n,i}(s,y)}_{p} \InPrd{\mathbf{z}-\mathbf{z}',\mathbf{h}_n(s,y')}\right)\\
\le &C |\mathbf{z}-\mathbf{z}'|\one_{\{t> T-2^{-n}\}}  \left(\Psi_n(t,x;0)+\Psi_n(t,x;1)\right).
\end{align*}

Now we group terms in order to apply the Picard iteration. 
By Lemma \ref{L:J^k->J}, \eqref{E:mu*}, \eqref{E:hG-J0} and \eqref{E:hG-J123}, we see that 
\begin{align}\label{E_:1246}
\sum_{i=1,2,4,6} I_i  \le C J_0^*(t,x)|\mathbf{z}-\mathbf{z}'|. 
\end{align}
Hence, by the same Picard iteration as in the proof of Lemma \ref{L:Moment-U} to see that
\begin{align}\label{E2_:IncHatU}
\sup_{|\mathbf{z}|\vee|\mathbf{z}|'\le \kappa } \Norm{\widehat{u}^{n,i}_{\mathbf{z}}(t,x)-\widehat{u}^{n,i}_{\mathbf{z}'}(t,x)}_p\le 
C |\mathbf{z}-\mathbf{z}'| J_0^*(t,x).
\end{align}

Finally, plugging this upper bound back to the upper bounds for $I_3$ and $I_5$ proves \eqref{E2:IncHatU}.

{\bigskip\bf\noindent Step 3.~} The proof for \eqref{E3:IncHatU} is similar to Step 2. We have, instead of six, fourteen terms:
\[
\sup_{|\mathbf{z}|\vee|\mathbf{z}|'\le \kappa } \Norm{\widehat{u}^{n,i,k}_{\mathbf{z}}(t,x)-\widehat{u}^{n,i,k}_{\mathbf{z}'}(t,x)}_p
\le C \sum_{j=1}^{14} I_j.
\]
In the following, we will specify each of these $I_j$ and give estimates on them. 
Recall that we can write the six parts of $\widehat{u}^{n,i,k}_{\mathbf{z}}(t,x)$ in \eqref{E:hatUnik} as
\begin{align}\label{E_:Unik6}
\widehat{u}^{n,i,k}_{\mathbf{z}}(t,x) =
\theta^{n,i,k}_{\mathbf{z}}(t,x) +\theta^{n,k,i}_{\mathbf{z}}(t,x)+ \sum_{\ell=1}^4 U_\ell^n(t,x).
\end{align}

{\noindent(1-2)} By the Lipschitz continuity and the boundedness of $\rho'$,
\begin{align*}
I_1:=&\sup_{|\mathbf{z}|\vee|\mathbf{z}|'\le \kappa } \Norm{\theta_{\mathbf{z}}^{n,i,k}(t,x)-\theta_{\mathbf{z}'}^{n,i,k}(t,x)}_p \\
\le&
\LIP_{\rho'} \int_{0}^t\ud s\iint_{\R^{2d}} \ud y \ud y'\: G(t-s,x-y)f(y-y') h_n^i(s,y')\\
&\qquad \times 
\sup_{|\mathbf{z}|\vee|\mathbf{z}|'\le \kappa }  
\left(
\Norm{\widehat{u}_{\mathbf{z}}^{n}(s,y)-\widehat{u}_{\mathbf{z}'}^{n}(s,y)}_{2p} 
\Norm{\widehat{u}^{n,k}_{\mathbf{z}}(s,y)}_{2p} \right)
\\
&+\Norm{\rho'}_{L^\infty} \int_{0}^t\ud s\iint_{\R^{2d}} \ud y \ud y'\: G(t-s,x-y)f(y-y') h_n^i(s,y')\\
&\qquad\times 
\sup_{|\mathbf{z}|\vee|\mathbf{z}|'\le \kappa } 
\left(
\Norm{\widehat{u}_{\mathbf{z}}^{n,k}(s,y)-\widehat{u}_{\mathbf{z}'}^{n,k}(s,y)}_p \right) \\
\le& C|\mathbf{z}-\mathbf{z}'| \left(\Psi_n(t,x;0)+\Psi_n(t,x;2)+\Psi_n^*(t,x;1)\right)\\
\le &C|\mathbf{z}-\mathbf{z}'| \Psi_n^*(t,x;1),
\end{align*}
where we have applied \eqref{E2_:IncHatU}, \eqref{E:supxKui} and \eqref{E_:HolderU} in the second inequality and Lemma \ref{L:Star} in the last inequality.
Similarly, 
\begin{align*} 
 I_2:=&\sup_{|\mathbf{z}|\vee|\mathbf{z}|'\le \kappa } \Norm{\theta_{\mathbf{z}}^{n,k,i}(t,x)-\theta_{\mathbf{z}'}^{n,k,i}(t,x)}_p \le C|\mathbf{z}-\mathbf{z}'| \Psi_n^*(t,x;1).
\end{align*}

{\medskip\noindent(3-5)} Terms from $I_3$ to $I_5$ come from $U_1$. By the Lipschitz continuity of $\rho''$, 
\begin{align*}
 I_3^2:=& \LIP_{\rho''}^2 \one_{\{t> T-2^{-n}\}} \int_{T-2^{-n}}^t\ud s \iint_{\R^{2d}} \ud y\ud y' \: G(t-s,x-y) 
 G(t-s,x-y') f(y-y')\\
 &\times 
 \sup_{|\mathbf{z}|\vee|\mathbf{z}|'\le \kappa } \left( 
 \Norm{\widehat{u}^{n}_{\mathbf{z}}(s,y)-\widehat{u}^{n}_{\mathbf{z'}}(s,y)}_{3p}\Norm{\widehat{u}^{n,i}_{\mathbf{z}}(s,y)}_{3p}
 \Norm{\widehat{u}^{n,k}_{\mathbf{z}}(s,y)}_{3p}
 \right)\\ 
 &\times 
 \sup_{|\mathbf{z}|\vee|\mathbf{z}|'\le \kappa } \left( 
 \Norm{\widehat{u}^{n}_{\mathbf{z}}(s,y')-\widehat{u}^{n}_{\mathbf{z'}}(s,y')}_{3p}
 \Norm{\widehat{u}^{n,i}_{\mathbf{z}}(s,y')}_{3p}
 \Norm{\widehat{u}^{n,k}_{\mathbf{z}}(s,y')}_{3p}
 \right).
\end{align*}
By \eqref{E2_:IncHatU} and \eqref{E:supxKui}, and by the fact that $1+J_0(t,x)\le C J_0^*(t,x)$, we see that 
\begin{align*}
 I_3\le & C |\mathbf{z}-\mathbf{z}'|\one_{\{t> T-2^{-n}\}}\bigg( \int_{T-2^{-n}}^t\ud s \iint_{\R^{2d}} \ud y\ud y'\: f(y-y') \\
 &\times G(t-s,x-y)  G(t-s,x-y') 
 J_0^*(s,y)^3 J_0^*(s,y')^3 \bigg)^{1/2}\\
 \le & C |\mathbf{z}-\mathbf{z}'|\one_{\{t> T-2^{-n}\}}  J_0^*(3t,x)^{3} 2^{-(1-\beta) n/2},
\end{align*}
where the last inequality is due to \eqref{E:hG-GGJJ} applied to $\mu^*$ and 
Assumption \ref{A:Rate-Sharp}.
Similarly, 
\begin{align*}
 I_4:=& \Norm{\rho''}_{L^\infty}^2\one_{\{t> T-2^{-n}\}}\bigg( \int_{T-2^{-n}}^t\ud s \iint_{\R^{2d}} \ud y\ud y' \: G(t-s,x-y) 
 G(t-s,x-y') f(y-y')\\
 &\times 
 \sup_{|\mathbf{z}|\vee|\mathbf{z}|'\le \kappa } \left( 
 \Norm{\widehat{u}^{n,i}_{\mathbf{z}}(s,y)-\widehat{u}^{n,i}_{\mathbf{z'}}(s,y)}_{2p}
 \Norm{\widehat{u}^{n,k}_{\mathbf{z}}(s,y)}_{2p}\right) \\ 
 &\times 
 \sup_{|\mathbf{z}|\vee|\mathbf{z}|'\le \kappa } \left( 
 \Norm{\widehat{u}^{n,i}_{\mathbf{z}}(s,y')-\widehat{u}^{n,i}_{\mathbf{z'}}(s,y')}_{2p}
  \Norm{\widehat{u}^{n,k}_{\mathbf{z}}(s,y')}_{2p}\right)\bigg)^{1/2}\\ 
 \le & C |\mathbf{z}-\mathbf{z}'|\one_{\{t> T-2^{-n}\}}  J_0^*(2t,x)^{2} 2^{-(1-\beta) n/2},
\end{align*}
and 
\begin{align*}
 I_5:=& \Norm{\rho''}_{L^\infty}^2\one_{\{t> T-2^{-n}\}}\bigg( \int_{T-2^{-n}}^t\ud s \iint_{\R^{2d}} \ud y\ud y' \: G(t-s,x-y) 
 G(t-s,x-y') f(y-y')\\
 &\times 
 \sup_{|\mathbf{z}|\vee|\mathbf{z}|'\le \kappa } \left( 
 \Norm{\widehat{u}^{n,i}_{\mathbf{z'}}(s,y)}_{2p}
 \Norm{\widehat{u}^{n,k}_{\mathbf{z}}(s,y)-\widehat{u}^{n,k}_{\mathbf{z'}}(s,y)}_{2p}\right)\\ 
 &\times 
 \sup_{|\mathbf{z}|\vee|\mathbf{z}|'\le \kappa } \left( 
 \Norm{\widehat{u}^{n,i}_{\mathbf{z'}}(s,y')}_{2p}
  \Norm{\widehat{u}^{n,k}_{\mathbf{z}}(s,y')-\widehat{u}^{n,k}_{\mathbf{z'}}(s,y')}_{2p}\right)\bigg)^{1/2}\\ 
 \le & C |\mathbf{z}-\mathbf{z}'|\one_{\{t> T-2^{-n}\}}  J_0^*(2t,x)^{2}  2^{-(1-\beta) n/2}.
\end{align*}

{\medskip\noindent(6-9)} Terms from $I_6$ to $I_9$ come from $U_2$. By the Lipschitz continuity of $\rho''$, 
\begin{align*}
 I_6:=& \LIP_{\rho''}\int_{0}^t\ud s \iint_{\R^{2d}} \ud y\ud y' \: G(t-s,x-y) 
 G(t-s,x-y') \InPrd{1,\mathbf{h}_n(s,y')}f(y-y')\\
 &\times 
 \sup_{|\mathbf{z}|\vee|\mathbf{z}|'\le \kappa } \left( 
 \Norm{\widehat{u}^{n}_{\mathbf{z'}}(s,y)-\widehat{u}^{n}_{\mathbf{z'}}(s,y)}_{3p}
 \Norm{\widehat{u}^{n,i}_{\mathbf{z}}(s,y)}_{3p}
 \Norm{\widehat{u}^{n,k}_{\mathbf{z}}(s,y)}_{3p}\right)\\
 \le & C|\mathbf{z}-\mathbf{z}'| \Psi_n^*(t,x;3),
\end{align*}
and  by the boundedness of $\rho''$,
\begin{align*}
 I_7:=& \Norm{\rho''}_{L^\infty} \int_{0}^t\ud s \iint_{\R^{2d}} \ud y\ud y' \: G(t-s,x-y) 
 G(t-s,x-y') \InPrd{1,\mathbf{h}_n(s,y')}f(y-y')\\
 &\times 
 \sup_{|\mathbf{z}|\vee|\mathbf{z}|'\le \kappa } \left( 
  \Norm{\widehat{u}^{n,i}_{\mathbf{z}}(s,y)-\widehat{u}^{n,i}_{\mathbf{z'}}(s,y)}_{2p}
 \Norm{\widehat{u}^{n,k}_{\mathbf{z}}(s,y)}_{2p}\right)\\
 \le & C|\mathbf{z}-\mathbf{z}'| \Psi_n^*(t,x;2),
\end{align*}
and 
\begin{align*}
 I_8:=& \Norm{\rho''}_{L^\infty}  \int_{0}^t\ud s \iint_{\R^{2d}} \ud y\ud y' \: G(t-s,x-y) 
 G(t-s,x-y') \InPrd{1,\mathbf{h}_n(s,y')}f(y-y')\\
 &\times 
 \sup_{|\mathbf{z}|\vee|\mathbf{z}|'\le \kappa } \left(
 \Norm{\widehat{u}^{n,i}_{\mathbf{z'}}(s,y)}_{2p}
 \Norm{\widehat{u}^{n,k}_{\mathbf{z}}(s,y)-\widehat{u}^{n,k}_{\mathbf{z'}}(s,y)}_{2p}\right)\\
 \le & C|\mathbf{z}-\mathbf{z}'| \Psi_n^*(t,x;2),
\end{align*}
and 
\begin{align*}
 I_9:=& \Norm{\rho''}_{L^\infty}\int_{0}^t\ud s \iint_{\R^{2d}} \ud y\ud y' \: G(t-s,x-y) 
 G(t-s,x-y') f(y-y')\\
 &\times 
 \sup_{|\mathbf{z}|\vee|\mathbf{z}|'\le \kappa } \left(
  \Norm{\widehat{u}^{n,i}_{\mathbf{z'}}(s,y)}_{2p}
 \Norm{\widehat{u}^{n,k}_{\mathbf{z'}}(s,y)}_{2p}\InPrd{\mathbf{z}-\mathbf{z'},\mathbf{h}_n(s,y')}\right)\\
 \le & C|\mathbf{z}-\mathbf{z}'| \Psi_n^*(t,x;2).
\end{align*}

{\medskip\noindent(10-11)} Terms for $I_{10}$ and $I_{11}$ come from $U_3$. By the Lipschitz continuity of $\rho'$,
\begin{align*}
 I_{10}:=& \LIP_{\rho'} \one_{\{t>T-2^{-n}\}}\bigg( \int_{T-2^{-n}}^t\ud s \iint_{\R^{2d}} \ud y\ud y' \: G(t-s,x-y) 
 G(t-s,x-y') f(y-y')\\
 &\times 
 \sup_{|\mathbf{z}|\vee|\mathbf{z}|'\le \kappa } \left(
 \Norm{\widehat{u}^{n}_{\mathbf{z}}(s,y)-\widehat{u}^{n}_{\mathbf{z'}}(s,y)}_{2p}\Norm{\widehat{u}^{n,i,k}_{\mathbf{z}}(s,y)}_{2p}
 \right)\\
 &\times 
 \sup_{|\mathbf{z}|\vee|\mathbf{z}|'\le \kappa } \left(
 \Norm{\widehat{u}^{n}_{\mathbf{z}}(s,y')-\widehat{u}^{n}_{\mathbf{z'}}(s,y')}_{2p}
  \Norm{\widehat{u}^{n,i,k}_{\mathbf{z}}(s,y')}_{2p}\right) \bigg)^{1/2}\\ 
 \le &  C|\mathbf{z}-\mathbf{z}'| \one_{\{t> T-2^{-n}\}} J_0^*(2t,x)^2 2^{-(1-\beta) n/2},
\end{align*}
and by the boundedness of $\rho'$,
\begin{align*}
 I_{11}:=& \Norm{\rho'}_{L^\infty} \one_{\{t> T-2^{-n}\}}\bigg(\int_{T-2^{-n}}^t\ud s \iint_{\R^{2d}} \ud y\ud y' \: G(t-s,x-y) 
 G(t-s,x-y') f(y-y')\\
 &\times 
 \sup_{|\mathbf{z}|\vee|\mathbf{z}|'\le \kappa } \left\{
 \Norm{\widehat{u}^{n,i,k}_{\mathbf{z}}(s,y)-
 \widehat{u}^{n,i,k}_{\mathbf{z'}}(s,y)}_p
 \Norm{\widehat{u}^{n,i,k}_{\mathbf{z}}(s,y')-
 \widehat{u}^{n,i,k}_{\mathbf{z'}}(s,y')}_p
 \right\}\bigg)^{1/2}.
\end{align*}

{\medskip\noindent(12-14)} Terms from $I_{12}$ to $I_{14}$ come from $U_4$. In particular,  \begin{align*}
 I_{12}:=& \LIP_{\rho'}\int_{0}^t\ud s \iint_{\R^{2d}} \ud y\ud y' 
\:G(t-s,x-y) \InPrd{\mathbf{1},\mathbf{h}_n(s,y')}f(y-y')\\
 &\times 
 \sup_{|\mathbf{z}|\vee|\mathbf{z}|'\le \kappa } \left(
 \Norm{\widehat{u}^{n}_{\mathbf{z'}}(s,y)-\widehat{u}^{n}_{\mathbf{z'}}(s,y)}_{2p}
 \Norm{\widehat{u}^{n,i,k}_{\mathbf{z}}(s,y)}_{2p}\right)\\
 \le & C|\mathbf{z}-\mathbf{z}'| \Psi_n^*(t,x;2),
\end{align*}
and by the boundedness of $\rho'$,
\begin{align*}
 I_{13}:=& \Norm{\rho'}_{L^\infty}\int_{0}^t\ud s \iint_{\R^{2d}} \ud y\ud y' \: 
G(t-s,x-y)\InPrd{\mathbf{1},\mathbf{h}_n(s,y')}f(y-y')\\
 &\times 
 \sup_{|\mathbf{z}|\vee|\mathbf{z}|'\le \kappa } 
 \Norm{\widehat{u}^{n,i,k}_{\mathbf{z}}(s,y)-\widehat{u}^{n,i,k}_{\mathbf{z'}}(s,y)}_{p},
\end{align*}
and 
\begin{align*}
 I_{14}:=& \Norm{\rho'}_{L^\infty}\int_{0}^t\ud s \iint_{\R^{2d}} \ud y\ud y' \: 
G(t-s,x-y) f(y-y')\\
 &\times 
  \sup_{|\mathbf{z}|\vee|\mathbf{z}|'\le \kappa } \left(
 \Norm{\widehat{u}^{n,i,k}_{\mathbf{z'}}(s,y)}_p \InPrd{\mathbf{z}-\mathbf{z'},\mathbf{h}_n(s,y')}
 \right)\\
 \le & C|\mathbf{z}-\mathbf{z}'| \Psi_n^*(t,x;1).
\end{align*}

Therefore, by Lemma \ref{L:Star}, we see that 
\begin{align*}
\mathop{\sum_{1\le i\le 14}}_{i\ne 11, i\ne 13} I_i 
& \le\sum_{\ell=1}^3\left(\Psi_n^*(t,x;\ell)+2^{-(1-\beta) n/2} \one_{\{t> T-2^{-n}\}} J_0^*(\ell t,x)^\ell\right)\\
&\le \one_{\{t> T-2^{-n}\}}|\mathbf{z}-\mathbf{z}'| J_0^{**}(t,x).
\end{align*}
Together with $I_{11}$ and $I_{13}$, we can apply the same Picard iteration scheme as that in the proof of Lemma \ref{L:Moment-U} to see that 
\[
\sup_{|\mathbf{z}|\vee|\mathbf{z}|'\le \kappa } \Norm{\widehat{u}^{n,i,k}_{\mathbf{z}}(t,x)-\widehat{u}^{n,i,k}_{\mathbf{z}'}(t,x)}_p\le C |\mathbf{z}-\mathbf{z}'| J_0^{**}(t,x).
\]
Then plugging this bounds back to $I_{11}$ and $I_{13}$ gives that 
\[
I_{11}\le C|\mathbf{z}-\mathbf{z}'| J_0^{**}(t,x) 2^{-(1-\beta) n/2}
\quad \text{and} \quad 
I_{13}\le C|\mathbf{z}-\mathbf{z}'| \Psi_n^{**}(t,x;1).
\]
Finally, we can use Lemma \ref{L:Star} to upgrade the moment bounds for $I_i$, $i\not\in\{11,13\}$, to those with double augmented initial measure $\mu^{**}$. 
With this, we complete the proof of Lemma \ref{L:HolderInZ}.
\end{proof}

\subsection{Almost convergence of \texorpdfstring{$\widehat{u}^{n,i}_0(T,x_i)$}{} to \texorpdfstring{$\rho(u(T,x_i))$}{}}
The aim of this part is to prove the following convergence
\[
\lim_{n\rightarrow\infty} \widehat{u}^{n,i}_0(T,x_i) = \rho(u(T,x_i)) \quad \text{a.s.,}
\]
which is used in step 2 of the proof of Theorem \ref{T:Pos}. 
This result is proved through the following three lemmas (see \eqref{E2:UniU-Close}): 

\begin{lemma}\label{L:Unz}
For any $\kappa >0$, $p \geq 2$, $n \in \mathbb{N}$ and $t\in [0,T]$ and $x\in \RR^d$ we have
\begin{equation}\label{E:UnU-Close}
\Norm{\sup_{|\mathbf{z}|\leq \kappa} |\widehat{u}^n_{\mathbf{z}}(t,x)-u(t,x)|}_p \le C\kappa 
\sum_{\ell=0,1}
\left(\Psi_n(t,x;\ell)+J_0^\ell(\ell t,x)2^{-(1-\beta) n/2}\right).
\end{equation}
\end{lemma}
\begin{proof}
We note that
\begin{align*}
\widehat{u}_{\mathbf{z}}^{n}(t,x) - u(t,x) &= \int_0^t \iint_{\R^{2d}} G(t-s, x-y)\rho(\widehat{u}_{\mathbf{z}}^n(s,y))\InPrd{\mathbf{z}, \mathbf{h}_n(s,y')} f(y-y') \ud s \ud y\ud y'\\
&\qquad  + \int_0^t \int_{\R^d} G(t-s, x-y)[\rho(\widehat{u}_{\mathbf{z}}^n(s,y)) - \rho(u(s,y))] 
W(\ud s \ud y).
\end{align*}
Hence, by the moment bounds for $\sup_{|\mathbf{z}|\le \kappa 
}\widehat{u}_{\mathbf{z}}^{n}(t,x)$ in \eqref{E:SupU}, we see that 
\begin{align*}
\sup_{|\mathbf{z}|\leq \kappa} &\Norm{\widehat{u}_{\mathbf{z}}^{n}(t,x) - 
u(t,x)}_p\\
\le& \quad C \kappa \int_0^t\ud s \iint_{\R^{2d}}\ud y\ud y'\:  G(t-s, x-y)\left(1+J_0(s,y)\right)\InPrd{\mathbf{1}, \mathbf{h}_n(s,y')} f(y-y') \\
& + C \Bigg(\int_0^t \ud s\iint_{\R^{2d}} \ud y\ud y'\: G(t-s, x-y) G(t-s, x-y') \\
&\quad \times 
\left(\sup_{|\mathbf{z}|\leq \kappa}  
\Norm{\widehat{u}_{\mathbf{z}}^{n}(s,y) - u(s,y)}_p \right)
\left( \sup_{|\mathbf{z}|\leq \kappa}
\Norm{\widehat{u}_{\mathbf{z}}^{n}(s,y') - 
u(s,y')}_p
\right)
f(y-y') \Bigg)^{1/2}\\
=: &\quad  I_1 + I_2.
\end{align*}
Notice that 
\[
I_1 = C \kappa \left(\Psi_n(t,x)+\Psi_n(t,x;1)\right)\le C \kappa (1+J_0(t,x)).
\]
By the same Picard iteration as in the proof of Lemma \ref{L:Moment-U}, we see that 
\begin{align}\label{E:uhat-u}
\sup_{|\mathbf{z}|\leq \kappa} \Norm{\widehat{u}_{\mathbf{z}}^{n}(t,x)-
u(t,x)}_p \le C\kappa(1+J_0(t,x)).
\end{align}
Plug this upper bound back to $I_2$ to see that 
\[
I_2\le C\kappa (1+J_0(t,x))2^{-2(1-\beta)n},
\]
which proves \eqref{E:UnU-Close} with the supremum outside of the 
$L^p(\Omega)$-norm. Finally, thanks to \eqref{E:IncHatU}, one can apply the 
Kolmogorov continuity theorem to move the supremum inside the norm. 
This completes the proof. 
\end{proof}

\begin{lemma}\label{L:thetaRho}
For any $\kappa >0$, $1 \leq i \leq d$ and $n \in \mathbb{N}$, it holds that 
\begin{align}\label{E:thetaRho}
\Norm{\sup_{|\mathbf{z}|\leq \kappa} |\theta_{\mathbf{z}}^{n,i}(T, x_i)-\rho(u(T, x_i))|}_p\leq 
C 2^{-n\alpha/2} + C \kappa\left(\Psi_n(t,x_i) + 
\Psi_n(t,x_i;1)\right)  \,,
\end{align}
where $\alpha\in(0,1]$ is the parameter in condition \eqref{E:Dalang2}.
As a consequence (together with \eqref{E4:SupU}), for all $x\in \RR^d$, with 
probability one, 
\begin{equation}\label{E2:thetaRho}
\lim_{n \to \infty} \theta _0^{n,i}(T,x) = \rho(u(T, x_i))\one_{\{x = x_i\}}\,.
\end{equation}
\end{lemma}
\begin{proof}
By \eqref{E:thetazni}, we see that
\begin{align*}
& \theta_{\mathbf{z}}^{n,i}(T,x_i) - \rho(u(T,x_i))\\
&= \int_0^T \ud s\iint_{\R^{2d}}\ud y \ud y'\:
G(T-s, x_i-y)[\rho(\widehat{u}_{\mathbf{z}}^n(s,y))-\rho(u(T,x_i))]  h_n^i(s,y')f(y-y') .
\end{align*}
Hence, 
\begin{align*}
&\Norm{\sup_{|\mathbf{z}|\leq \kappa}| \theta_{\mathbf{z}}^{n,i}(T,x_i) - \rho(u(T,x_i))|}_p\\
 \le &\quad C \int_0^T\ud s \iint_{\R^{2d}}\ud y\ud y'\: G(T-s, x_i-y)
 \Norm{\sup_{|\mathbf{z}|\leq \kappa}|\widehat{u}_{\mathbf{z}}^n(s,y)-u(T,x_i)| }_p h_n^i(s,y')f(y-y')  \\
 \leq &\quad  C \int_0^T \ud s\iint_{\R^{2d}}\ud y\ud y'\: G(T-s, x_i-y)
 \Norm{\sup_{|\mathbf{z}|\leq \kappa}|\widehat{u}_{\mathbf{z}}^n(s,y)-u(s,y)| }_p h_n^i(s,y')f(y-y') \\
 &+C \int_0^T \ud s\iint_{\R^{2d}} \ud y\ud y'\: G(T-s, x_i-y)
 \Norm{u(s,y)-u(T, x_i)}_p h_n^i(s,y')f(y-y')\\
 :=&\quad I_1 + I_2\,. 
\end{align*}
By \eqref{E:uhat-u} in the proof of Lemma \ref{L:Unz}, we have
\begin{align*}
I_1 \leq& C \kappa \int_0^T \ud s\iint_{\R^{2d}}\ud y\ud y'\: G(T-s, x_i-y) (1+J_0(s,y)) h_n^i(s,y')f(y-y')\\
\le& C \kappa\left(\Psi_n(T,x_i) + \Psi_n(T,x_i;1)\right)\le C \kappa. 
\end{align*}
Applying the H\"older continuity of $u(s,y)$ (see Theorem \ref{T:Holder}), we have that
\begin{align*}
I_2 \leq &C c_n \int_{T-2^{-n}}^T\ud s \iint_{\R^{2d}} \ud y\ud y'\: G(T-s, x_i-y)G(T-s, x_i-y')\left(|T-s|^{\alpha/2}+|x_i-y|^\alpha\right)\\
=:& I_{2,1}+I_{2,2}.
\end{align*}
By the Plancherel theorem and recalling that $k(\cdot)$ is defined in \eqref{E:k}, we have that
\[
I_{2,1} = C c_n \int_0^{2^{-n}} s^{\alpha/2} k(2s)\ud s
\le  C c_n 2^{-n\alpha/2} V(2^{-n})\le C 2^{-n\alpha/2}.
\]
As for $I_{2,2}$, notice that 
\begin{align}\notag
G(t,x) |x|^\alpha \le & C t^{\alpha/2}
G(2t,x) \exp\left(-\frac{|x|^2}{4t}\right) \left|\frac{x}{\sqrt{t}}\right|^\alpha\\
\notag
&\le C t^{\alpha/2} G(2t,x)\left(\sup_{z \in \R} e^{-\frac{z^2}{4\pi}}|z|^\alpha\right)\\
&\le C t^{\alpha/2} G(2t,x).
\label{E:Gx-x}
\end{align}
We can apply the above inequality to see that
\begin{align*}
I_{2,2} \leq &C c_n \int_0^{2^{-n}}\ud s\: s^{\alpha/2}\iint_{\R^{2d}} \ud y\ud y'\: G(2s, x_i-y)G(s, x_i-y')\\
=& 
C c_n \int_0^{2^{-n}}\ud s\: s^{\alpha/2}\int_{\R^{d}}\widehat{f}(\ud\xi) \exp\left(-\frac{3s}{2}|\xi|^2\right)\\
\le & Cc_n 2^{-n\alpha/2} \int_0^{2^{-n}}\ud s\int_{\R^{d}}\widehat{f}(\ud\xi) \exp\left(-s|\xi|^2\right)\\
=&Cc_n 2^{-n\alpha/2} V(2^{-n})  \le C 2^{-n\alpha/2}.
\end{align*}

Combining the estimates of $I_1$ and $I_2$ proves \eqref{E:thetaRho}.
Finally, \eqref{E2:thetaRho} can be obtained by an application of the Borel-Cantelli lemma thanks to \eqref{E:thetaRho} when $x=x_i$ and \eqref{E4:SupU} when $x\ne x_i$. This completes the proof of Lemma \ref{L:thetaRho}.
\end{proof}
%

\begin{lemma}\label{L:UniU-Close}
For any $\kappa>0$, $1\leq i \leq m$ and $n \in \mathbb{N}$, it holds that 
\begin{equation}\label{E:UniU-Close}
\Norm{\sup_{|\mathbf{z}|\leq \kappa} |\widehat{u}_{\mathbf{z}}^{n,i}(T, x_i)- \rho (u (T, x_i))|}_p \leq C \left(2^{-\alpha n/2} + 2^{-(1-\beta) n/2}  + \kappa\right)\,,
\end{equation}
where $\alpha\in(0,1]$ is the parameter in condition \eqref{E:Dalang2}. As a consequence, for all $x\in \RR^d$, 
\begin{equation}\label{E2:UniU-Close}
\lim_{n \to \infty} \widehat{u}_0^{n,i}(T, x) = \rho(u(T, x_i))\one_{\{x=x_i\}} \quad a.s. 
\end{equation}
\end{lemma}

\begin{proof}
We begin by writing 
\begin{align*}
&\widehat{u}^{n,i}_{\mathbf{z}}(T,x) - \rho(u(T, x_i))\\
=&
\quad \theta_{\mathbf{z}}^{n,i}(T,x_i) - \rho(u(T, x_i))\\
&+
\int_{T-2^{-n}}^T \int_{\R^d} G(T-s,x_i-y)\rho'(\widehat{u}_{\mathbf{z}}^n(s,y))
\widehat{u}^{n,i}_{\mathbf{z}}(s,y) W(\ud s\ud y)\\
 &+
\int_{0}^T \ud s\iint_{\R^{2d}} \ud y\ud y'\: G(T-s,x_i-y)\rho'(\widehat{u}_{\mathbf{z}}^n(s,y))
\widehat{u}^{n,i}_{\mathbf{z}}(s,y)\InPrd{\mathbf{z},\mathbf{h}_n(s,y')}f(y-y') \\
:=& I_1 + I_2 +I_3\,.
\end{align*}
Lemma \ref{L:thetaRho} shows that 
\begin{align*}
\Norm{\sup_{|\mathbf{z}|\le \kappa} |I_1|}_p
& \leq 
C 2^{-n\alpha/2} + C \kappa\left(\Psi_n(T,x_i) + \Psi_n(T,x_i;1)\right)\\
& \leq 
C 2^{-n\alpha/2} + C \kappa.
\end{align*}
As for $I_2$, by \eqref{E2_:SupU}) and the boundedness of $\rho'$, we see that
\begin{align*}
&I_2({\bf z}) - I_2({\bf z'})\\
=&\int_{T-2^{-n}}^T \int_{\R^d} G(T-s,x_i-y)\left[\rho'(\widehat{u}_{\mathbf{z}}^n(s,y))
\widehat{u}^{n,i}_{\mathbf{z}}(s,y) - \rho'(\widehat{u}_{\mathbf{z'}}^n(s,y))
\widehat{u}^{n,i}_{\mathbf{z'}}(s,y)\right] W(\ud s\ud y)\\
= & \int_{T-2^{-n}}^T \int_{\R^d} G(T-s,x_i-y)\left[\rho'(\widehat{u}_{\mathbf{z}}^n(s,y))
 -\rho'( \widehat{u}_{\mathbf{z'}}^n(s,y))\right]
\widehat{u}^{n,i}_{\mathbf{z'}}(s,y) W(\ud s\ud y)\\
& + \int_{T-2^{-n}}^T \int_{\R^d} G(T-s,x_i-y)\rho'(\widehat{u}_{\mathbf{z'}}^n(s,y))\left[
\widehat{u}^{n,i}_{\mathbf{z}}(s,y) - 
\widehat{u}^{n,i}_{\mathbf{z'}}(s,y)\right] W(\ud s\ud y)\\
:=& I_{21}({\bf z}) + I_{21}({\bf z'})\,.
\end{align*}
For $I_{21}({\bf z})$, we note that 
\begin{align*}
\Norm{I_{21}}_p\leq &\bigg(\int_{T-2^{-n}}^T \ud s \iint_{\RR^d}  \ud y\ud y'\: 
G(T-s,x_i-y)G(T-s,x_i-y')\\
&\times\left\|\widehat{u}_{\mathbf{z}}^n(s,y)
 -\widehat{u}_{\mathbf{z'}}^n(s,y)\right\|_{2p}
\|\widehat{u}^{n,i}_{\mathbf{z'}}(s,y)\|_{2p}\\
&\times\left\|\widehat{u}_{\mathbf{z}}^n(s,y')
 -\widehat{u}_{\mathbf{z'}}^n(s,y')\right\|_{2p}
\|\widehat{u}^{n,i}_{\mathbf{z'}}(s,y')\|_{2p}f(y-y') \bigg)^{1/2}\,.
\end{align*}
From \eqref{E2_:IncHatU} and \eqref{E:supxKui} we see that 
\begin{align*}
\Norm{I_{21}}_p \leq & |z-z'|\left(\int_{T-2^{-n}}^T \ud 
s\iint_{\RR^{2d}} \ud y \ud y'\: G(T-s, x_i-y)G(T-s, x_i-y')J_0^*(s,y)^2 
J_0^*(s,y')^2  \right)^{1/2}\\
\leq & C  |z-z'| 2^{-(1-\beta)n/2}\,,
\end{align*}
where the last inequality follows from \eqref{E:hG-GGJJ} applied to $\mu^*$. 
Similarly we have 
$$
\Norm{I_{22}}_p \leq C |z-z'| 2^{-(1-\beta)n/2}\,.
$$
Thus, an application of Kolmogorov continuity theorem shows that 
$$
\left\|\sup_{|z|\leq \kappa}|I_2|\right\|_{p}\leq 2^{-(1-\beta)n/2}\,.
$$
The case for $I_3$ can be proved in a similar way:
\begin{align*}
\Norm{\sup_{|\mathbf{z}|\le\kappa}|I_3|}_p
&\le C \kappa \int_{0}^T\ud s\iint_{\R^{2d}}\ud y\ud y'\:
G(T-s,x_i-y)J_0^*(s,y)\InPrd{\mathbf{1},\mathbf{h}_n(s,y')}f(y-y')\\
&= C \kappa \Psi_n^*(T,x_i;1)\le C\kappa J_0^*(T,x_i) = C' \kappa.
\end{align*}
This proves \eqref{E:UniU-Close}. Finally, when $\kappa=0$, \eqref{E2:UniU-Close} is proved by an application of the Borel-Cantelli lemma thanks to \eqref{E:UniU-Close} when $x=x_i$ and \eqref{E2:SupU}. This completes the proof of Lemma \ref{L:UniU-Close}.
\end{proof}

\subsection{Conditional boundedness}
\label{SS:Bddedness}

The aim of this subsection  is to prove the following proposition.
\begin{proposition}\label{P:thetaBdd}
Let $y\in\R^d$ be a point chosen as in Theorem \ref{T:Pos}.
For all $\kappa>0$ and $r>0$,
there exists constant  $K>0$ such that
for all $1\le i,k\le m$,
\begin{align}
\label{E:thetabdd}
\lim_{n\rightarrow\infty}\bbP\Bigg(
\Big[&\:
\sup_{|\mathbf{z}|\le\kappa} \left|\widehat{u}_{\mathbf{z}}^{n}(T,x_i)\right|
\vee
\sup_{|\mathbf{z}|\le\kappa} \left|\widehat{u}_{\mathbf{z}}^{n,i}(T,x_i)\right|
\vee
\sup_{|\mathbf{z}|\le\kappa} \left|\widehat{u}_{\mathbf{z}}^{n,i,k}(T,x_i)\right|
\Big]
\le K\:\: \Big| \:\: Q_r\Bigg) =1,
\end{align}
where
\[
 Q_r:=
\Big\{\left|\{u(T,x_i)-y_i\}_{1\le i\le m}\right|\le r\Big\}.
\]
\end{proposition}

\subsubsection{Spatial H\"older continuity of \texorpdfstring{$\widehat{u}_{\mathbf{z}}^{n}(t,x)$}{} and its two derivatives}
We need first prove several lemmas.

\begin{lemma}\label{L:HolderuHatx}
For $T>0$, there exists $C=C(T)$ such that 
for all $n\in\bbN$, $p\ge 2$, $1\le i,k\le m$, $\kappa>0$, $t\in [T-2^{-n},T]$ and $x,y\in\R^d$,
\begin{align} \label{E1:Space}
\Norm{\sup_{|\mathbf{z}|\le\kappa}\left| \widehat{u}^n_{\mathbf{z}}(t,x)-\widehat{u}^n_{\mathbf{z}}(t,y)\right| }_p
&\le C \left(1+J_0(2t,x)+J_0(2t,y)\right)|x-y|^\alpha,\\
\label{E2:Space}
\Norm{\sup_{|\mathbf{z}|\le\kappa}\left|
\widehat{u}^{n,i}_{\mathbf{z}}(t,x)-\widehat{u}^{n,i}_{\mathbf{z}}(t,y)
\right|
}_p
&\le C \left(J_0^*(2t,x)+J_0^*(2t,y)\right)|x-y|^\alpha,\\
\label{E3:Space}
\Norm{
\sup_{|\mathbf{z}|\le\kappa}\left| 
\widehat{u}^{n,i,k}_{\mathbf{z}}(t,x)-\widehat{u}^{n,i,k}_{\mathbf{z}}(t,y)
\right|}_p
&\le  C \left(J_0^{**}(2t,x)+J_0^{**}(2t,y)\right)|x-y|^\alpha,
\end{align}
where $\alpha\in(0,1]$ is the parameter that is given in \eqref{E:Dalang2}.
\end{lemma}
\begin{proof}
We prove these three inequalities in three steps. The workhorse is the following two inequalities (see  Lemma 3.1 of \cite{CH16Comparison}): 
For all $\alpha\in (0,1]$, $t'\ge t>0$ and $x,y\in\R^d$, it holds that
\begin{align}\label{E:Gx-Gy}
\left|G(t,x)-G(t,y)\right|&\le \frac{C}{t^{\alpha/2}} 
\left[G(2t,x)+G(2t,y)\right]|x-y|^\alpha,\\
\left|G(t',x)-G(t,x)\right|&\le \frac{C}{t^{\alpha/2}} 
G(4t',x)|t'-t|^{\alpha/2}.
\label{E:Gt-Gs}
\end{align}
In the following, we will apply the above two inequalities with $\alpha$ that is given in \eqref{E:Dalang2}.

{\medskip\bf\noindent Step 1.~}
We first prove \eqref{E1:Space}. Notice that
\begin{align*}
\widehat{u}^n_{\mathbf{z}}(t,x)-\widehat{u}^n_{\mathbf{z}}(t,y)
=&J_0(t,x)-J_0(t,y)\\
&+\int_0^t\int_{\R^d} \left[G(t-s,x-z)-G(t-s,y-z)\right]\rho\left(\widehat{u}^n_{\mathbf{z}}(s,z)\right)W(\ud s\ud z)\\
&+
\sum_{i=1}^m z_i \int_0^t\ud s\iint_{\R^{2d}} \ud z\ud z'\: \left[G(t-s,x-z)-G(t-s,y-z)\right]\\
&\qquad \times \rho\left(\widehat{u}^n_{\mathbf{z}}(s,z)\right)h_n^i(s,z')f(z-z')\\
=:&I_1 + I_2 +I_3.
\end{align*}
By \eqref{E:Gx-Gy} 
\[
|I_1|\le \frac{C}{t^{\alpha/2}} \left(J_0(2t,x)+J_0(2t,y)\right)|x-y|^\alpha.
\]
Then use the fact that $t >T/2$ to absorb the factor $t^{-\alpha/2}$ into the constant.

Denote 
\begin{align}\label{E:widehatInitData}
\widetilde{\mu}(\ud x) := \mu(\ud x)+ \ud x\quad\text{and}\quad
\widetilde{J}_0(t,x) := (G(t,\cdot)*\widetilde{\mu})(x) = 1+J_0(t,x).
\end{align}
By \eqref{E:SupU}, we see that
\begin{align*}
\sup_{|\mathbf{z}|\le\kappa}\Norm{I_2}_p^2\le &\int_0^t\ud s\iint_{\R^{2d}} 
\ud z\ud z'\:f(z-z')\\
&\times \left(G(t-s,x-z)-G(t-s,y-z)\right)\widetilde{J}_0(s,z)\\
&\times \left(G(t-s,x-z')-G(t-s,y-z')\right)\widetilde{J}_0(s,z')\\
\le & C\left(J_0(2t,x)+J_0(2t,y)\right)^2 |x-y|^{2\alpha},
\end{align*}
where the last inequality is proved in Step 1 of the proof of Theorem 1.8 in \cite[Section 3]{CH16Comparison}.
then an application of Kolmogorov continuity theorem shows that 
$$
\left\|\sup_{|{\bf z}|\leq \kappa} |I_2|\right\|_p \leq C (J_0(2t, x)+ J_0(2t, 
y)) |x-y|^{\alpha}\,.
$$
As for $I_3$, by \eqref{E:SupU}, we see that
\begin{align}\label{E_:SpaceI3}
\begin{aligned}
\Norm{\sup_{|\mathbf{z}|\le\kappa}|I_3|}_p
\le & C \kappa \int_0^t\ud s \iint_{\R^{2d}}\ud z\ud z'\: 
\left|G(t-s,x-z)-G(t-s,y-z)\right|\\
&\qquad \times \widetilde{J}_0(s,z) \InPrd{\mathbf{1},\mathbf{h}_n(s,z')} f(z-z').
\end{aligned}
\end{align}
By the fact that $G(t,x)\le 2^{d/2} G(2t,x)$ and

we see that 
\[
\Norm{\sup_{|\mathbf{z}|\le\kappa}|I_3|}_p\le C\kappa |x-y|^\alpha
\left( \Theta(x) +  \Theta(y)\right),
\]
where 
\begin{align*}
\Theta(x):=&\sum_{i=1}^m \one_{\{t> T-2^{-n}\}}\int_{T-2^{-n}}^t\frac{\ud s}{(t-s)^{\alpha/2}}\int_{\R^d} \widetilde{\mu}(\ud \widetilde{z})  \iint_{\R^{2d}}\ud z\ud z'\: \\
&\times 
G(2(t-s),x-z) G(2s, z-\widetilde{z}) G(2(T-s),x_i-z') f(z-z'). 
\end{align*}
Then apply \eqref{E:GG} for the first two $G$'s and use the Fourier transform to see that 
\begin{align*}
 \Theta(x)\le & \sum_{i=1}^m \one_{\{t> T-2^{-n}\}}\int_{T-2^{-n}}^t\frac{\ud s}{(t-s)^{\alpha/2}}\int_{\R^d} \widetilde{\mu}(\ud \widetilde{z})  G(2t,x-\widetilde{z})\iint_{\R^{2d}}\ud z\ud z'\: \\
&\times 
G\left(\frac{2(t-s)s}{t},z-\widetilde{z}-\frac{s}{t}(x-\widetilde{z})\right)  G(2(T-s),x_i-z') f(z-z')\\
\le& C \sum_{i=1}^m \one_{\{t> T-2^{-n}\}}\widetilde{J}_0(2t,x)
\int_{0}^{2^{-n}+t-T}\frac{\ud s}{s^{\alpha/2}} \int_{\R^{d}} \widehat{f}(\ud \xi) \\
&\times \exp\left(-\left[\frac{s(t-s)}{t}+T-t+s\right]|\xi|^2\right)\\
\le& C \one_{\{t> T-2^{-n}\}}\widetilde{J}_0(2t,x)
\int_{0}^{2^{-n}+t-T}\frac{\ud s}{s^{\alpha/2}} \int_{\R^{d}} \widehat{f}(\ud \xi) \exp\left(-\frac{3}{2}s|\xi|^2\right),
\end{align*}
where in the last inequality we have used \eqref{E_:FracToLin} and the fact that $t\le T$. 
Notice that 
\begin{align*}
\int_{0}^{2^{-n}}\frac{\ud s}{s^{\alpha/2}} \int_{\R^{d}} \widehat{f}(\ud \xi) \exp\left(-\frac32s|\xi|^2\right) &\le e^{2^{-n}} \int_{0}^{\infty}\frac{\ud s}{s^{\alpha/2}} \int_{\R^{d}} \widehat{f}(\ud \xi) \exp\left(-s(1+|\xi|^2)\right)\\
&=C \int_{\R^d}\frac{\widehat{f}(\ud \xi)}{(1+|\xi|^2)^{1-\alpha/2}} <\infty,
\end{align*}
which implies that $\Theta(x)\le \one_{\{t> T-2^{-n}\}}\widetilde{J}_0(2t,x)$.
Therefore, 
\[
\Norm{\sup_{|\mathbf{z}|\le\kappa}|I_3|}_p\le C\kappa |x-y|^\alpha
\one_{\{t> T-2^{-n}\}}\left(1+J_0(2t,x)+J_0(2t,y)\right).
\]
Combining these bounds prove \eqref{E1:Space}. 

{\medskip\bf\noindent Step 2.~}
Now we prove \eqref{E2:Space}.
From \eqref{E:hatUni}, write the difference $\hat{u}^{n,i}_{\mathbf{z}}(t,x)-
\hat{u}^{n,i}_{\mathbf{z}}(t,y)$ in three parts as above.
By the moment bound \eqref{E:SupU}, we see that
the difference for $\theta^{n,i}_{\mathbf{z}}$ reduces to
\begin{align*}
\Norm{\sup_{|\mathbf{z}|\le\kappa}\left|
 \theta^{n,i}_{\mathbf{z}}(t,x)-\theta^{n,i}_{\mathbf{z}}(t,y)
 \right|}_p \le C \int_0^t\ud s &\iint_{\R^{2d}}\ud z\ud z'\:  \widetilde{J}_0(s,z) h_n^i(s,z')f(z-z')\\
 &\times  \left|G(t-s,x-z)-G(t-s,y-z)\right|.
\end{align*}
Comparing the right-hand side of the above inequality with that of \eqref{E_:SpaceI3}, we see that 
\[
\Norm{\sup_{|\mathbf{z}|\le\kappa}\left|
 \theta^{n,i}_{\mathbf{z}}(t,x)-\theta^{n,i}_{\mathbf{z}}(t,y)
 \right|}_p\le 
 C\kappa |x-y|^\alpha
\one_{\{t> T-2^{-n}\}}\left(1+J_0(2t,x)+J_0(2t,y)\right).
\]
Thanks to the boundedness of $\rho'$, by the same arguments using moment bound \eqref{E2:SupU} in the form of 
\[
\Norm{\sup_{|\mathbf{z}|\le \kappa} \left|\widehat{u}^{n,i}_{\mathbf{z}}(t,x)\right|}_p\le C J_0^*(t,x),
\]
both the second and third part can be proved in exactly the same way as  Step 1, except that $\widetilde{J}_0(t,x)$ should be replaced by $J_0^*(t,x)$. 
This proves \eqref{E2:Space}.

{\medskip\bf\noindent Step 3.~}
The result \eqref{E3:Space} can be proved in the same way. We leave the details for interested readers.
This completes the proof of Lemma \ref{L:HolderuHatx}.
\end{proof}

\subsubsection{Some Grownwall-type inequalities}

We will need the following lemma, which is Lemma A.3 in \cite{CHN16Density} with $\alpha\in (1,2]$ replaced by $\beta=1/\alpha$. Note that the range of $\alpha$ for Lemma A.3 {\it ibid.} could be any $\alpha >1$ just as Lemma A.2 {\it ibid.} Recall that the two-parameter {\it Mittag-Leffler function} is defined as 
\begin{align}
E_{\alpha,\beta}(z):=\sum_{k=0}^\infty \frac{z^k}{\Gamma(\alpha k +\beta)},\quad \alpha>0,\; \beta>0,\;
z\in\bbC.
\end{align}

\begin{lemma}(Lemma A.3 of \cite{CHN16Density})
\label{L:Contraction}
Suppose that $\beta\in [0,1)$, $\lambda>0$, $T>\epsilon>0$, and $\theta_\epsilon:\R_+\mapsto\R$ is a locally integrable function.
\begin{enumerate}
\item[(1)] If $f$ satisfies that
\begin{align}\label{E:Contraction}
f(t)= \theta_\epsilon(t) + \lambda \epsilon^{-(1-\beta)}\one_{\{t>T-\epsilon\}}\int_{T-\epsilon}^t (t-s)^{-\beta} f(s)\ud s
\end{align}
for all $t\in (0,T]$, then
\begin{align}\label{E2:Contraction}
f(t)= \theta_\epsilon(t)+\one_{\{t>T-\epsilon\}}\int_{T-\epsilon}^t K_{\lambda \epsilon^{-(1-\beta)}}(t-s)\theta_\epsilon(s)\ud s,
\end{align}
for all $t\in (0,T]$, where $K_{\lambda \epsilon^{-(1-\beta)}}(t)$ is defined as 
\[
K_\lambda(t):= t^{-\beta}\lambda \Gamma(1-\beta) 
E_{1-\beta,1-\beta}(t^{1-\beta}\lambda\Gamma(1-\beta)).
\]
Moreover, when $\theta_\epsilon(t)\ge 0$, if the equality in \eqref{E:Contraction} is replaced by ``$\le$'' (resp. ``$\ge$''),
then the equality in the conclusion \eqref{E2:Contraction} should be replaced by
``$\le$'' (resp. ``$\ge$'') accordingly.
\item[(2)] When $\theta_\epsilon(t)\ge 0$, for some nonnegative constant $C$ that depends on $\alpha$, $\lambda$ and $T$ (not on $\epsilon$),
we have
\begin{align}\label{E3:Contraction}
f(t)\le \theta_\epsilon(t)+C \one_{\{t>T-\epsilon\}}\epsilon^{-(1-1/\alpha)}
\int_{T-\epsilon}^t (t-s)^{-1/\alpha}\theta_\epsilon(s)\ud s,
\end{align}
for all $t\in [0,T]$.
Moreover, if $\theta_\epsilon(t)\equiv \theta_\epsilon$ is a nonnegative 
constant, then for the same constant $C$, it holds that
\begin{align}
 f(t)&\le C  \theta_\epsilon,\quad\text{for all $t\in[0,T]$}.
\label{E4:Contraction}
\end{align}
\end{enumerate}
\end{lemma}

\bigskip
Similar to \cite{CHN16Density}, we still need to introduce the some functionals:
for any function $\theta:\R_+\mapsto\R$, $n\in\bbN$, $t\in(0,T]$ and $\beta\in[0,1)$, define
\begin{align}\label{E:F}
F_n\left[\theta\right](t):=
\theta(t) + \left(1-\beta\right)\one_{\{t>T-2^{-n}\}}2^{n(1-\beta)}
\int_{T-2^{-n}}^t(t-s)^{-\beta}
\theta(s)\ud s,
\end{align}
and for $m\ge 1$,
\begin{align}\label{E:Fm}
F_{n}^m\left[\theta\right](t):=
m \left(1-\beta\right)\one_{\{t>T-2^{-n}\}}2^{mn(1-\beta)}
\int_{T-2^{-n}}^t(t-s)^{m(1-\beta)-1}
\theta(s)\ud s.
\end{align}
\begin{lemma}(Lemma 8.17 of \cite{CHN16Density})
\label{L:F}
For all $n,m,m'\ge 1$ and $T>0$, the following properties hold for all $t\in (0,T]$:
\begin{gather}
\label{E:F1}
F_n[\theta](t)=\theta(t)+F_n^1[\theta](t),\\
\label{E:F2}
\lim_{n\rightarrow\infty}F_n^m[\theta](t)\one_{\{t<T\}} = 0,\\
\label{E:F5}
F_n^m[|\theta|](t)\le \left(F_n^m[|\theta|^{m'}](t)\right)^{1/m'},\\
F_n^m\left[F_n^{m'}[\theta]\right](t) = C_{m,m',\alpha}\: F_n^{m+m'}[\theta](t)\,.
\label{E:F3}
\end{gather}
If $\theta$ is left-continuous at $T$, then
\begin{align}\label{E:F6}
\lim_{n\rightarrow\infty}F_n^m[\theta](T) = \theta(T).
\end{align}
\end{lemma}

\subsubsection{Proof of the conditionally boundedness (Proposition \ref{P:thetaBdd})}

Now we are ready to prove our last result, Proposition \ref{P:thetaBdd}, in order to complete the proof of Theorem \ref{T:Pos}.

\begin{proof}[Proof of Proposition \ref{P:thetaBdd}]
In this  proof  we assume  $t\in (0,T]$.
Throughout the proof, $p$ is an arbitrary number that is greater than or equal to $2$.
The proof consists of the following four steps. 

{\bigskip\noindent\bf Step 1.~}
We first prove \eqref{E:thetabdd} for $\widehat{u}^n_{\mathbf{z}}(T,x_i)$.
Notice that
\begin{align*}
&\widehat{u}^n_{\mathbf{z}}(t,x_i)-u(t,x_i)\\
=&\int_0^t\int_{\R^d} G(t-s,x_i-y)\left[\rho\left(\widehat{u}^n_{\mathbf{z}}(s,y)\right)-
\rho\left(u(s,y)\right)\right]W(\ud s\ud y)\\
&
+\sum_{j\ne i} z_j
\int_0^t\iint_{\R^{2d}} G(t-s,x_i-y)\rho\left(\widehat{u}^n_{\mathbf{z}}(s,y)\right)h_n^j(s,y') f(y-y')\ud s\ud y\ud y'\\
&+ z_i
\int_0^t\iint_{\R^{2d}} G(t-s,x_i-y)
\left[
\rho\left(\widehat{u}^n_{\mathbf{z}}(s,y)\right)
-
\rho\left(\widehat{u}^n_{\mathbf{z}}(s,x_i)\right)
\right]
h_n^i(s,y')f(y-y') \ud s\ud y\ud y'\\
&+ z_i
\int_0^t\iint_{\R^{2d}} G(t-s,x_i-y)
\left[
\rho\left(u(s,x_i)\right)
-
\rho\left(u(t,x_i)\right)
\right]
h_n^i(s,y')f(y-y')\ud s\ud y\ud y'\\
&+z_i\Psi^i_n(t,x_i)\rho\left(u(t,x_i)\right) \\
&+ z_i
\int_0^t\iint_{\R^{2d}} G(t-s,x_i-y)
\left[
\rho\left(\widehat{u}^n_{\mathbf{z}}(s,x_i)\right)
-
\rho\left(u(s,x_i)\right)
\right]
h_n^i(s,y')f(y-y') \ud s\ud y\ud y'\\
=:& \sum_{\ell=1}^6 I_\ell(t).
\end{align*}
Then by the Lipschitz continuity of $\rho$, 
\begin{align*}
|I_6|\le & |z_i|c_n\LIP_\rho 
\one_{\{t> T-2^{-n}\}}\int_{T-2^{-n}}^t\ud s \left|\widehat{u}^n_{\mathbf{z}}(s,x_i)- u(s,x_i)\right|\\
&\times \iint_{\R^{2d}} \ud y\ud y'\: G(t-s,x_i-y) G(T-s,x_i-y')f(y-y')\\
=& C c_n \one_{\{t> T-2^{-n}\}} \int_{T-2^{-n}}^t\ud s \left|\widehat{u}^n_{\mathbf{z}}(s,x_i)- u(s,x_i)\right|\\
&\times \int_{\R^d}\widehat{f}(\ud \xi) \exp\left(-\frac{1}{2}\left[t-s+T-s\right]|\xi|^2\right)\\
\le& C c_n \one_{\{t> T-2^{-n}\}} \int_{T-2^{-n}}^t\ud s \left|\widehat{u}^n_{\mathbf{z}}(s,x_i)- u(s,x_i)\right|\int_{\R^d}\widehat{f}(\ud \xi) \exp\left(-(t-s)|\xi|^2\right)\\
=& C c_n\one_{\{t> T-2^{-n}\}} \int^{t}_{T-2^{-n}} \ud s 
\left|\widehat{u}^n_{\mathbf{z}}(s,x_i)- u(s,x_i)\right|k(2(t-s)),
\end{align*}
where $k(t)$ is defined in \eqref{E:k}.
Then by Assumption \ref{A:Rate-Sharp} and Lemma \ref{L:Rate}, we 
have that $c_n\le C 2^{(1-\beta)n}$ and
\begin{align*}
\sup_{|\mathbf{z}|\le\kappa}
&\left|
\widehat{u}^n_{\mathbf{z}}(t,x_i)-u(t,x_i)
\right|\\
&\le \sum_{\ell=1}^5 \sup_{|\mathbf{z}|\le\kappa} |I_\ell(t)|
+
C 2^{n(1-\beta)}\one_{\{t>T-2^{-n}\}}\int_{T-2^{-n}}^t (t-s)^{-\beta}
\sup_{|\mathbf{z}|\le\kappa}\left|
\widehat{u}^n_{\mathbf{z}}(s,x_i)
- u(s,x_i)
\right|\ud s.
\end{align*}
Hence, by Lemma \ref{L:Contraction} (see \eqref{E3:Contraction}) and by \eqref{E:F}, with probability one,
\begin{align}\label{e1:uHat}
\begin{aligned}
 \sup_{|\mathbf{z}|\le\kappa}\left|
\widehat{u}^n_{\mathbf{z}}(t,x_i)-u(t,x_i)
\right|
&\le
C \sum_{\ell=1}^5 F_n\left[
\sup_{|\mathbf{z}|\le\kappa}|I_\ell(\cdot)|
\right](t)\\
&=:C M_n^*(t) +C F_n\left[
\sup_{|\mathbf{z}|\le\kappa}|I_5(\cdot)|
\right](t).
\end{aligned}
\end{align}
We estimate each term in the above sum separately. Using the same method as in 
the proof of Lemma \ref{L:Unz} or the proof of Lemma \ref{L:HolderInZ} for 
$I_1$ in proving \eqref{E:IncHatU}, an application of Kolmogorov continuity 
theorem shows that 
\[
\sup_{t\in[0,T]}\Norm{\sup_{|\mathbf{z}|\le\kappa}\left| I_1(t)\right|}_p\le  C 2^{-n(1-\beta)/2}.
\]
By \eqref{E:SupU} and \eqref{E:hG-Small}, and since $i\ne j$, we see that
\[
\sup_{t\in[0,T]}\Norm{\sup_{|\mathbf{z}|\le\kappa}
\left| I_2(t)\right| }_p\le  C 2^{-n\beta}.
\]
As for $I_3(t)$ and $I_4(t)$, we claim that
\begin{align}\label{E_:I3}
\sup_{t\in[0,T]}\Norm{\sup_{|\mathbf{z}|\le\kappa}
\left| I_\ell(t)\right| }_p\le  C 2^{-n\alpha/2},\quad\text{for $\ell=3,4.$}
\end{align}
We first prove the case for $I_3$. By the Minkowski inequality,
\begin{align*}
\Norm{\sup_{|\mathbf{z}|\le\kappa}
\left|I_3(t)\right|}_p\le
& C2^{n(1-\beta)}\one_{\{t> T-2^{-n}\}}\int_{T-2^{-n}}^t\ud s\iint_{\R^{2d}} \ud y\ud y'\:f(y-y') \\
&\times G(t-s,x_i-y)G(T-s,x_i-y')
\Norm{
\sup_{|\mathbf{z}|\le\kappa}
\left|
\widehat{u}^n_{\mathbf{z}}(s,y)-\widehat{u}^n_{\mathbf{z}}(s,x_i)
\right|}_p 
\end{align*}
Then by Lemma \ref{L:HolderuHatx}, for $s\in [T-2^{-n},t]$, 
\[
\Norm{\sup_{|\mathbf{z}|\le\kappa}
\left|\widehat{u}^n_{\mathbf{z}}(s,y)-\widehat{u}^n_{\mathbf{z}}(s,x_i)
\right|}_p\le
C(1+J_0(2s,y)) |y-x_i|^\alpha.
\]
Now we use the notation $\widetilde{\mu}$ and $\widetilde{J}_0(t,x)$ as in \eqref{E:widehatInitData}.
Thus,
\begin{align*}
\Norm{\sup_{|\mathbf{z}|\le\kappa}
\left|I_3(t)\right|}_p\le
& C2^{n(1-\beta)}\one_{\{t>T-2^{-n}\}}\int_{T-2^{-n}}^t\ud s\int_{\R^d}\widetilde{\mu}(\ud z)
\iint_{\R^{2d}} \ud y\ud y' \:f(y-y') \\
&\times G(t-s,x_i-y)G(T-s,x_i-y') G(2s, y-z) |y-x_i|^\alpha.
\end{align*}
Now we apply the inequality \eqref{E:Gx-x} to see that 
\begin{align*}
\Norm{\sup_{|\mathbf{z}|\le\kappa}
\left|I_3(t)\right|}_p\le
& C2^{n(1-\beta)}\one_{\{t>T-2^{-n}\}}\int_{T-2^{-n}}^t\ud s \: (t-s)^{\alpha/2} \int_{\R^d}\widetilde{\mu}(\ud z)
\iint_{\R^{2d}} \ud y\ud y' \:f(y-y') \\
&\times G(2(t-s),x_i-y)G(T-s,x_i-y') G(2s, y-z)\\
=& C2^{n(1-\beta)}\one_{\{t>T-2^{-n}\}}\int_{T-2^{-n}}^t\ud s\: (t-s)^{\alpha/2} \int_{\R^d}\widetilde{\mu}(\ud z) G(2t, x_i-z)\\
&\times \iint_{\R^{2d}} \ud y\ud y' \:f(y-y') G\left(\frac{2s(t-s)}{t},y-z-\frac{s}{t}(x_i-y)\right)G(T-s,x_i-y') \\
\le & C2^{n(1-\beta)}\widetilde{J}_0(2t,x_i)\one_{\{t>T-2^{-n}\}}\int_{0}^{2^{-n}}\ud s\: s^{\alpha/2}  \int_{\R^{d}} \widehat{f}(\ud\xi) \\
&\times 
\exp\left(-\frac{1}{2}\left(\frac{2s(t-s)}{t}+T-t+s\right)|\xi|^2\right),
\end{align*}
where we have applied \eqref{E:GG}. Then by \eqref{E_:FracToLin} and Assumption \ref{A:Rate-Sharp}, we see that 
\begin{align*}
\Norm{\sup_{|\mathbf{z}|\le\kappa}
\left|I_3(t)\right|}_p
\le & C2^{n(1-\beta)}\widetilde{J}_0(2t,x_i)\one_{\{t> T-2^{-n}\}}\int_{0}^{2^{-n}}\ud s\: s^{\alpha/2}  \int_{\R^{d}} \widehat{f}(\ud\xi) \exp\left(-s|\xi|^2\right)\\
\le & C2^{n(1-\beta)}\widetilde{J}_0(2t,x_i)\one_{\{t> T-2^{-n}\}}2^{-n\alpha/2}V_d(2^{-n})\\
\le & C2^{-n\alpha/2},
\end{align*}
which proves \eqref{E_:I3} for $\ell=3$.

As for $I_4$, by the H\"older continuity of $u(t,x)$ (see Steps 2 \& 3 of the proof of Theorem 1.8 in \cite{CH16Comparison}), we see that
\begin{align*}
\Norm{
\sup_{|\mathbf{z}|\le\kappa}\left|
I_4(t)\right|}_p \le
& C2^{n(1-\beta)} \one_{\{t> T-2^{-n}\}} \int_{T-2^{-n}}^t \ud s \Norm{u(s,x_i)-u(t,x_i)}_p  \\
&\times \iint_{\R^{2d}} \ud y \ud y'  G(t-s,x_i-y)G(T-s,x_i-y')f(y-y')\\
\le 
& C2^{n(1-\beta)}\one_{\{t> T-2^{-n}\}} \int_{T-2^{-n}}^t (t-s)^{-\beta} \Norm{u(s,x_i)-u(t,x_i)}_p \ud s\\
\le&
C2^{n(1-\beta)}\one_{\{t> T-2^{-n}\}} \int_{T-2^{-n}}^{t} (t-s)^{-\beta} (t-s)^{\alpha/2} \ud s\\
\le & C 2^{-n\alpha/2},
\end{align*}
which proves \eqref{E_:I3} for $\ell=4$.

Hence, the above computations show that
\[
\sup_{t\in[0,T]}\Norm{M_n^*(t)}_p = 
\sup_{t\in[0,T]}\Norm{\sum_{\ell=1}^4
F_n\left[\sup_{|\mathbf{z}|\le\kappa}|
I_\ell(\cdot)|
\right](t)}_p \le C 2^{-\frac{n}{2}\min(2\beta, 1-\beta ,\alpha) }.
\]
By the Borel-Cantelli lemma,
we see that, for all $t\in[0,T]$,
\begin{align}\label{e2:uHat}
\lim_{n\rightarrow\infty} M_n^*(t) = \lim_{n\rightarrow\infty}
\sum_{\ell=1}^4
F_n\left[
\sup_{|\mathbf{z}|\le\kappa}|I_\ell(\cdot)|
\right](t) = 0,\quad\text{a.s.}
\end{align}
As for $I_5$, since $\Psi_n(t,x_i)$ is bounded by one (see \eqref{E:hG-J0}), we have that
\begin{align}
\label{e3-0:uHat}
|\Psi_n^i(t,x_i)\rho(u(t,x_i))|\le
|\rho(u(t,x_i))|,\quad\text{a.s.}
\end{align}
Hence, for all $n\in\bbN$,
\begin{align}
F_n\left[|I_5(\cdot)|\right](t)
&\le F_n\big[\left|\rho(u(\cdot,x_i))\right|\big](t).
\label{e3:uHat}
\end{align}
Therefore, by combining \eqref{e1:uHat}, \eqref{e2:uHat} and \eqref{e3:uHat}
we have that
\begin{align}\label{e4:uHat}
\sup_{|\mathbf{z}|\le\kappa}
 \left|
 \widehat{u}_{\mathbf{z}}^n(t,x_i)
 -
 u(t,x_i)
 \right|\le &\:
 M_n^*(t)+C\left|\rho(u(t,x_i))\right|+C F_n^1\big[\left|\rho(u(\cdot,x_i))\right|\big](t)
\quad
 \text{a.s. }
\end{align}
for all $n\in\bbN$.
Finally, by letting $t=T$ and
sending $n\rightarrow\infty$, and by
the H\"older continuity of $s\mapsto \rho(u(s,x_i))$, we have that
\begin{align*}
\limsup_{n\rightarrow\infty}\sup_{|\mathbf{z}|\le\kappa}
 \left|
 \widehat{u}_{\mathbf{z}}^n(T,x_i)
\right|
 &\le
 \limsup_{n\rightarrow\infty}\sup_{|\mathbf{z}|\le\kappa}
 \left|
 \widehat{u}_{\mathbf{z}}^n(T,x_i)
 -
 u(T,x_i)
 \right| + |u(T,x_i)|\\
 & \le C |\rho(u(T,x_i))|+ |u(T,x_i)|, \quad\text{a.s.}
\end{align*}
This proves Proposition \ref{P:thetaBdd} for $\widehat{u}^n_{\mathbf{z}}(T,x)$.

{\bigskip\noindent\bf Step 2.~}
Now we prove Proposition \ref{P:thetaBdd} for $\widehat{u}^{n,i}_{\mathbf{z}}(T,x)$.
Notice that
\begin{align*}
&\widehat{u}^{n,i}_{\mathbf{z}}(t,x_i) - \rho(u(t,x_i))\\
=& \:\: \theta^{n,i}_{\mathbf{z}}(t,x_i)\\
&+\one_{\{t>T-2^{-n}\}}\int_{T-2^{-n}}^t\int_{\R^d} G(t-s,x_i-y)
\rho'\left(\widehat{u}^{n}_{\mathbf{z}}(s,y)\right)
\widehat{u}^{n,i}_{\mathbf{z}}(s,y)W(\ud s\ud y)\\
&+\sum_{j\ne i}z_j\int_0^t\iint_{\R^{2d}} G(t-s,x_i-y)
\rho'\left(\widehat{u}^{n}_{\mathbf{z}}(s,y)\right)
\widehat{u}^{n,i}_{\mathbf{z}}(s,y)
h_n^j(s,y')f(y-y')\ud s\ud y\ud y'\\
&+z_i\int_0^t\iint_{\R^{2d}} G(t-s,x_i-y)
\rho'\left(\widehat{u}^{n}_{\mathbf{z}}(s,y)\right)
\left[\widehat{u}^{n,i}_{\mathbf{z}}(s,y)-\widehat{u}^{n,i}_{\mathbf{z}}(s,x_i)\right]
h_n^i(s,y')f(y-y')\ud s\ud y\ud y'\\
&+z_i\int_0^t\iint_{\R^{2d}} G(t-s,x_i-y)
\rho'\left(\widehat{u}^{n}_{\mathbf{z}}(s,y)\right)
\left[\rho\left(u(s,x_i)\right)-\rho\left(u(t,x_i)\right)\right]
h_n^i(s,y')f(y-y')\ud s\ud y\ud y'\\
&+z_i\rho\left(u(t,x_i)\right)\Psi_n^i(t,x_i)\\
&+z_i\int_0^t\iint_{\R^{2d}} G(t-s,x_i-y)
\rho'\left(\widehat{u}^{n}_{\mathbf{z}}(s,y)\right)
\left[\widehat{u}^{n,i}_{\mathbf{z}}(s,x_i)-\rho\left(u(s,x_i)\right)\right]
h_n^i(s,y')f(y-y')\ud s\ud y\ud y'\\
=:&\sum_{\ell=0}^6 I_\ell(t).
\end{align*}
By similar arguments as those in Step 1, we see that
\begin{align}
\label{e0:uHatD}
\begin{aligned}
\sup_{|\mathbf{z}|\le\kappa}\left|
\widehat{u}^{n,i}_{\mathbf{z}}(t,x_i)-\rho(u(t,x_i))
\right|
\le&  \sum_{\ell=0}^5 \sup_{|\mathbf{z}|\le\kappa} |I_\ell(t)|
+
C \one_{\{t>T-2^{-n}\}}2^{n(1-\beta)}\\
&\times \int_{T-2^{-n}}^t(t-s)^{-\beta}
\sup_{|\mathbf{z}|\le\kappa}
\left|
\widehat{u}^{n,i}_{\mathbf{z}}(s,x_i)
-
\rho(u(s,x_i))
\right|\ud s,
\end{aligned}
\end{align}
and (by Lemma \ref{L:Contraction}), with probability one, 
\begin{align}\label{e1:uHatD}
\begin{aligned}
\sup_{|\mathbf{z}|\le\kappa}
\left|
\widehat{u}^{n,i}_{\mathbf{z}}(t,x_i)-\rho(u(t,x_i))
\right|
& \le
C \sum_{\ell=0}^5 F_n\left[
\sup_{|\mathbf{z}|\le\kappa}
|I_\ell(\cdot)|\right](t).
\end{aligned}
\end{align}

We first consider $I_0(t)$. Decompose $z_i\theta^{n,i}_{\mathbf{z}}(t,x_i)$ into three parts
\begin{align*}
z_i\theta^{n,i}_{\mathbf{z}}(t,x_i)= &
\widehat{u}^{n}_{\mathbf{z}}(t,x_i)-u(t,x_i)\\
&+\int_0^t\int_{\R^d} G(t-s,x_i-y) \left[\rho\left(u(s,y)\right)-\rho\left(\widehat{u}^n_{\mathbf{z}}(s,y)\right)\right]W(\ud s\ud y)\\
&-\sum_{j\ne i} z_j \theta_{\mathbf{z}}^{n,j}(t,x_i)\\
=:& I_{0,1}(t)-I_{0,2}(t) - I_{0,3}(t).
\end{align*}
Notice that $I_{0,2}(t)$ is equal to $I_1(t)$ in Step 1.
From \eqref{E4:SupU} and Proposition \ref{P:Psi}, we see that
\begin{align*}
\sup_{t\in [0,T]}\Norm{F_n\left[\sup_{|\mathbf{z}|\le\kappa}
|I_{0,3}(\cdot)|\right](t)}_p \le C 2^{-\beta n}.
\end{align*}
As for $I_{0,1}(t)$,
by \eqref{e4:uHat} and \eqref{E:F3}, we see that with probability one, 
\begin{align}
\label{e4:uHatD}
F_n\left[\sup_{|\mathbf{z}|\le\kappa}\left|I_{0,1}(\cdot)\right|\right](t)
\le&
F_n[M_n^{*}](t)+C|\rho(u(t,x_i))|+C\sum_{\ell=1}^2 F_n^\ell \big[|\rho(u(\cdot,x_i))|\big](t).
\end{align} 

The terms $I_1$ to $I_4$ are similar to those in Step 1 and by the same arguments as those in Step 1 and using the fact that $\rho'$ is bounded, we see that
\[
\sup_{t\in[0,T]}\Norm{\sum_{\ell=1}^4F_n\left[
\sup_{|\mathbf{z}|\le\kappa}|I_\ell(\cdot)|\right](t)
}_p \le C 2^{-\frac{n}{2}\min(2\beta,1-\beta,\alpha)}.
\]
Set 
\[
M^{**}_n(t) := 
\sum_{\ell=0}^1 F_n\left[
\sup_{|\mathbf{z}|\le\kappa}
|I_{0,\ell}(\cdot)|
\right](t)
+\sum_{\ell=1}^4 F_n\left[
\sup_{|\mathbf{z}|\le\kappa}
|I_\ell(\cdot)|
\right](t) + F_n[M_n^*](t). 
\]
Hence, 
\[
\sup_{t\in[0,T]}\Norm{M_n^{**}(t)}_p \le C 
2^{-\frac{n}{2}\min(2\beta,1-\beta,\alpha)}
\]
and by the Borel-Cantelli lemma, 
\begin{align}\label{e2:uHatD}
\lim_{n\rightarrow\infty}
M^{**}_n(t) = 0,\quad\text{a.s. for all $t\in[0,T]$.}
\end{align}
The term $I_5$ part is identical to the term $I_5$ in Step 1, so that we have the bound \eqref{e3:uHat}.

Combining \eqref{e1:uHatD}, \eqref{e4:uHatD}, \eqref{e2:uHatD} and
\eqref{e3:uHat} shows that
for all $t\in(0,T]$ and $n\in\bbN$,
\begin{align}
\label{e5:uHatD}
\begin{aligned}
\sup_{|\mathbf{z}|\le\kappa}
&\left|
 \widehat{u}_{\mathbf{z}}^{n,i}(t,x_i)
 - \rho(u(t,x_i))
\right|
\\
&\le
C M_n^{**}(t)+C\left|\rho(u(t,x_i))\right|+C \sum_{\ell=1}^2 F_n^\ell \left[\left|\rho(u(\cdot,x_i))\right|\right](t)\quad\text{a.s.}
\end{aligned}
\end{align}
Finally, by letting $t=T$ and sending $n\rightarrow\infty$,
and by the H\"older continuity of $s\mapsto\rho(u(s,x_i))$,
we have that
\[
\limsup_{n\rightarrow\infty}
\sup_{|\mathbf{z}|\le\kappa}\left|
 \widehat{u}_{\mathbf{z}}^{n,i}(T,x_i)
\right|\le C\left|\rho(u(T,x_i))\right|\quad\text{a.s.}
\]
This proves Proposition \ref{P:thetaBdd} for $\widehat{u}^{n,i}_{\mathbf{z}}(T,x_i)$.

{\bigskip\noindent\bf Step 3.~}
In this step, we will prove for all $t\in [0,T]$, with probability one,
\begin{align}\label{e1:thetaDD}
F_n\left[\sup_{|\mathbf{z}|\le\kappa}\left|\theta^{n,i,k}_{\mathbf{z}}(\cdot,x_i)\right|\right](t)
\le
C M_n^\dagger(t)+C\left|\rho(u(t,x_i))\right|
+C \sum_{\ell=1}^3 F_n^\ell \left[\left|\rho(u(\cdot,x_i))\right|\right](t)
\end{align}
with 
\[
M_n^\dagger(t):=
F_n\left[\sum_{j\ne i}\sup_{|\mathbf{z}|\le\kappa}\left|z_j \theta^{n,i,k}_{\mathbf{z}}(\cdot,x_i) \right|\right](t) + F_n[M^{**}_n](t)
\]
satisfying that
\[
\sup_{t\in [0,T]}\Norm{M_n^\dagger(t)}_p\le  
2^{-\frac{n}{2}\min(2\beta,1-\beta,\alpha)}\quad\text{and}
\quad 
\lim_{n\rightarrow\infty}M_n^\dagger(t) = 0,\quad\text{a.s. for all $t\in[0,T]$.}
\]

Indeed, by \eqref{E:thetaIK-bd}, we need only to consider the case when $k=i$ (see, e.g., the arguments leading to \eqref{E_:ThetaIJK} below).
Notice from \eqref{E:hatUni} that
\begin{align*}
z_i \theta^{n,i,i}_{\mathbf{z}}(t,x_i)
=& \widehat{u}^{n,i}_{\mathbf{z}}(t,x_i) -
\theta^{n,i}_{\mathbf{z}}(t,x_i) -\sum_{j\ne i}z_j \theta^{n,i,k}_{\mathbf{z}}(t,x_i)-I(t,x_i),
\end{align*}
where
\[
I(t,x_i):=\one_{\{t>T-2^{-n}\}}\int_{T-2^{-n}}^t\int_{\R^d} G(t-s,x_i-y)
\rho'\left(\widehat{u}^{n}_{\mathbf{z}}(s,y)\right)
\widehat{u}^{n,i}_{\mathbf{z}}(s,y)W(\ud s\ud y).
\]
By \eqref{e5:uHatD}, we see that
\begin{align*}
F_n\left[
\sup_{|\mathbf{z}|\le\kappa}\left|
 \widehat{u}_{\mathbf{z}}^{n,i}(\cdot,x_i)
\right|\right](t)
&\le
F_n\left[
\sup_{|\mathbf{z}|\le\kappa}\left|
 \widehat{u}_{\mathbf{z}}^{n,i}(\cdot,x_i)
 - \rho(u(\cdot,x_i))
\right|\right](t) + F_n\Big[|\rho(u(\cdot,x_i))|\Big](t)\\
&\le
C F_n[M_n^{**}](t)+C\left|\rho(u(t,x_i))\right|+C \sum_{\ell=1}^3 F_n^\ell \left[\left|\rho(u(\cdot,x_i))\right|\right](t)\quad\text{a.s.}
\end{align*}
Notice that $\theta_{\mathbf{z}}^{n,i}(t,x_i)$ is the $I_0$ term in Step 2, hence by \eqref{e4:uHatD}
\[
F_n\left[\sup_{|\mathbf{z}|\le\kappa}\left|\theta_{\mathbf{z}}^{n,i}(\cdot)\right|\right](t)
\le
C M_n^{**}(t)+C|\rho(u(t,x_i))|+C\sum_{\ell=1}^2 F_n^\ell \big[|\rho(u(\cdot,x_i))|\big](t)\quad\text{a.s.}
\]
For $i\ne j$, from \eqref{E5:SupU} and Proposition \ref{P:Psi}, we see that
\begin{align}\label{E_:ThetaIJK}
\sup_{t\in[0,T]}\Norm{\sup_{|\mathbf{z}|\le\kappa}\sum_{j\ne i}\left|z_j \theta^{n,i,k}_{\mathbf{z}}(t,x_i) \right|}_p \le C 2^{-n\beta}.
\end{align}
The $I$ term coincides with the $I_1$ term in Step 2. 
Combining the above four terms proves \eqref{e1:thetaDD}.

{\bigskip\noindent\bf Step 4.~}
In this last step, we will prove Proposition \ref{P:thetaBdd} for $\widehat{u}^{n,i,k}_{\mathbf{z}}(T,x_i)$.
Write the six parts of $\widehat{u}^{n,i,k}_{\mathbf{z}}(t,x_i)$
in \eqref{E:hatUnik} as in \eqref{E:Unik6}, that is,
\begin{align}\label{e:Unik6}
\widehat{u}^{n,i,k}_{\mathbf{z}}(t,x_i) =
\theta^{n,i,k}_{\mathbf{z}}(t,x_i) +\theta^{n,k,i}_{\mathbf{z}}(t,x_i)+ \sum_{\ell=1}^4 U_\ell^n(t,x_i).
\end{align}
We first consider the term $U^n_4(t,x_i)$ which contributes to the recursion. Write $U^n_4(t,x_i)$ in three parts
\begin{align*}
&U^n_4(t,x_i) \\ 
= &\sum_{j\ne i}
 z_j\int_0^t\iint_{\R^{2d}} G(t-s,x_i-y)\rho'\left(\widehat{u}^n_{\mathbf{z}}(s,y)\right)
 \widehat{u}^{n,i,k}_{\mathbf{z}}(s,y)h_n^j(s,y')f(y-y')\ud s\ud y\ud y'\\
 &+
 z_i\int_0^t\iint_{\R^{2d}} G(t-s,x_i-y)\rho'\left(\widehat{u}^n_{\mathbf{z}}(s,y)\right)
 \left[
 \widehat{u}^{n,i,k}_{\mathbf{z}}(s,y)-\widehat{u}^{n,i,k}_{\mathbf{z}}(s,x_i)
 \right]h_n^i(s,y')f(y-y')\ud s\ud y\ud y'\\
 &+
 z_i\int_0^t\iint_{\R^{2d}} G(t-s,x_i-y)\rho'\left(\widehat{u}^n_{\mathbf{z}}(s,y)\right)
 \widehat{u}^{n,i,k}_{\mathbf{z}}(s,x_i)
 h_n^i(s,y')f(y-y')\ud s\ud y\ud y'\\
 =:&U^n_{4,1}(t,x_i)+U^n_{4,2}(t,x_i)+U^n_{4,3}(t,x_i).
\end{align*}
Notice that
\[
\sup_{|\mathbf{z}|\le\kappa}
\left|U_{4,3}^n(t,x_i)\right|
\le C\one_{\{t>T-2^{-n}\}}2^{n(1-\beta)}\int_{T-2^{-n}}^{t} (t-s)^{-\beta}
\sup_{|\mathbf{z}|\le\kappa}
|\widehat{u}^{n,i,k}_{\mathbf{z}}(s,x_i)|  \ud s,\quad\text{a.s.}
\]
Therefore, by Lemma \ref{L:Contraction},
\begin{align}\label{E:2ndD}
\sup_{|\mathbf{z}|\le\kappa}
|\widehat{u}^{n,i,k}_{\mathbf{z}}(t,x_i)|
\le C \sum F_n \left[\sup_{|\mathbf{z}|\le\kappa} I_{\mathbf{z}}^n(\cdot,x_i)\right](t)
\quad\text{a.s.,}
\end{align}
where the summation is over all terms on the right-hand side  of \eqref{e:Unik6} except $U_{4,3}^n(t,x_i)$
and $I_{\mathbf{z}}^n$ stands for such a generic term.

By the similar arguments as in the
 previous steps, using \eqref{E3:Space}, one can show that
\begin{align}\label{E_:U4-13}
\lim_{n\rightarrow\infty}
F_n\left[\sup_{|\mathbf{z}|\le\kappa}\left|U^n_{4,\ell}(\cdot,x_i)\right|\right](t) = 0,\quad\text{a.s. for $\ell=1,2$ and $t\in [0,T]$.}
\end{align}

By \eqref{E2:SupU}, \eqref{E3:SupU}, the boundedness of both $\rho'$ and $\rho''$, and \eqref{E:hG-J123},
we can obtain moment bounds for both $\sup_{|\mathbf{z}|\le\kappa}\left|U_1^n(t,x_i)\right|$
and
$\sup_{|\mathbf{z}|\le\kappa}\left|U_3^n(t,x_i)\right|$
and then argue using the Borel-Cantelli lemma as above to conclude that
\begin{align}\label{E_:U13}
\lim_{n\rightarrow\infty}
F_n\left[\sup_{|\mathbf{z}|\le\kappa}
\left|U_\ell^n(\cdot,x_i)\right|\right](t) =0, \quad \text{a.s. for $\ell=1,3$ and $t\in[0,T]$.}
\end{align}

As for the term $U_2^{n}$, because
\[
|\widehat{u}_{\mathbf{z}}^{n,i}(s,y)
\widehat{u}_{\mathbf{z}}^{n,k}(s,y)|\le\frac{1}{2}\left(
\widehat{u}_{\mathbf{z}}^{n,i}(s,y)^2+
\widehat{u}_{\mathbf{z}}^{n,k}(s,y)^2\right),
\]
we only need to consider the case when $i=k$.
By the boundedness of $\rho''$, we see that
\begin{align*}
U_2^{n}(t,x_i)\le& C
\int_0^t\iint_{\R^{2d}} G(t-s,x_i-y)
\widehat{u}_{\mathbf{z}}^{n,i}(s,y)^2
\InPrd{\mathbf{z},\mathbf{h}_n(s,y')}f(y-y')\ud s\ud y\ud y'\\
\le&C\int_0^t\iint_{\R^{2d}} G(t-s,x_i-y)
\left[\widehat{u}_{\mathbf{z}}^{n,i}(s,y)-\widehat{u}_{\mathbf{z}}^{n,i}(s,x_i)\right]^2
\InPrd{\mathbf{z},\mathbf{h}_n(s,y')}{f(y-y')}\ud s\ud y \ud y'\\
&+C\int_0^t\iint_{\R^{2d}} G(t-s,x_i-y)
\left[\widehat{u}_{\mathbf{z}}^{n,i}(s,x_i)-\rho(u(s,x_i))\right]^2
\InPrd{\mathbf{z},\mathbf{h}_n(s,y')}f(y-y')\ud s\ud y\ud y'\\
&+C\int_0^t\iint_{\R^{2d}} G(t-s,x_i-y)
\left[\rho(u(s,x_i))-\rho(u(t,x_i))\right]^2
\InPrd{\mathbf{z},\mathbf{h}_n(s,y')} f(y-y')\ud s\ud y\ud y'\\
&+ C \rho(u(t,x_i))^2 \\
=:& U_{2,1}^{n}(t,x_i)+U_{2,2}^{n}(t,x_i)
+U_{2,3}^{n}(t,x_i)+ C \rho(u(t,x_i))^2.
\end{align*}
By \eqref{E2:Space} and the H\"older continuity of $s\mapsto \rho(u(s,x_i))$
one can prove in the same way as before that
\begin{align}
 \label{E_:U2-13}
\lim_{n\rightarrow\infty}
F_n\left[
\sup_{|\mathbf{z}|\le\kappa}
U_{2,\ell}^{n}(\cdot,x_i)
\right](t)=0,\quad\text{a.s. for $\ell=1,3$ and $t\in[0,T]$.}
\end{align}
Notice that
\[
\sup_{|\mathbf{z}|\le\kappa}U_{2,2}^{n}(t,x_i)
\le
C F_n^1 \left[\sup_{|\mathbf{z}|\le\kappa}
\left[\widehat{u}_{\mathbf{z}}^{n,i}(\cdot,x_i)-\rho(u(\cdot,x_i))\right]^2
\right](t).
\]
By applying \eqref{E:F5} on \eqref{e5:uHatD} with $m'=2$, we see that with probability one, 
\begin{align*}
\sup_{|\mathbf{z}|\le\kappa}
\left[\widehat{u}_{\mathbf{z}}^{n,i}(s,x_i)-\rho(u(s,x_i))\right]^2
&\le
C [M_n^{**}(t)]^2+C\rho^2(u(t,x_i))+C \sum_{\ell=1}^2 F_n^\ell \left[\rho^2(u(\cdot,x_i))\right](t).
\end{align*}
Hence, by \eqref{E:F3}, for all $n\in\bbN$,
\begin{align*}
\sup_{|\mathbf{z}|\le\kappa}U_{2,2}^{n}(t,x_i)
\le &
C  F_n^1 [M_n^{**}](t)+C \sum_{\ell=1}^3 F_n^\ell \left[\rho^2(u(\cdot,x_i))\right](t)\quad\text{a.s.}
\end{align*}
Then another application of \eqref{E:F3} shows that
\begin{align*}
F_n\left[\sup_{|\mathbf{z}|\le\kappa}U_{2,2}^{n}(\cdot,x_i)\right](t)
\le &
C \sum_{\ell=1}^2   F_n^\ell [M_n^{**}](t)+C \sum_{\ell=1}^4 F_n^\ell \left[\rho^2(u(\cdot,x_i))\right](t)\quad\text{a.s.,}
\end{align*}
for all $t\in[0,T]$. Therefore,
\begin{align}\label{E_:U2-2}
\mathop{\lim\sup}_{n\rightarrow\infty}F_n\left[\sup_{|\mathbf{z}|\le\kappa}
\left|U_{2,2}^{n}(\cdot,x_i)\right|\right](t)
\le &
C\rho^2(u(t,x_i))\quad\text{a.s. for all $t\in[0,T]$.}
\end{align}

Finally, by combining \eqref{e1:thetaDD}, \eqref{E_:U4-13}, \eqref{E_:U13}, \eqref{E_:U2-13} and \eqref{E_:U2-2}, setting $t=T$, and sending $n$ to infinity, we can conclude that 
\[
\limsup_{n\rightarrow\infty}
\sup_{|\mathbf{z}|\le\kappa}\left|
 \widehat{u}_{\mathbf{z}}^{n,i,k}(T,x_i)
\right|\le C \left|\rho(u(T,x_i))\right|+C \rho^2(u(T,x_i)).
\]
This proves the case for
$\widehat{u}^{n,i,k}_{\mathbf{z}}(T,x)$.
With this, we have completed the whole  proof of Proposition \ref{P:thetaBdd}.
\end{proof}

\begin{small}

\bigskip
\bigskip
\hfill\begin{minipage}{0.5\textwidth}
{\bf Le CHEN}\\[0.2em]
Department of Mathematical Sciences\\
University of Nevada, Las Vegas\\
Box 454020, \\
4505 S. Maryland Pkwy.\\
Las Vegas, NV 89154-4020, USA.\\
E-mails: \: \url{le.chen@unlv.edu}\\
\phantom{E-mails: \: }\url{chenle02@gmail.com}
\end{minipage}
\\[1em]

\hfill\begin{minipage}{0.5\textwidth}
\noindent\textbf{Jingyu Huang}\\
\noindent School of Mathematics\\
\noindent University of Birmingham\\
\noindent Edgbaston, Birmingham, \\
B15 2TT, UK\\
\noindent\emph{Email:} \url{j.huang.4@bham.ac.uk}\\
\end{minipage}

\end{small}

\end{document}